\documentclass[a4paper,notitlepage,twoside,reqno,11pt]{amsart}

\usepackage{anysize}
\marginsize{3.4cm}{3.4cm}{2.8cm}{2.8cm}

\usepackage{bbm,pifont,latexsym,dcolumn,indentfirst}
\usepackage[hypertexnames=false,bookmarksopen=true,linktocpage=true,pdfstartview={XYZ null null 1.25}]{hyperref}
\usepackage{amsthm,amsmath,amssymb,amscd,amsfonts,mathrsfs}
\usepackage{color,graphicx,xcolor,graphics,subfigure,extarrows,caption2}
\usepackage{pdfpages,titletoc}

\usepackage{autonum}

\newtheorem{thm}{Theorem}[section]
\newtheorem{lem}[thm]{Lemma}
\newtheorem{cor}[thm]{Corollary}
\newtheorem{prop}[thm]{Proposition}

\newtheorem*{mainthm}{Main Theorem}
\newtheorem*{mainlem}{Main Lemma}
\newtheorem*{mainlempri}{Main Lemma$'$}

\theoremstyle{definition}
\newtheorem*{defi}{Definition}
\newtheorem*{rmk}{Remark}

\newcommand{\EC}{\widehat{\mathbb{C}}}
\newcommand{\A}{\mathbb{A}}
\newcommand{\C}{\mathbb{C}}
\newcommand{\D}{\mathbb{D}}
\newcommand{\BH}{\mathbb{H}}

\newcommand{\K}{\mathbb{K}}
\newcommand{\N}{\mathbb{N}}

\newcommand{\R}{\mathbb{R}}
\newcommand{\T}{\mathbb{T}}

\newcommand{\Z}{\mathbb{Z}}

\newcommand{\MI}{\mathcal{I}}

\newcommand{\MO}{\mathcal{O}}
\newcommand{\MP}{\mathcal{P}}

\newcommand{\MV}{\mathcal{V}}

\newcommand{\MCV}{\mathcal{CV}}

\newcommand{\ii}{\textup{i}}

\newcommand{\re}{\textup{Re}\,}
\newcommand{\im}{\textup{Im}\,}

\newcommand{\Exp}{\textup{Exp}}
\newcommand{\Crit}{\textup{Crit}}

\newcommand{\diam}{\textup{diam}}
\newcommand{\dist}{\textup{dist}}
\newcommand{\Mod}{\textup{mod}\,}

\makeatletter\@addtoreset{equation}{section}\makeatother

\begin{document}

\author[S. Wang]{Shuyi Wang}
\address{School of Mathematical Sciences, Qufu Normal University, Qufu 273165, P. R. China}
\email{sywang$\underline{~~}$math@126.com}

\author[F. Yang]{Fei Yang}
\address{School of Mathematics, Nanjing University, Nanjing 210093, P. R. China}
\email{yangfei@nju.edu.cn}

\author[G. Zhang]{Gaofei Zhang}
\address{School of Mathematics, Nanjing University, Nanjing 210093, P. R. China}
\email{zhanggf@nju.edu.cn}

\author[Y. Zhang]{Yanhua Zhang}
\address{School of Mathematical Sciences, Qufu Normal University, Qufu 273165, P. R. China}
\email{zhangyh0714@qfnu.edu.cn}

\title[Local connectivity of some Julia sets]{Local connectivity of Julia sets of some rational maps with Siegel disks}

\begin{abstract}
We prove that a long iteration of rational maps is expanding near boundaries of bounded type Siegel disks. This leads us to extend Petersen's local connectivity result on the Julia sets of quadratic Siegel polynomials to a general case. A new key feature in the proof is that the puzzles are not used.
\end{abstract}



\date{\today}



\maketitle

{
\tableofcontents
}

\section{Introduction}

The local connectivity of the Julia sets of rational maps is one of the central themes in complex dynamics.
By constructing expanding metrics, the connected Julia sets of hyperbolic, subhyperbolic and geometrically finite rational maps have been proved to be locally connected (see \cite{DH84}, \cite{TY96}). For the maps which do not have expanding metric near the Julia sets, constructing puzzles is the main idea in the study of the local connectivity of the Julia sets.
In the 1980s, Yoccoz proved that if the non-renormalizable quadratic polynomials have no irrationally indifferent periodic points, then their Julia sets are locally connected provided they are connected \cite{Hub93}. The so-called Yoccoz puzzle plays a crucial role in the proof. To study the topology of the Julia sets of cubic polynomials, the Branner-Hubbard puzzle was introduced in \cite{BH92}. By developing the puzzle techniques further, various Julia sets without Siegel disks have been proved to be locally connected. See \cite{HJ93}, \cite{Lyu97}, \cite{LS98}, \cite{KSS07b}, \cite{KL09a}, \cite{KS09}, \cite{Roe10}, \cite{CST17} for polynomials, and \cite{Roe08}, \cite{QWY12}, \cite{DS22}, \cite{WYZ23} for some special rational maps (see also \cite{CDKS22}).

Based on the Douady-Ghys surgery on cubic Blaschke products, Petersen proved that the Julia sets of quadratic polynomials with a fixed bounded type Siegel disk are locally connected \cite{Pet96} (see also \cite{Pet98} for the quadratic polynomials with periodic bounded type Siegel disks). Later Yampolsky derived this classical result from a priori bounds of unicritical analytic circle maps \cite{Yam99}. One of the main tools in their proofs is a puzzle structure, which is refereed to as Petersen's puzzle.
Petersen's puzzle construction can be adapted to polynomials of higher degrees by which J. Yang proved that the Julia sets of any polynomials (assumed to be connected) are locally connected at the boundary points of bounded type Siegel disks \cite{Yan23J}.
However, for rational maps, except in rare cases, for instances, the mating examples of Yampolsky and Zakeri \cite{YZ01} and some other examples in \cite{Yan15}, there is in general no puzzle structure. Hence the study of the topology of the Julia sets of rational maps is always a great challenge, especially when there is no expanding metric near the Julia sets.

Near boundaries of Siegel disks, the rational maps are far from being expanding.
Since the puzzle construction may not exist for general rational maps with Siegel disks, the corresponding Julia sets may have exotic topology in this situation. For instance, when a rational map has both Siegel disks and attracting basins, a prior, the boundary of an immediate attracting basin may spiral around a Siegel disk in a very complicated manner and the Julia set may be non-locally connected.

\subsection{Main results}

The main purpose of this paper is to extend Petersen's local connectivity result to a class of rational maps with bounded type Siegel disks. In particular, it implies that the exotic topology of the whole Julia sets mentioned above cannot occur when the critical orbits are well controlled, and moreover, it provides an alternative proof of Petersen's result without using puzzles.

\begin{mainthm}
Suppose $f$ is a rational map with Siegel disks such that the Julia set $J(f)$ is connected, and moreover, the forward orbit of every critical point of $f$ satisfies one of the following:
\begin{enumerate}
\item It is finite; or
\item It lies in an attracting basin; or
\item It intersects the closure of a bounded type Siegel disk.
\end{enumerate}
Then $J(f)$ is locally connected.
\end{mainthm}

As an immediate application of the Main Theorem, the Julia sets of all cubic Siegel polynomials in the Zakeri curve and in capture domains are locally connected, where the Zakeri curve is the collection of all cubic polynomials with both critical points on the boundary of a fixed Siegel disk of a given bounded type rotation number \cite{Zak99}.
As another application of the Main Theorem, we conclude that if the cubic Newton map or the McMullen map $z\mapsto z^n+\lambda/z^n$ with $n\geqslant 3$ has a bounded type Siegel disk, then the Julia sets are locally connected. For in this case their Julia sets are connected and their critical points lie either in attracting basins or on the preimages of the boundaries of bounded type Siegel disks. This complements the results of \cite[Theorem 4]{Roe08} and \cite[Theorem 1.3]{QWY12}.

\medskip
For a rational map $f$ with degree at least two, let $\Crit(f)$ be the set of all critical points of $f$. The \textit{postcritical set} of $f$ is
$$
\MP(f):=\overline{\bigcup_{n\geqslant 1 } f^n\big(\Crit(f)\big)}.
$$
Let $\EC$, $\dist_{\EC}(\cdot,\cdot)$ and $\diam_{\EC}(\cdot)$ denote respectively the Riemann sphere, the distance and diameter with respect to the spherical metric.
The following lemma is the key ingredient in the proof of  the Main Theorem.

\begin{mainlem}
Let $f$ be a rational map with degree at least two and having a fixed bounded type Siegel disk $\Delta$.  Suppose $\dist_{\EC}(\MP(f)\setminus\partial{\Delta},\partial{\Delta})>0$. Then for any $\varepsilon>0$ and any Jordan disk $V_0\subset \EC\setminus\overline{\Delta}$ satisfying $\overline{V}_0 \cap \MP(f)\neq\emptyset$ and $\overline{V}_0 \cap \MP(f) \subset \partial\Delta$, there exists $N \geqslant 1$, such that $\diam_{\EC}(V_n)<\varepsilon$ for all $n\geqslant N$, where $V_n$ is any connected component of $f^{-n}(V_0)$.
\end{mainlem}

The condition $\dist_{\EC}(\MP(f)\setminus\partial{\Delta},\partial{\Delta})>0$ in the Main Lemma is regarded as automatically satisfied if $\MP(f)\setminus\partial{\Delta}=\emptyset$.
The Main Lemma asserts that a contraction property of the composition of inverse branches near boundaries of bounded type Siegel disks holds for all pullback sequences and moreover, the contraction is ``uniform''. That is, for given $V_0$, the contraction depends only on the length of the pullback sequence, but not on the choice of the inverse branch for each pullback.
Such a contraction property may also be obtained by using complex a prior bounds of multi-critical analytic circle maps \cite{ESY22} (see \cite{Yam99} for the case of unicritical and also \cite{Yam19}).

The uniform contraction in the Main Lemma is essential for mating Siegel polynomials as well as the proof of local connectivity of the Julia sets with Siegel disks. In the celebrated work of Yampolsky-Zakeri \cite{YZ01} on the mating of quadratic Siegel disks, such uniform contraction property is obtained by using complex a prior bounds together with certain puzzle structure. See \cite{BV06} and \cite{Yan15} for a similar situation.
The Main Lemma does not assume the existence of any puzzle structure which may not exist for general rational maps.
Besides the applications of the Main Lemma in rational dynamics, a similar result as the Main Lemma can be used to prove the local connectivity of the Julia sets of transcendental entire functions with bounded type Siegel disks, where the puzzles are definitely not available (This will appear in a forthcoming paper, see \cite{WYZZ21} for partial results).

\subsection{Sketch of the proofs}

Let us first sketch the general idea of the proof of the Main Lemma. Let $f$ be the rational map in the Main Lemma having a fixed bounded type Siegel disk $\Delta$.  Note that analytic Blaschke models may not exist for general rational maps.
In Section \ref{sec-basic-tools}, we construct a \textit{quasi-Blaschke model} $G$ (which is quasi-regular) for $f$ such that the Siegel disk $\Delta$ (resp. the boundary $\partial\Delta$) of $f$ is replaced by the unit disk $\D$ (resp. the unit circle $\T$) of $G$. Such a model was first used by Petersen to study the Herman-\'{S}wi\c{a}tek theorem for holomorphic self-homeomorphisms of quasi-circles \cite{Pet04}. The heart of the proof in this paper is to show that a long iteration of $G$ is expanding near the unit circle.

Let $\MP(G)$ be the corresponding postcritical set of the model map $G$ in $\EC\setminus\D$. Since $\partial\Delta\subset\MP(f)$, we have $\T\subset \MP(G)$ (see Section \ref{subsec:quasi-Blaschke}). For simplicity, here we only consider the Jordan disk $V_0\subset \EC\setminus \overline{\D}$ which satisfies that $\overline{V}_0 \cap \MP(G)$ is a non-empty subarc on $\T$ and the general case can be reduced to such case (see Section \ref{subsec:key-thm}). We use $\{V_n\}_{n\geqslant 0}$ to denote the pullbacks of $V_0$ in $\EC\setminus\overline{\D}$ under $G$. Note that $G^n: V_n\to V_0$ is conformal and each $V_n$ is a Jordan disk (see Lemma \ref{lem:take-pre}). For some $N_0$ large enough, $\overline{V}_{N_0} \cap \mathbb T$ is contained in a small open arc $I_0\subset\T$ (see Lemma \ref{lem:eventually-short}) and the Jordan disk $V_{N_0}$ is contained in the following finite union in $\EC\setminus\overline{\D}$:
\begin{equation}
 H_d(I_0)  \cup \bigcup_{1\leqslant i \leqslant M_0}  W_i,
\end{equation}
where each $W_i$ is a Jordan disk compactly contained in $\EC\setminus\overline{\D}$ satisfying $\sharp\big(\overline{W}_i\cap\MP(G)\big)\leqslant  1$ and $H_d(I_0)$ is a half hyperbolic $d$-neighborhood of $I_0$ with $d>0$ being a given large number (see \eqref{equ:H-d-I} and Figure \ref{Fig_half-nbd} for the definition of $H_d(I)$ and Section \ref{subsec:key-thm} for the way to choose $W_i$'s).
By applying the classical Shrinking Lemma (see Lemma \ref{lem:semi-hyperbolic}) to each $W_i$, the Main Lemma can be reduced to proving Lemma \ref{lem:back-1}: For any $\varepsilon>0$, if $V_0=H_d(I_0)$ with $I_0$ small enough, then
\begin{equation}\label{equ:V-n-infty}
\diam_{\EC}(V_n)<\varepsilon \text{\quad for all } n\geqslant 0.
\end{equation}

To this end, for any small subarc $I_0 \subset \mathbb T$, we first give a rule in Section \ref{sec-pull-back} by which one can dynamically define an infinite sequence of arcs $\{I_n\}_{n\geqslant 0}$ on $\T$ by choosing two constants $0<\delta<\eta<1/2$ suitably, which is called a \textit{$(\delta, \eta)$-admissible sequence}.
For each $n\geqslant 0$, one of the branches of $G^{-1}$, say $\Phi_n$, is associated to $I_n \to I_{n+1}$. We call $I_{n+1}$ the \textit{pullback} of $I_n$ but note that sometimes $G(I_{n+1})\neq I_n$, especially when $I_n$ contains a critical value.
The process $I_n \to I_{n+1}$ is called a \textit{critical pullback} if there is a critical value on $\T$ close to $I_n$ (see Section \ref{sec-pull-back} for the precise definition).
We show that finitely many critical pullbacks will decrease the \textit{dynamical length} (see \eqref{equ:dynam-length} for the definition) $\sigma(I_n)$ of $I_n$ by a definite amount. In particular, $\sigma(I_n)< 2 \sigma(I_m)$ for all $n>m\geqslant 0$ and $\sigma(I_n) \to 0$ as $n\to\infty$ (see Lemma \ref{lemma:shorten}).

Let $I_{n_j} \to I_{n_j+1}$, where $j  \geqslant 1$, be the sequence of all critical pullbacks.
For a small $s_0 > 0$, let $D_{n_j}^{s_0}$ be the half Euclidean $s_0$-neighborhood of $I_{n_j}$ in $\C\setminus\overline{\D}$ which is disjoint with $\MP(G)$ (see \eqref{equ:D-n-j} and Figure \ref{Fig_neigh-XYZ}). Denote by $\Omega$ the unique component of $\EC \setminus(\MP(G) \cup \overline{\D})$ with $\T \subset \partial \Omega$.
The main result in Section \ref{sec-pull-back} is Proposition \ref{prop:three special regions}: There exists a number $T_0\geqslant 1$ depending only on $G$, such that for any $z\in D_{n_j}^{s_0}$, the composition $\Phi_{n_{j+T_0}} \circ \cdots \circ \Phi_{n_j+1}\circ \Phi_{n_j}$ of the inverses of $G$ either contracts the hyperbolic metric $\rho_\Omega(z)|dz|$ uniformly at $z$, or maps $z$ into a half hyperbolic $d$-neighborhood $H_d(J)$ of an arc $J\subset\T$ containing the base interval $I_{n_{j+T_0}+1}$, where the dynamical lengths of $J$ and $I_{n_{j+T_0}+1}$ are comparable.

For the proof of Proposition \ref{prop:three special regions}, we analyze the combinatoric information of the critical orbit of $G$ on $\T$ and apply the contraction property of $G^{-1}$ (see Lemma \ref{lemma:Pe}) several times at most of the place in $D_{n_j}^{s_0}$, except in a triangle-shaped region where the points will be mapped very close to the base interval $I_{n_{j+T_0}+1}$ because of the finitely many critical pullbacks. The contraction principle in Lemma \ref{lemma:Pe} was first observed by Petersen in \cite{Pet96}.

\medskip
Based on Proposition \ref{prop:three special regions} and the property $\sigma(I_n) \to 0$ as $n\to\infty$, for each $(\delta, \eta)$-admissible sequence $\{I_n\}_{n\geqslant 0}$, we construct an \textit{improved} $(\delta, \eta)$-admissible sequence $\{(J_n,F_n,d_n)\}_{n\geqslant 0}$ in Section \ref{subsec:improv-d-e}, where
\begin{itemize}
\item $J_n$ is an arc of $\T$ containing $I_n$ so that $\sigma(J_n) < 2\sigma(J_0)=2\sigma(I_0)$ and $\sigma(J_n) \to 0$  as $n \to \infty$; and
\item $F_n$ is the union of at most countably many Jordan disks intersecting $H_{d_n}(J_n)$ with $d_n\in[d,d+1]$ and having hyperbolic diameters with respect to $\rho_\Omega(z)|dz|$ less than a constant $K=K(\delta,\eta,d)>0$.
\end{itemize}
Such a sequence allows us to capture further the contraction property of a long iteration of the inverse of $G$ near $\T$ as the following (see Section \ref{subsec:improv-d-e}):
\begin{equation}
\Phi_{m-1}\circ\cdots\circ\Phi_{n+1}\circ\Phi_n\big(H_{d_n}(J_n)\cup F_n\big)\subset H_{d_m}(J_m) \cup F_m\subset D_m^{s_0},
\end{equation}
where $m>n\geqslant 0$. We have $\diam_{\EC}(H_{d_n}(J_n)\cup F_n)\to 0$ as $n \to \infty$ since $\sigma(J_n) \to 0$ as $n \to \infty$, $d_n\in[d,d+1]$ and the hyperbolic diameter of every component of $F_n$  is bounded above by $K$.

Suppose $\{V_n\}_{n\geqslant 0}$ is the sequence obtained by pulling back $V_0=H_d(I_0)$ under $G$. For given $\varepsilon>0$, by assuming $I_0$ small, if there is a sequence of improved  $(\delta, \eta)$-admissible sequences $\big\{\{(J_n^{(m)}, F_n^{(m)}, d_n^{(m)})\}_{n\geqslant 0}:m\geqslant 0\big\}$ such that for any $m\geqslant 0$,
\begin{equation}
V_n \subset H_{d_n^{(m)}}(J_n^{(m)}) \cup F_n^{(m)} \text{\quad for all } 0\leqslant  n\leqslant  m,
\end{equation}
then $\diam_{\EC}(V_n)<\varepsilon$ for all $n\geqslant 0$.
Otherwise, there exist a maximal $N<+\infty$ and an improved $(\delta,\eta)$-admissible sequence $\{(J_n,F_n,d_n)\}_{n\geqslant 0}$ such that $V_n \subset H_{d_n}(J_n)\cup F_n$ for all $0 \leqslant n \leqslant N$ but this does not hold for $N+1$.  In this case we say that the pullback $V_N \to V_{N+1}$ is a \textit{jump off}.
There are two types of jump offs.
For the first type jump off, $V_{N+1}$ is bounded away from $\mathbb T$. In this case, the size of all the subsequent pullbacks will be small by a routine argument.
For the second type jump off, $V_{N+1}$ is still very close to $\mathbb T$ so that $V_{N+1} \subset H_d(J)$ with $J$ being arbitrarily small provided that $I_0$ is small enough, and most importantly, the hyperbolic diameter of $V_{N+1}$ is bounded above by some constant, and the pullback $V_N \to V_{N+1}$ decreases the hyperbolic metric by a definite amount.
So if $I_0$ is small enough and $V_N \to V_{N+1}$ is a second type jump off, we can consider the jump off of the sequence $\{V_k\}_{k\geqslant N+1}$ with respect to $J$.  This process can thus be repeated. So either we will have a first type jump off at some time, or the number of the second type jump offs will be eventually large enough   so that   the hyperbolic diameter of $V_n$  is smaller than a given number. This implies Lemma \ref{lem:back-1} and the Main Lemma follows.

\medskip

The proof of the Main Theorem is presented in  Section \ref{sec-rational-case}.
There are two basic tools in the proof --  the Main Lemma and a criterion of Whyburn. The criterion says that a compact subset $X$  of $\widehat{\mathbb C}$ is locally connected if and only if the boundary of every component of $\widehat{\mathbb C} \setminus X$ is locally connected, and moreover, for any $\varepsilon > 0$, the number of the components of $\widehat{\mathbb C} \setminus X$ whose size is greater than $\varepsilon$ is finite.

The second condition of Whyburn's criterion follows directly from the Main Lemma.  Note that the boundaries of  bounded type Siegel disks are quasi-circles  \cite{Zha11}.
To verify the first condition, it suffices to show that the boundaries of immediate attracting basins  are locally connected if  the map has attracting cycles. The argument here is a bit subtle.
As we mentioned at the beginning of the introduction, the boundary of an attracting basin may turn around some Siegel disk in a very complicated manner so that the uniform contraction property may not hold -- for a large number of pullbacks are needed to unwind the object first before shrinking it.
To overcome this we will show that the \textit{homotopy complexity} of the internal rays in the immediate attracting basins actually have a uniform upper bound (see Lemma \ref{lem:essential-bd}). This allows us to apply the Main Lemma to deduce that the equipotential curves in the immediate attracting basins converge uniformly. Thus the boundaries of immediate attracting basins are locally connected, the first condition of Whyburn's criterion follows and the Main Theorem holds.

\medskip
\noindent\textbf{Notations.} We will use the following notations throughout this paper.
\begin{itemize}
\item Let $\R$, $\C$ and $\EC$, resp. be the real axis, complex plane and Riemann sphere.
\item For $a\in\C$ and $r>0$, denote $\D_r(a):=\{z\in\C:|z-a|<r\}$, $\D_r:=\D_r(0)$, $\D:=\D_1$, $\T_r:=\{z\in\C:|z|=r\}$ and $\T:=\T_1$.
\item For $r>1$, denote the annulus $\A_r:=\{z\in\C:1/r<|z|<r\}$.
\item If $X=\EC$ (resp. $\C$, or a hyperbolic domain), let $\dist_{X}(\cdot,\cdot)$ and $\diam_X(\cdot)$ be the distance and diameter with respect to the spherical (resp. Euclidean, or hyperbolic) metric. In particular, $\rho_X(z)|dz|$ is the hyperbolic metric in a hyperbolic domain $X$.
\item Two positive numbers $a,b$ are said to be $C$-comparable if $b/C\leqslant a \leqslant b C$. For two family of positive numbers $\{a_\lambda\}$ and $\{b_\lambda\}$ depending on the parameter $\lambda$, we write $a_\lambda\preccurlyeq b_\lambda$ if there exists a constant $C>1$ such that $a_\lambda\leqslant C b_\lambda$ for all $\lambda$. We write $a_\lambda\asymp b_\lambda$ if $b_\lambda\preccurlyeq a_\lambda\preccurlyeq b_\lambda$.
\end{itemize}

\medskip
\noindent\textbf{Acknowledgements.}
The authors are very grateful to the referee for very careful reading, insightful and detailed comments, suggestions and corrections.
They would also like to thank D. Dudko, Z. Li, M. Lyubich, W. Qiu, W. Shen, W. Su and J. Yang for helpful comments/suggestions on the paper.
This work was supported by NSFC (Grant Nos.\,12222107, 12171276).

\section{Quasi-Blaschke models}\label{sec-basic-tools}

In this section we define quasi-Blaschke models for the maps in the Main Lemma and introduce some tools in hyperbolic geometry.

\subsection{Quasi-Blaschke models and dynamical lengths}\label{subsec:quasi-Blaschke}

In the following, we fix the rational map $f$ which has a fixed bounded type Siegel disk $\Delta$ in the Main Lemma. According to \cite{Zha11}, $\partial\Delta$ is a quasi-circle containing at least one critical point of $f$. Without loss of generality, we assume that $\Delta$ is bounded in the complex plane $\C$.
Let
\begin{equation}\label{equ:qc-1}
\phi:\EC\setminus\overline{\Delta}\to\EC\setminus\overline{\D}
\end{equation}
be a conformal isomorphism fixing $\infty$.

For $z\in\EC$ (resp. $Z\subset\EC$), let $z^*=1/\overline{z}$ (resp. $Z^*=\{z^*: z\in Z\}$) be the symmetric image of $z$ (resp. $Z$) about the unit circle $\T$.
Let $\Crit(f)$ be the set of all critical points and $\MP(f)$ the postcritical set of $f$.
We extend $\phi$ to a quasiconformal homeomorphism $\phi:\EC\to\EC$ by considering the following two cases:
\begin{itemize}
\item (Non-capture case) If $\MP(f)\cap\Delta=\emptyset$, let $\phi:\EC\to\EC$ be any quasiconformal extension of \eqref{equ:qc-1};
\item (Capture case) If $\MP(f)\cap\Delta\neq\emptyset$, then the following set is nonempty and finite:
\begin{equation}
\MV:=\left\{f^n(c)
\left|
\begin{array}{l}
c\in\Crit(f)\text{ and } n\geqslant 1 \text{ such that} \\
f^{n-1}(c)\not\in\Delta \text{ and } f^n(c)\in\Delta
\end{array}
\right.
\right\}.
\end{equation}
Suppose $\MV=\{b_1,\cdots, b_m\}$, where $m\geqslant 1$. Let $c_1\in\Crit(f)$ and $n_1\geqslant 1$ such that $f^{n_1}(c_1)=b_1\in\Delta$ but $f^{n_1-1}(c_1)\not\in\Delta$.
We take $m$ different points $b_1',\cdots,b_m'$ in $\EC\setminus\overline{\Delta}$ and $m$ positive integers $k_1, \cdots, k_m$ such that
\begin{equation}\label{equ:f-k-i}
f^{k_i}(b_i')=c_1 \text{\quad for all } 1\leqslant  i\leqslant  m.
\end{equation}
Let $\phi:\EC\to\EC$ be a quasiconformal extension of \eqref{equ:qc-1} such that
\begin{equation}\label{equ:phi-b-i}
\phi(b_i)=\big(\phi(b_i')\big)^*\in\D \text{\quad for all } 1\leqslant  i\leqslant  m.
\end{equation}
\end{itemize}

In both cases, we define a \textit{quasi-Blaschke model} corresponding to $f$:
\begin{equation}\label{equ:G}
G(z):=\left\{
\begin{array}{ll}
\phi \circ f \circ \phi^{-1}(z) & \text{\quad if } z\in\EC\setminus\D,\\
\big(\phi \circ f \circ \phi^{-1}(z^*)\big)^* &\text{\quad if } z\in\D.
\end{array}
\right.
\end{equation}
By the construction and the assumption $\dist_{\EC}(\MP(f)\setminus\partial{\Delta},\partial{\Delta})>0$ in the Main Lemma, $G$ has the following properties:
\begin{itemize}
\item[(i)] $G$ commutes with $z\mapsto z^*$, i.e., $G(z^*)=\big(G(z)\big)^*$;
\item[(ii)] $G:\EC\setminus\D\to\EC$ is conjugate to $f:\EC\setminus\Delta\to\EC$ by the quasiconformal mapping $\phi^{-1}:\EC\to\EC$;
\item[(iii)] $G:\EC\to\EC$ is a quasi-regular map and analytic in $\EC\setminus (Q\cup Q^* )$, where $Q=\phi\big(\overline{f^{-1}(\Delta)\setminus\Delta}\big)$; and
\item[(iv)] $\dist_{\EC}(\widetilde{\MP}(G)\setminus\T,\T)>0$, where
\begin{equation}\label{equ:P-G-tilde}
\widetilde{\MP}(G):=\overline{\bigcup_{n\geqslant 1,\,c\in\Crit(f)} G^n\big(\phi(c)\cup \phi(c)^*\big)}.
\end{equation}
\end{itemize}
Indeed, the property (iv) holds since by \eqref{equ:f-k-i} and \eqref{equ:phi-b-i}, we have the forward orbit
\begin{equation}
(\phi(b_i))^* \overset{G^{k_i}}{\longmapsto} \phi(c_1) \overset{G^{n_1}}{\longmapsto} \phi(b_1) \overset{G^{k_1}}{\longmapsto} (\phi(c_1))^*
\overset{G^{n_1}}{\longmapsto} (\phi(b_1))^* \overset{G^{k_1}}{\longmapsto} \phi(c_1),
\end{equation}
i.e., the forward orbits of $\phi(b_i)$ and $(\phi(b_i))^*$ fall into periodic cycles, where $1\leqslant  i\leqslant  m$.
The property (iv) guarantees that the domain in \eqref{equ:Omega-I} and the half hyperbolic neighborhoods in \eqref{equ:H-d-I} are well defined.
Denote
\begin{equation}\label{equ:P-G}
\MP(G):=\phi\big(\MP(f)\setminus\Delta\big).
\end{equation}
Then $\T\subset\MP(G)\subset\widetilde{\MP}(G)$ since $\partial\Delta\subset\MP(f)$.
Note that $\widetilde{\MP}(G)$ is the postcritical set of $G$. But in the following we also call $\MP(G)$ the \textit{postcritical set} of $G$ for convenience, since we are mainly interested in the dynamics of $G$ in $\EC\setminus\D$.

\medskip
Let $V_0$ be a Jordan disk in $\EC\setminus\overline{\D}$. We call $\{V_n\}_{n\geqslant 0}$ a \textit{pullback sequence} of $V_0$ if $V_{n+1}$ is a connected component of $G^{-1}(V_n)$ in $\EC\setminus\overline{\D}$  for all $n\geqslant 0$.
The following result implies the Main Lemma immediately.

\begin{mainlempri}
For any $\varepsilon>0$ and any Jordan disk $V_0\subset \EC\setminus\overline{\D}$ with $\overline{V}_0 \cap \MP(G) \neq \emptyset$ and $\overline{V}_0 \cap \MP(G) \subset \T$, there exists $N \geqslant 1$, such that for any pullback sequence $\{V_n\}_{n\geqslant 0}$ of $V_0$, we have $\diam_{\EC}(V_n)<\varepsilon$ for all $n\geqslant N$.
\end{mainlempri}

This lemma implies that a long iteration of the quasi-Blaschke model $G$ is expanding near the unit circle.
In the following we fix the model map $G$ and omit the dependence of all the constants on $G$.
The heart of this paper is the proof of the Main Lemma$'$, which occupies Section \ref{sec-basic-tools} to Section \ref{sec:key-thm}.

\medskip
The following result is useful when we consider the pullbacks of Jordan disks. For a proof, see \cite[Proposition 2.8]{Pil96}.

\begin{lem}\label{lem:take-pre}
If a Jordan disk $V\subset\EC\setminus\overline{\D}$ contains no critical values, or its closure contains at most one critical value,
then every component $U$ of $G^{-1}(V)$ is a Jordan disk. Moreover, $G:\overline{U}\to \overline{V}$ is a homeomorphism in the first case.
\end{lem}

Let $\R$ be the real axis. An orientation-preserving homeomorphism $\widetilde{g}:\mathbb{R}\rightarrow\mathbb{R}$ is called a \textit{quasi-symmetric map} if there exists $k>1$ such that
\begin{equation}\label{defi-quasi-symmetric}
k^{-1}\leqslant \frac{\widetilde{g}(x+t)-\widetilde{g}(x)}{\widetilde{g}(x)-\widetilde{g}(x-t)}\leqslant  k
\end{equation}
for all $x\in\mathbb{R}$ and all $t>0$. Let $g:\T\to\T$ be an orientation-preserving homeomorphism of the unit circle $\T$ and $\widetilde{g}:\mathbb{R}\rightarrow\mathbb{R}$ a lift of $g$ under the covering map $x\mapsto e^{2\pi \ii x}$. The map $g:\T\to\T$ is called \textit{quasi-symmetric} if \eqref{defi-quasi-symmetric} holds for $\widetilde{g}$.

\medskip

Let $\Delta$ be the fixed bounded type Siegel disk of $f$ with rotation number $\alpha\in(0,1)$. Since $\partial\Delta$ is a quasi-circle, there exists a quasiconformal mapping $\psi:\EC\to\EC$ with $\psi(\partial\Delta)=\T$ such that
\begin{equation}
\psi\circ f\circ \psi^{-1}(\zeta)=R_\alpha(\zeta):=e^{2\pi\ii\alpha}\zeta \text{\quad for all } \zeta\in\T.
\end{equation}
The restriction $G=\phi \circ f \circ \phi^{-1}:\T\to\T$ of the quasi-regular map defined in \eqref{equ:G} has rotation number $\alpha$. Therefore,
\begin{equation}
h:=\psi\circ\phi^{-1}:\EC\to\EC
\end{equation}
is a quasiconformal mapping satisfying
\begin{equation}
h\circ G\circ h^{-1}(\zeta)=R_\alpha(\zeta) \text{\quad for all ~~}\zeta\in\T.
\end{equation}
In particular, $h:\T\to\T$ is a quasi-symmetric map conjugating $G$ to the rigid rotation $R_\alpha$. Note that $h$ is unique up to a post-composite rigid rotation of $\T$.

\begin{defi}[Dynamical length]
For any arc $I \subset \T$, we use $|I|$ to denote the Euclidean length of $I$. Define
\begin{equation}\label{equ:dynam-length}
\sigma(I): = |h(I)|.
\end{equation}
We call $\sigma(I)$ the \textit{dynamical length} of $I$. Clearly, $\sigma(\cdot)$ is $G$-invariant in the sense that
$\sigma(G(I)) = \sigma(I)$.
\end{defi}

Since both $h:\T\to\T$ and $h^{-1}:\T\to\T$ are quasi-symmetric maps, the following result follows from \eqref{defi-quasi-symmetric} immediately.

\begin{lem}\label{lem:comparable}
We have $\sigma(I)\asymp \sigma(J)$ if and only if $|I|\asymp|J|$, where $I$ and $J$ are disjoint and adjacent subarcs of $\T$.
\end{lem}

In this paper, the anticlockwise direction of the unit circle is regarded as the positive direction. This induces an orientation for any subarc $I\subset\T$. If $I$ is a subarc of $\T$ with $\overline{I}=[a,b]\neq\T$, we say that $a$ is on the \textit{left} of $b$ (or, $b$ is on the \textit{right} of $a$). If $I=\T\setminus\{a\}$, we denote $I=(a^-,a^+)$ for convenience.
For any two different $a,b\in\T$, by definition we have $(a,b)\cup (b,a)\cup\{a,b\}=\T$.

\subsection{First return properties} \label{subsec:return}

We first give a brief account of the combinatorics of the closest returns and the associated dynamical partitions of the unit circle under the irrational rotation $R_\alpha(\zeta)=e^{2\pi\ii\alpha}\zeta$. One may refer to \cite[Appendix C]{Mil06} for the details.
For $n\geqslant 1$, let $p_n/q_n:=[0;a_1,\cdots,a_n]$ be the $n$-th approximation of the continued fraction expansion of $\alpha\in(0,1)$, where $p_n$ and $q_n$ are coprime positive integers. Denote $p_0=0$ and $q_0=1$. Note that $p_1=1$ and $q_1=a_1$. It is well known that for all $n\geqslant 2$,
\begin{equation}\label{equ:p-n-q-n}
p_n=a_n p_{n-1}+p_{n-2} \text{\quad and\quad} q_{n}=a_{n}q_{n-1}+q_{n-2}.
\end{equation}

Denote $\lambda=e^{2\pi\ii\alpha}$. We say that the sequence $\lambda^1$, $\lambda^2$, $\lambda^3$, $\cdots$ has a \textit{close return} to $\lambda^0=1$ at the time $q\geqslant 1$ if $\lambda^q$ is closer to 1 than any of its predecessors:
\begin{equation}
\big|\lambda^q-1\big|<\big|\lambda^k-1\big| \text { for } k=1,2,3, \cdots, q-1.
\end{equation}
Actually the collection of all close return times are exactly the $q_n$'s: $1=q_0<q_1<q_2<\cdots$.
Let $d_n$ be the Euclidean length of the shortest arc in $\T$ connecting $1$ with $\lambda^{q_n}$. Then
$2\pi\min\{\alpha,1-\alpha\}=d_0>d_1>d_2>\cdots>0$ and for all $n\geqslant 2$,
\begin{equation}\label{equ:d-n}
d_n=d_{n-2}-a_n d_{n-1} \text{\quad and \quad} a_n<\frac{d_{n-2}}{d_{n-1}}<a_n+1.
\end{equation}
If $0<\alpha<1/2$, then $\lambda^{q_1}$ lies in the open arc $(\lambda,1)\subset\T$ and $\lambda^{q_2}\in(1,\lambda)\subset\T$. Moreover, we have $\lambda^{q_n}\in(\lambda^{q_{n-2}},1)$ if $n$ is odd and $\lambda^{q_n}\in(1,\lambda^{q_{n-2}})$ if $n$ is even, where $n\geqslant 3$.
In the closed arc $[\lambda^{q_{n-2}},\lambda^{q_{n-1}}]\subset\T$ containing $1$, the points with the form $\{\lambda^k: 0\leqslant k\leqslant q_n\}$ are (listed by the positive direction on $\T$):
\begin{equation}
\lambda^{q_{n-2}},~\lambda^{q_{n-2}+q_{n-1}},~\lambda^{q_{n-2}+2q_{n-1}},~\cdots,~\lambda^{q_n}=\lambda^{q_{n-2}+a_n q_{n-1}},~1,~\lambda^{q_{n-1}}.
\end{equation}
If $1/2<\alpha<1$, then $\lambda^{q_1}\in(1,\lambda)$, $\lambda^{q_2}\in(\lambda,1)$ and the above statements are still true by permuting the parity of $n$.

\medskip
In the following, every subarc in $\T$ is assumed to be open and nonempty unless otherwise specified. Let $I\subset\T$ be an arc. For any $x \in I$, we denote
\begin{equation}
N(x,I):=\min\{n\geqslant 1: ~ G^n(x)\in I\}.
\end{equation}
Define
\begin{equation}
N(I) := \max_{x\in I} N(x, I)  \text{\quad and\quad} n(I) := \min_{x\in I} N(x, I).
\end{equation}

\begin{lem}\label{lem:r-1}
We have $N(I) \preccurlyeq n(I)$ for any subarc $I \subset \T$.
\end{lem}

\begin{proof}
We assume that $|I|$ is small since otherwise, the conclusion is obvious.
For any $x \in I = (a, b)$, let $n = N(x, I)\geqslant 1$ be the first return time of $x$.
Then there exists an integer $k  = k(x) \geqslant 2$ such that
\begin{itemize}
\item $(a, x) \subset \big(G^{q_{k-1}}(x), x\big)$ and $(x, b) \subset \big(x, G^{q_{k-2}}(x)\big)$;
\end{itemize}
and moreover,
\begin{itemize}
\item $n = q_{k-1} + j q_k$ with $1 \leqslant j \leqslant a_{k+1}$ if $G^n(x) \in (a,x)$; or
\item $n = q_{k-2} + j q_{k-1}$ with $1 \leqslant j \leqslant a_k$ if $G^n(x) \in (x,b)$.
\end{itemize}
Hence we have either $G^{q_{k+1}}(x)\in (a,x)\subset I$ or $G^{q_k}(x)\in (x,b)\subset I$.

Let $\sigma(I)$ be the dynamical length of $I\subset\T$ defined in \eqref{equ:dynam-length}.
Since $\alpha$ is of bounded type, by \eqref{equ:d-n} we have $\sigma(I)\asymp |(1 , R_{\alpha}^{q_k}(1))|$.
Hence
\begin{equation}
|(1 , R_{\alpha}^{q_{k(x)}}(1))|\asymp |(1 , R_{\alpha}^{q_{k(y)}}(1))| \text{\quad for any } x,y\in I.
\end{equation}
Therefore, there exists an integer $\ell_0=\ell_0(\alpha) \geqslant 1$ such that $|k(x) - k(y)| \leqslant \ell_0$ for any $x ,y \in I$.
This, together with \eqref{equ:p-n-q-n} and the fact that $\alpha$ is of bounded type, implies the lemma.
\end{proof}

\begin{lem}\label{lem:r-2}
We have $\sigma(I)\asymp \sigma(J)$ if and only if $N(I)\asymp N(J)$, where $I,J$ are subarcs of $\T$.
\end{lem}

\begin{proof}
We only prove the necessity since the sufficiency is completely similar.
From the proof of Lemma \ref{lem:r-1}, it follows that there exist two integers $k_1, k_2 \geqslant 2$
such that
\begin{equation}
\begin{split}
\sigma(I)\asymp |(1, R_\alpha^{q_{k_1}}(1))|, \quad & N(I)\asymp q_{k_1} \text{\quad and}\\
\sigma(J)\asymp |(1, R_\alpha^{q_{k_2}}(1))|, \quad & N(J)\asymp q_{k_2}.
\end{split}
\end{equation}
Thus if $\sigma(I)\asymp\sigma(J)$, then $|(1, R_\alpha^{q_{k_1}}(1))|\asymp|(1, R_\alpha^{q_{k_2}}(1))|$. By the fact that $\alpha$ is of bounded type, this is equivalent to that $q_{k_1}\asymp q_{k_2}$. Thus $N(I)\asymp N(J)$.
\end{proof}

\begin{lem}\label{lemma:fr}
There exists $\lambda_0>1$ depending only on $\alpha$ such that for any arc $I \subset \T$ and $x\in I$, if $G^{n_1}(x)$ and $G^{n_2}(x)$ are the first and the second returns of $x$ to $I$, i.e.,
\begin{equation}
n_1=N(x, I)  \text{\quad and\quad} n_2=n_1+N(G^{n_1}(x), I),
\end{equation}
then $\sigma(I)$ and $\sigma(J)$ are $\lambda_0$-comparable, where $J$ is the shortest arc between two points of $x$, $G^{n_1}(x)$ and $G^{n_2}(x)$ with respect to the dynamical length.
\end{lem}

\begin{proof}
By Lemmas \ref{lem:r-1} and \ref{lem:r-2},
the dynamical lengths of $(x, G^{n_1}(x))$ (or $(G^{n_1}(x),x)$) and $(G^{n_1}(x), G^{n_2}(x))$ (or $(G^{n_2}(x), G^{n_1}(x))$)
are comparable with $\sigma(I)$. Without loss of generality, we assume that $(x,G^{n_1}(x))\subset I$, i.e., $G^{n_1}(x)$ is on the right of $x$. There are the following two cases:

\medskip
\textbf{Case I}: Suppose $I':=(x,G^{n_2}(x))\subset I$. If $G^{n_1}(x) \in I'$, then
\begin{equation}
\sigma\big((x, G^{n_1}(x))\big)\asymp \sigma\big((G^{n_1}(x),G^{n_2}(x))\big)\asymp\sigma(I).
\end{equation}
Otherwise, $G^{n_2}(x)\in(x,G^{n_1}(x))$.  Then for any sufficiently small $\varepsilon>0$, we have
\begin{equation}
n_2=N(x,I_\varepsilon'), \text{\quad where } I_\varepsilon':=(e^{-\ii\varepsilon}x, e^{\ii\varepsilon}G^{n_2}(x)).
\end{equation}
Note that $n_1 = N(x, I)$ and $n_2- n_1 = N(G^{n_1}(x), I)$. By Lemma \ref{lem:r-1}, we have $n_1\asymp n_2-n_1$ and hence $n_1\asymp n_2$. By Lemmas \ref{lem:r-1} and \ref{lem:r-2} and letting $\varepsilon\to 0$, we conclude that $\sigma(I)\asymp\sigma(I')$. In the first paragraph, we know that $\sigma(I)\asymp\sigma(I'')$, where $I''=\big(G^{n_2}(x),G^{n_1}(x)\big)$.
Therefore, $\sigma(I)\asymp \sigma(I')\asymp\sigma(I'')$.

\medskip
\textbf{Case II}: Suppose $I':=(G^{n_2}(x),x)\subset I$. Then one can define $I_\varepsilon':=(e^{-\ii\varepsilon}G^{n_2}(x)$, $e^{\ii\varepsilon}x)$ as in Case I and the rest of the argument is completely the same.
\end{proof}

\subsection{Half hyperbolic neighborhoods}\label{subsec:half-nbd}

We have assumed that the bounded type Siegel disk $\Delta$ is bounded in $\C$. Without loss of generality, in the following we assume further that $f(\infty)\in\Delta$ (hence $G(\infty)\in\D$). Then for any connected set $V_0$ in $\EC\setminus\overline{\D}$, every component of $G^{-1}(V_0)$ in $\EC\setminus\overline{\D}$ is a bounded set in $\C\setminus\overline{\D}$.
Let $\widetilde{\MP}(G)$ be defined in \eqref{equ:P-G-tilde}. There exists a constant $r_0>1$ such that
\begin{equation}\label{equ:r-pri}
(\widetilde{\MP}(G)\setminus\T)\cap\overline{\A}_{r_0}=\emptyset,
\end{equation}
where $\A_{r_0}=\{z\in\C:1/r_0< |z|<  r_0\}$. Let $I=(a,b) \subset \T$ be an open arc with $\overline{I}\neq\T$. Then $\widetilde{\MP}(G) \setminus I$ is a compact subset of $\EC$. We denote
\begin{equation}\label{equ:Omega-I}
\Omega_I:= \text{The component of } \EC\setminus (\widetilde{\MP}(G) \setminus I) \text{ containing } I.
\end{equation}
Then $\Omega_I$ is a domain which is symmetric about $\T$.

\begin{defi}[Half hyperbolic neighborhoods]
For any given $d>0$, let
\begin{equation}\label{equ:H-d-I}
H_d(I) := \big\{z \in \Omega_I \,:\, \dist_{\Omega_I}(z , I) < d \text{~and~} |z| > 1 \big\}
\end{equation}
be a \textit{half hyperbolic neighborhood} of $I$ in $\EC\setminus\overline{\D}$.
See Figure \ref{Fig_half-nbd}.
\end{defi}

\begin{figure}[!htpb]
  \setlength{\unitlength}{1mm}
  \centering
  \includegraphics[width=0.7\textwidth]{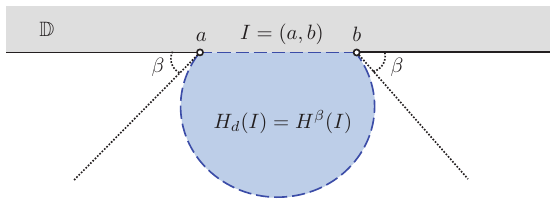}
  \caption{A half hyperbolic neighborhood $H_d(I)$ of the open interval $I=(a,b)$. }
  \label{Fig_half-nbd}
\end{figure}

Let $\Crit(f)$ be the set of all critical points of $f$ and we denote
\begin{equation}
\widetilde{\Crit}(f):=\{c\in\Crit(f):~\exists\, n\geqslant 0 \text{ such that } f^n(c)\in\partial\Delta\}.
\end{equation}
Note that $\Crit(f)\subset\EC\setminus\Delta$. Denote
\begin{equation}
\Crit(G):=\phi(\Crit(f)) \text{\quad and\quad} \widetilde{\Crit}(G):=\phi(\widetilde{\Crit}(f)).
\end{equation}
We denote
\begin{equation}\label{equ:crit-T}
\begin{split}
\MCV:=&~\{G(c):c\in\Crit(G)\cap\T\}, \text{ and}\\
\widetilde{\MCV}:=&~\left\{G^n(c)
\left|
\begin{array}{l}
c\in\widetilde{\Crit}(G)\text{ and } n\geqslant 1 \text{ is minimal} \\
\quad\qquad\text{such that } G^n(c)\in\T
\end{array}
\right.
\right\}.
\end{split}
\end{equation}
Then $\MCV\subset\widetilde{\MCV}\subset\T$ and $\widetilde{\MCV}$ is a finite set.

\medskip
Let $r_0>1$ be the constant introduced in \eqref{equ:r-pri}.
Then there exists $0<\widetilde{r}_0\leqslant  \min\{\frac{1}{4},r_0-1\}$ such that every subarc $I\subset \T$ satisfies:
\begin{equation}\label{equ:r0-pri}
\text{if } \min\{|I|,|\sigma(I)|\}< \widetilde{r}_0,  \text{\quad then } \sharp\big(\overline{I}\cap G^n(\widetilde{\MCV})\big)\leqslant  1 \text{ for every } n\geqslant 0.
\end{equation}

\begin{lem}\label{lemma:swz}
For any $d>0$, there exists $r_1=r_1(d)\in(0, \widetilde{r}_0]$ such that if $I\subset \T$ is an arc satisfying $\min\{|I|,|\sigma(I)|\}< r_1$, then $H_{\tilde{d}}(I)$ is a bounded Jordan disk in $\C\setminus\overline{\D}$ for all $\tilde{d}\in(0,d+1]$.
\end{lem}

\begin{proof}
Let $d>0$ and $I\subset\T$ be an arc with $\overline{I}\neq\T$. Consider the hyperbolic domain
\begin{equation}\label{equ:Omega-I-pri}
\widetilde{\Omega}_I:=\EC\setminus(\T\setminus I)
\end{equation}
and denote
\begin{equation}\label{equ:H-widetilde}
\widetilde{H}_d(I): = \big\{z \in \widetilde{\Omega}_I \,:\, \dist_{\widetilde{\Omega}_I}(z , I) < d\big\}.
\end{equation}
Then the boundary $\partial\widetilde{H}_d(I)$ consists of two subarcs of Euclidean circles which are symmetric to $\T$ (see \cite[Chap.\,VI, Section 5]{MS93} and Figure \ref{Fig-Omega-I-hat}).
Let $\beta\in(0,\pi)$ be the angle between $\partial \widetilde{H}_d(I)$ and $\T\setminus I$. Then we can write
\begin{equation}\label{equ:beta-d}
\widetilde{H}^{\beta}(I)=\widetilde{H}_d(I), \text{\quad where\quad} \log\cot\big(\tfrac{\beta}{4}\big)=d.
\end{equation}
Thus there exists a small $r_1=r_1(d)>0$ such that if $|I|< r_1$ or $|\sigma(I)|< r_1$, then the Euclidean diameter of $\widetilde{H}_{d+1}(I)$ is small and
\begin{equation}\label{equ:H-2d}
\widetilde{H}_{d+1}(I)\subset\A_{r_0}\setminus(\T\setminus I),
\end{equation}
where $r_0>1$ is the number in \eqref{equ:r-pri}.
Since $\Omega_I\subset \widetilde{\Omega}_I$, we have
\begin{equation}\label{equ:H-tilde-d}
H_{\tilde{d}}(I)\subset \widetilde{H}_{\tilde{d}}(I)\setminus\overline{\D} \text{\quad for all } \tilde{d}\in(0,d+1].
\end{equation}

If $\Omega_I=\widetilde{\Omega}_I$, then the lemma holds immediately. Without loss of generality, we assume that $\Omega_I\subsetneq\widetilde{\Omega}_I$.
Note that $\Omega_I$ is a hyperbolic domain which is symmetric about the unit circle. The arc $I\subset\T$ is a geodesic in $\Omega_I$ with respect to the hyperbolic metric $\rho_{\Omega_I}(z)|dz|$. Consider the holomorphic universal covering $\pi:\D\to\Omega_I$. The preimages of $I$ under $\pi$ consist of countably many pairwise disjoint Euclidean arcs $\{L_i:i\in\N\}$ in $\D$ which are orthogonal to the unit circle. The $(d+1)$-neighborhood
\begin{equation}
S_{d+1}(L_i):=\{w\in\D: \dist_{\D}(w,L_i)<d+1\}
\end{equation}
of each $L_i$ with respect to the hyperbolic metric in $\D$ is a simply connected domain.
By \eqref{equ:H-2d} and \eqref{equ:H-tilde-d}, for any different $i$ and $j$, $S_{d+1}(L_i)$ and $S_{d+1}(L_j)$ lie in their respective fundamental domains and $\overline{S_{d+1}(L_i)}\cap \overline{S_{d+1}(L_j)}=\emptyset$. This implies that the restriction of $\pi$ in an open neighborhood of $\overline{S_{d+1}(L_i)}\cap\D$ in $\D$ is conformal and $\pi(S_{\tilde{d}}(L_i))=H_{\tilde{d}}(I)\cup (H_{\tilde{d}}(I))^*$ is a bounded Jordan disk for all $\tilde{d}\in(0,d+1]$. In particular, $H_{\tilde{d}}(I)$ is a bounded Jordan disk in $\C\setminus\overline{\D}$ for all $\tilde{d}\in(0,d+1]$.
\end{proof}

\begin{lem}\label{lemma:swz-2}
For any $d>0$, there exist $C_0'=C_0'(d)>0$ and $r_2=r_2(d)\in (0, r_1]$ such that if $I\subset \T$ is an arc satisfying $\min\{|I|,|\sigma(I)|\}< r_2$, then
\begin{enumerate}
\item $H_d(I)\subset\widetilde{H}_d(I)\setminus\overline{\D}\subset H_{d'}(I)$, where $d'=d+C_0'|I|<d+1$;
\item If $I\cap \big(\bigcup_{k=0}^{n-1}G^k(\MCV)\big)=\emptyset$ for some $n\geqslant 1$, let $J \subset \T$ be the arc satisfying $G^n(J) = I$ and $Q \subset \EC\setminus\overline{\D}$ be the component of $G^{-n}(H_d(I))$ satisfying $\partial Q \cap \T = J$. Then $Q \subset H_{d'}(J)$.
\end{enumerate}
\end{lem}

\begin{proof}
(a) Given $d>0$. By Lemma \ref{lemma:swz}, we first assume that $\min\{|I|,|\sigma(I)|\}< r_1$ such that $H_{\tilde{d}}(I)$ is a Jordan disk for all $\tilde{d}\in(0,d+1]$. Without loss of generality, we assume that $I=\{e^{\ii\theta}:|\theta|<\tfrac{|I|}{2}\}$. Denote $\BH_+=\{w\in\C:\re w>0\}$ and let $\widetilde{\Omega}_I=\EC\setminus(\T\setminus I)$ be defined in \eqref{equ:Omega-I-pri}.
Let
\begin{equation}\label{equ:varphi}
\varphi=\varphi_3\circ\varphi_2\circ\varphi_1:\widetilde{\Omega}_I\to\BH_+
\end{equation}
be the composition of the following three conformal maps (see Figure \ref{Fig-Omega-I-hat}), where
\begin{equation}
\begin{split}
\zeta=\varphi_1(z)=\frac{\ii}{\tan(\frac{|I|}{4})}\,\frac{1-z}{1+z}: ~&~\widetilde{\Omega}_I\to\C\setminus\big(\R\setminus(-1,1)\big), \\
\xi=\varphi_2(\zeta)=\frac{1+\zeta}{1-\zeta}: ~&~\C\setminus\big(\R\setminus(-1,1)\big)\to\C\setminus(-\infty,0], \text{ and}\\
w=\varphi_3(\xi)=\sqrt{\xi}: ~&~ \C\setminus(-\infty,0]\to \BH_+.
\end{split}
\end{equation}
Then we have $\varphi(I)=\varphi_3\circ\varphi_2\big((-1,1)\big)=\varphi_3(\R_+)=\R_+$, where $\R_+:=(0,+\infty)$.
Let $\widetilde{H}_{d}(I)$ be defined in \eqref{equ:H-widetilde}. A direct calculation shows that
\begin{equation}
\begin{split}
\varphi(\widetilde{H}_d(I))
=&~\big\{w \in \BH_+ \,:\, \dist_{\BH_+}(w , \R_+) < d\big\} \\
=&~\big\{r e^{\ii\theta}:r>0 \text{ and } |\theta|<\tfrac{\pi}{2}-\tfrac{\beta}{2}\big\},
\end{split}
\end{equation}
where $\beta\in(0,\pi)$ satisfies $d=\log\cot(\tfrac{\beta}{4})$. See \eqref{equ:beta-d}.

\medskip
Let $r_0>1$ be the number introduced in \eqref{equ:r-pri}. Define
\begin{equation}
\widehat{\Omega}_I:=\A(r_0)\setminus\big((\T\setminus I)\cup (-r_0,-1/r_0)\big).
\end{equation}
Let $\widetilde{\MP}(G)$ be the set defined in \eqref{equ:P-G-tilde}. By \eqref{equ:H-2d}, $\widehat{\Omega}_I$ is a simply connected domain satisfying
\begin{equation}\label{equ:Omega-several}
\widehat{\Omega}_I\cap(\widetilde{\MP}(G)\setminus\T)=\emptyset \text{\quad and\quad}
H_d(I)\subset\widetilde{H}_d(I)\subset\widehat{\Omega}_I\subset\Omega_I\subset\widetilde{\Omega}_I.
\end{equation}

\begin{figure}[!htpb]
  \setlength{\unitlength}{1mm}
  \centering
  \includegraphics[width=0.82\textwidth]{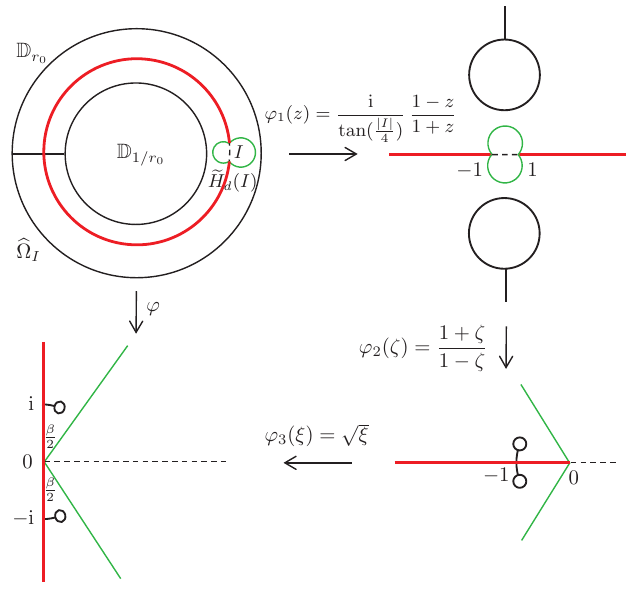}
  \caption{The hyperbolic domains $\widetilde{\Omega}_I$, $\widehat{\Omega}_I$ and their images under the conformal map $\varphi=\varphi_3\circ\varphi_2\circ\varphi_1$. }
  \label{Fig-Omega-I-hat}
\end{figure}

For a subset $Z$ of $\EC$, we denote $-Z:=\{-z:\,z\in Z\}$.
Let $\D_r(a)$ be the Euclidean disk centered at $a\in\C$ with radius $r>0$ and we denote $\D_r:=\D_r(0)$.
By a direct calculation, we have
\begin{equation}
\varphi_1(\widehat{\Omega}_I)=\C\setminus\Big(\big(\R\setminus(-1,1)\big)\cup \overline{D}\cup(-\overline{D})\cup L\cup (-L) \Big),
\end{equation}
where $D=\D_{\frac{R-r}{2}}\big(\frac{R+r}{2}\ii\big)$ and $L=[R,+\infty)\ii$ with
\begin{equation}
R=\frac{1}{\tan(\frac{|I|}{4})}\,\frac{r_0+1}{r_0-1} \text{\quad and\quad} r=\frac{1}{\tan(\frac{|I|}{4})}\,\frac{r_0-1}{r_0+1}.
\end{equation}
Therefore, there exist constants $C_1, C_2>0$ such that
\begin{equation}
\begin{split}
  \varphi_3\circ\varphi_2\big(\overline{D}\cup(-\overline{D})\cup L\cup (-L)\big)
  \subset &~\varphi_3\big(\D_{C_1|I|}(-1)\setminus\R\big) \\
  \subset &~\D_{C_2|I|}(\ii)\cup\D_{C_2|I|}(-\ii).
\end{split}
\end{equation}
Decreasing the length of $I$ if necessary, there exists $C_3>0$ such that
\begin{equation}\label{equ:appendix-varphi}
  \varphi(\widehat{\Omega}_I)\supset U_0\cup U_1 \cup U_\infty,
\end{equation}
where $U_0:=\D_{\frac{1}{2}}\cap\BH_+$, $U_\infty:=\BH_+\setminus\overline{\D}_{2}$ and
\begin{equation}
U_1:=\big\{r e^{\ii\theta}:r>0 \text{ and } |\theta|<\tfrac{\pi}{2}-C_3|I|\big\} \text{\quad with } C_3|I|<\tfrac{\beta}{4}.
\end{equation}

Note that $d=\log\cot(\frac{\beta}{4})$ with $\beta\in(0,\pi)$. By \eqref{equ:appendix-varphi}, if $|I|$ is small, then for every $z\in \partial\widetilde{H}_d(I)\setminus\partial I$ and $w=\varphi(z)$, we have
\begin{equation}
\dist_{\widehat{\Omega}_I}(z,I)\leqslant  \dist_{U_1}(w,\R_+)=\log\cot\big(\tfrac{\widehat{\beta}}{4}\big),
\end{equation}
where
\begin{equation}
\widehat{\beta}=2\left(\frac{\pi}{2}-\frac{\frac{\pi}{2}}{\frac{\pi}{2}-C_3|I|}\cdot\Big(\frac{\pi}{2}-\frac{\beta}{2}\Big)\right)
=\beta-\eta_0|I| \text{\quad with } \eta_0:=\frac{2C_3(\pi-\beta)}{\pi-2 C_3|I|}.
\end{equation}
If $|I|$ is small, then there exists $C_0'=C_0'(d)>0$ such that for all $z\in \widetilde{H}_d(I)$,
\begin{equation}\label{equ:app-compare}
\dist_{\widehat{\Omega}_I}(z,I)\leqslant  d+\log\frac{1+\tan\big(\frac{\eta_0|I|}{4}\big)\tan\big(\frac{\beta}{4}\big)}{1-\tan\big(\frac{\eta_0|I|}{4}\big)/\tan\big(\frac{\beta}{4}\big)}< d+C_0'|I|.
\end{equation}
Decreasing the length of $I$ if necessary, we assume that $C_0'|I|<1$ and hence
\begin{equation}
d':=d+C_0'|I|<d+1.
\end{equation}
Denote
\begin{equation}
\widehat{H}_{d'}(I): =\big\{z \in \widehat{\Omega}_I \,:\, \dist_{\widehat{\Omega}_I}(z , I) < d'\big\}.
\end{equation}
By \eqref{equ:Omega-several} and \eqref{equ:app-compare}, we have
\begin{equation}\label{equ:H-d-H-d-pri}
H_d(I)\subset \widetilde{H}_d(I)\setminus\overline{\D}\subset \widehat{H}_{d'}(I)\setminus\overline{\D}\subset H_{d'}(I).
\end{equation}

(b) Suppose $I\cap \big(\bigcup_{k=0}^{n-1}G^k(\MCV)\big)=\emptyset$ for some $n\geqslant 1$.
Let $\widehat{Q}$ be the connected component of $G^{-n}(\widehat{\Omega}_I)$ containing $J$, where $J\subset\T$ is the unique arc satisfying $G^n(J)=I$. We claim that $\widehat{Q}$ is simply connected and that $G^n:\widehat{Q}\to\widehat{\Omega}_I$ is a conformal map.

In fact, let $\widehat{Q}_1^+\subset\EC\setminus\overline{\D}$ be the unique component of $G^{-1}(\widehat{\Omega}_I\setminus\overline{\D})$ such that $J_1\subset \partial \widehat{Q}_1^+$, where $J_1:=G^{n-1}(J)$. Since $(\widehat{\Omega}_I\setminus\overline{\D})\cap\widetilde{\MP}(G)=\emptyset$ and $\widehat{\Omega}_I\setminus\overline{\D}$ is simply connected, it follows that $G:\widehat{Q}_1^+\to \widehat{\Omega}_I\setminus\overline{\D}$ is conformal. Let $\widehat{Q}_1^-\subset\D$ be the unique component of $G^{-1}(\widehat{\Omega}_I\cap\D)$ such that $J_1\subset \partial \widehat{Q}_1^-$. Then similarly, $G:\widehat{Q}_1^-\to \widehat{\Omega}_I\cap\D$ is also conformal. Let $\widehat{Q}_1$ be the component of $G^{-1}(\widehat{\Omega}_I)$ containing $J_1$. We have $\widehat{Q}_1\supset \widehat{Q}_1^+\cup \widehat{Q}_1^-\cup J_1$. By \eqref{equ:r0-pri}, we have $\sharp(\overline{I}\cap \widetilde{\MCV})\leqslant  1$. By the Riemann-Hurwitz formula, $\widehat{Q}_1$ is simply connected and contains at most one critical point (without counting multiplicity). If $c\in \widehat{Q}_1\setminus\T$ is a critical point, then its symmetric image $c^*=1/\overline{c}$ about $\T$ is also a critical point which is contained in $\widehat{Q}_1$. This is a contradiction. Since $I\cap \MCV=\emptyset$, it follows that $\widehat{Q}_1$ contains no critical point and $G:\widehat{Q}_1\to \widehat{\Omega}_I$ is a homeomorphism. In particular, we have $\widehat{Q}_1=\widehat{Q}_1^+\cup \widehat{Q}_1^-\cup J_1$ and hence $G:\widehat{Q}_1\to \widehat{\Omega}_I$ is conformal. By \eqref{equ:r0-pri}, $\sharp(I\cap G^k(\widetilde{\MCV}))\leqslant  1$ for every $0\leqslant  k\leqslant  n-1$.
Since $(\widehat{Q}_1\setminus\T)\cap\widetilde{\MP}(G)=\emptyset$, it follows that $G^n:\widehat{Q}\to \widehat{\Omega}_I$ is conformal by an inductive argument.

\medskip
Let $Q'$ be the component of $G^{-n}(\widehat{H}_{d'}(I))$ containing $J$. Since $\widehat{H}_{d'}(I)\subset \widehat{\Omega}_I$, we have $Q'\subset \widehat{Q}$. Note that $G^n:\widehat{Q}\to \widehat{\Omega}_I$ is conformal. By the Schwarz-Pick Lemma, we have
\begin{equation}
Q' = \{z \in \widehat{Q} \,:\, \dist_{\widehat{Q}}(z , J) < d'\}.
\end{equation}
By \eqref{equ:H-d-H-d-pri}, we have $H_d(I)\subset \widehat{H}_{d'}(I)\setminus\overline{\D}$. Then for the component $Q \subset \EC\setminus\overline{\D}$ of $G^{-n}(H_d(I))$ satisfying $\partial Q \cap \T = J$, we have $Q\subset Q'$.
Since $\widehat{Q}\subset \Omega_J$, where $\Omega_J$ is defined similarly as $\Omega_I$ in \eqref{equ:Omega-I}, it follows that $Q'\subset H_{d'}(J)$ and hence $Q \subset H_{d'}(J)$.
\end{proof}

Based on the proof of Lemma \ref{lemma:swz-2} (see also Lemma \ref{lemma:swz}), we show that the angle between $\partial H_d(I)\setminus I$ and $\T\setminus I$ is well-defined and equal to the angle between $\partial \widetilde{H}_d(I)$ and $\T\setminus I$. Indeed, we verify this by comparing the hyperbolic metrics of $\Omega_I$ with $\widetilde{\Omega}_I$ near the end points of $I$.
For $U_0=\D_{\frac{1}{2}}\cap\BH_+$, the following map is conformal:
\begin{equation}
\psi(w)=\frac{1}{\ii}\left(\frac{2w-\ii}{2w+\ii}\right)^2: U_0\to \BH_+.
\end{equation}
Note that $\rho_{\BH_+}(w)=\frac{1}{\re w}$. A direct calculation shows that for every $w=x+y\ii\in\D_{\frac{1}{4}}\cap\BH_+$, we have
\begin{equation}\label{equ:hyper-compare}
\begin{split}
&~\frac{\rho_{U_0}(w)}{\rho_{\BH_+}(w)}
=\frac{\re w\,|\psi'(w)|}{\re \psi(w)}=\frac{4x\cdot\frac{|2w-\ii|}{|2w+\ii|^3}}{\re\big(\frac{2w-\ii}{2w+\ii}\big)\,\im\big(\frac{2w-\ii}{2w+\ii}\big)} \\
=&~\left(1+\frac{x^2}{\frac{1}{4}-y^2-x^2}\right) \left(1+\frac{x^2}{(\frac{1}{2}+y)^2}\right)^{\frac{1}{2}} \left(1+\frac{x^2}{(\frac{1}{2}-y)^2}\right)^{\frac{1}{2}} = 1+\MO(x^2).
\end{split}
\end{equation}
For $U_\infty=\BH_+\setminus\overline{\D}_2$ and $w\in\BH_+\setminus\overline{\D}_4$, we have
\begin{equation}
\frac{\rho_{U_\infty}(w)}{\rho_{\BH_+}(w)}=\frac{\rho_{U_0}(\frac{1}{w})}{\rho_{\BH_+}(\frac{1}{w})}=1+\MO\big(\big(\re\tfrac{1}{w}\big)^2\big).
\end{equation}
By \eqref{equ:varphi}, there exists $C_1\in(0,1]$ such that for any $z\in\widetilde{\Omega}_I$ satisfying $\dist_{\C}(z,\partial I)=\kappa|I|$ with $\kappa\in(0, C_1]$, then $z\in\widehat{\Omega}_I\subset\Omega_I\subset\widetilde{\Omega}_I$ and $w=\varphi(z)\in \D_{\frac{1}{4}}\cap\BH_+$ or $w\in \BH_+\setminus\overline{\D}_4$.
By \eqref{equ:appendix-varphi}, the hyperbolic metrics satisfy
\begin{equation}\label{equ:rho-compare}
1\leqslant  \frac{\rho_{\Omega_I}(z)}{\rho_{\widetilde{\Omega}_I}(z)}\leqslant  \frac{\rho_{\widehat{\Omega}_I}(z)}{\rho_{\widetilde{\Omega}_I}(z)}
=\frac{\rho_{\varphi(\widehat{\Omega}_I)}(w)}{\rho_{\BH_+}(w)} \leqslant  1+\MO(\kappa).
\end{equation}
This implies that the angle $\beta=\beta(d)\in (0,\pi)$ between $\partial H_d(I)\setminus I$ and $\T\setminus I$ is well-defined and equal to the angle between $\partial \widetilde{H}_d(I)$ and $\T\setminus I$. Hence we can write (compare \eqref{equ:beta-d} and see Figure \ref{Fig_half-nbd})
\begin{equation}\label{equ:H-beta}
H^{\beta}(I)=H_d(I), \text{\quad where\quad} \log\cot\big(\tfrac{\beta}{4}\big)=d.
\end{equation}

For two different points $z_1,z_2\in\C$, we use $\lfloor z_1,z_2 \rfloor$ to denote the segment in $\C$ connecting $z_1$ with $z_2$ (This notation is used to distinguish from arcs on the circle).
For $0<\omega<1$ and an arc $I\subset\T$, we denote
\begin{equation}
S_{\omega}(I):=\big\{z\in\C:1<|z|<1+\omega|I| \text{ and } \lfloor 0,z \rfloor\cap I\neq\emptyset\big\}.
\end{equation}

\begin{lem}\label{lemma:hn}
For any given $n\geqslant 1$, there exist $\widetilde{r}_2=\widetilde{r}_2(n)>0$, $0<\omega=\omega(n)<1$ and $0<\beta_0=\beta_0(n)<\pi$ such that for any $n-1$ points $x_1 <x_2< \cdots < x_{n-1}$ in $I= (a, b) =(x_0,x_n)\subset \T$ with $|I|< \widetilde{r}_2$, the following two statements hold:
\begin{itemize}
\item $H^{\beta_0}\big((x_i, x_{i+1})\big)$ is a Jordan disk for all $0\leqslant  i\leqslant  n-1$; and
\item $\overline{S_{\omega}(I)}\setminus\T\subset \bigcup_{i=0}^{n-1}H^{\beta_0}\big((x_i, x_{i+1})\big)$.
\end{itemize}
\end{lem}

This lemma implies that $\widetilde{r}_2$, $\omega$ and $\beta_0$ depend on the number (i.e., $n+1$) of the points $x_0$, $x_1$, $\cdots$, $x_n$ but not on their position.

\begin{proof}
For any $\beta\in (0,\pi)$, by \eqref{equ:beta-d} and \eqref{equ:H-beta}, we write $H_d(I)=H^\beta(I)$ and $\widetilde{H}_d(I)=\widetilde{H}^\beta(I)$ for the subarc $I\subset\T$ with small length, where
\begin{equation}
d=\log\cot\big(\tfrac{\beta}{4}\big).
\end{equation}
For given $\beta_0\in (0,\pi)$, by Lemmas \ref{lemma:swz} and \ref{lemma:swz-2}(a), there exists a small $r_2=r_2(\beta_0)>0$ such that if $|I|< r_2$, then
\begin{itemize}
\item $\widetilde{H}^\beta(I)$ and $H^{\beta}(I)$ are Jordan disks, for any $\beta\in[\beta_0,\pi)$;
\item $\widetilde{H}_d(I)$ and $H_d(I)$ are Jordan disks, for any $d\in(0,d_0+1]$, where $d_0=\log\cot\big(\frac{\beta_0}{4}\big)$;
\item $\widetilde{H}_d(I)\setminus\overline{\D}\subset H_{d'}(I)$, where $d\in(0,d_0]$ and $d'=d+C_0'|I|<d+1$.
\end{itemize}
Therefore, $\widetilde{H}^\beta(I)\setminus\overline{\D}\subset H^{\beta'}(I)$, where $\beta'\in(0,\pi)$ satisfies
$\log\cot\big(\tfrac{\beta'}{4}\big)=d'$.
Without loss of generality we assume that $r_2>0$ is small such that for any $d\in(0,d_0]$ and any arc $I\subset\T$ with $|I|<r_2$,
\begin{equation}
4\arctan \big(e^{-(d+C_0'|I|)}\big)> 3\arctan\big(e^{-d}\big).
\end{equation}
This implies that $\beta'> \frac{3}{4}\beta$. Therefore, we have
\begin{equation}\label{equ:compare}
\widetilde{H}^\beta(I)\setminus\overline{\D}\subset H^{\beta'}(I) \subset H^{\frac{3\beta}{4}}(I).
\end{equation}

If $n=1$, we take $\beta_1=\frac{\pi}{8}$. There exist small $r_1'>0$ and $0<\omega_1<1$ such that if $|I|<r_1'$, then $H^{\beta_1}(I)$ is a Jordan disk and
\begin{equation}
\overline{S_{\omega_1}(I)}\setminus\T\subset \widetilde{H}^{\frac{4\beta_1}{3}}(I)\setminus\overline{\D}\subset H^{\beta_1}(I).
\end{equation}
Inductively, we assume that for every $1\leqslant k\leqslant n-1$, where $n\geqslant 2$, there exist $r_k'>0$, $0<\omega_k<1$ and $0<\beta_k<\pi$ such that for any $k-1$ points $x_1' <x_2'< \cdots < x_{k-1}'$ in $I= (a, b) =(x_0',x_k')\subset \T$ with $|I|< r_k'$, the following two statements hold:
\begin{itemize}
\item $H^{\beta_k}\big((x_i', x_{i+1}')\big)$ is a Jordan disk for all $0\leqslant  i\leqslant  k-1$; and
\item $\overline{S_{\omega_k}(I)}\setminus\T\subset \bigcup_{i=0}^{k-1}H^{\beta_k}\big((x_i', x_{i+1}')\big)$.
\end{itemize}

Note that any $n-1$ points $x_1 <x_2< \cdots < x_{n-1}$ in $I= (a, b) =(x_0,x_n)\subset \T$ divides $I$ into $n$ subarcs and there exists one subarc $I_j=(x_j,x_{j+1})$ with $0\leqslant j\leqslant n-1$ satisfying $|I_j|\geqslant|I|/n$. Hence there exist positive numbers
\begin{equation}
r_n'\leqslant\min_{1\leqslant k\leqslant n-1}\{r_k'\}, \quad
\omega_n\leqslant\min_{1\leqslant k\leqslant n-1}\{\omega_k\} \text{\quad and\quad}
\beta_n\leqslant\min_{1\leqslant k\leqslant n-1}\{\beta_k\}
\end{equation}
depending on $n$
such that for any $n-1$ points $x_1 <x_2< \cdots < x_{n-1}$ in $I$ with $|I|< r_n'$, then
\begin{itemize}
\item $H^{\beta_n}\big((x_i, x_{i+1})\big)$ is a Jordan disk for all $0\leqslant  i\leqslant  n-1$; and
\item $\overline{S_{\omega_n}(I)}\setminus\T=\overline{S_{\omega_n}\big((x_0,x_j)\big)\cup S_{\omega_n}\big(I_j\big)\cup S_{\omega_n}\big((x_{j+1},x_n)\big)}\setminus\T$ is contained in $\bigcup_{i=0}^{j-1}H^{\beta_j}\big((x_i, x_{i+1})\big)\cup H^{\beta_n}\big(I_j\big)\cup\bigcup_{i=j+1}^{n-1}H^{\beta_{n-j-1}}\big((x_i, x_{i+1})\big)$ and is also contained in $\bigcup_{i=0}^{n-1}H^{\beta_n}\big((x_i, x_{i+1})\big)$, where $I_j=(x_j,x_{j+1})$ is a subarc satisfying $|I_j|\geqslant|I|/n$.
\end{itemize}
The proof is complete by setting $\widetilde{r}_2:=r_n'$, $\omega:=\omega_n$ and $\beta_0:=\beta_n$.
\end{proof}

Lemma \ref{lemma:hn} will be used in the proof of Proposition \ref{prop:three special regions} and the corresponding number $n\geqslant 1$ there is an integer determined by $G$.

\subsection{Contraction of $G^{-1}$}\label{subsec:G-1}

Let $z_0,z_1,z_2\in\C\setminus\D$ be three different points. In this paper we assume that the angle $\angle\, z_1z_0z_2$ is measured in the \textit{logarithmic plane} of $G$. Specifically, when we write $\beta=\angle\, z_1z_0z_2\in[0,2\pi)$, it means that there exists $x\in\R$ such that the half-strip $\{w\in\C:x<\re w<x+1 \text{ and }\im w\leqslant  0\}$ contains $z_0', z_1', z_2'$ with $z_i=\Exp(z_i')$, where $\Exp(z)=e^{2\pi\ii z}$ for $1\leqslant  i\leqslant  3$ and
\begin{equation}\label{equ:angle}
\beta=\arg\left(\frac{z_2'-z_0'}{z_1'-z_0'}\right)\in [0,2\pi).
\end{equation}

Let $\MP(G)=\phi\big(\MP(f)\setminus\Delta\big)$ be the postcritical set defined in \eqref{equ:P-G}. Note that $\MP(G)\subset\widetilde{\MP}(G)$. By \eqref{equ:r-pri}, the following set is non-empty:
\begin{equation}\label{equ:Omega}
\Omega:=\text{The unique component of } \EC \setminus(\MP(G) \cup \overline{\D}) \text{ with }\T \subset \partial \Omega.
\end{equation}
In this subsection, we specify the places where $G^{-1}$ contracts the hyperbolic metric $\rho_{\Omega}(z)|dz|$ in $\Omega$ strictly.
By the definition of $H_d(I)$ in \eqref{equ:H-d-I}, we have $H_d(I)\subset\Omega$.
Let $\MCV$ and $\widetilde{r}_0>0$ be defined as in \eqref{equ:crit-T} and \eqref{equ:r0-pri} respectively.
The principle of the following result is observed essentially in \cite[Lemma 1.11]{Pet96}.

\begin{lem}\label{lemma:Pe}
Let $V$ be a Jordan disk in $\Omega$ such that $\overline{I}=\overline{V}\cap \T$, where $I$ is an arc with $\sigma(I)< \widetilde{r}_0$, $\partial I\cap\MCV=\{v\}$ and the \emph{angle} between $\partial V\setminus I$ and $\T\setminus I$ at $v$ is well-defined and positive. Then there exists $0<\mu=\mu(V)<1$ such that for every component $U$ of $G^{-1}(V)\cap\Omega$ with $\sharp(\overline{U}\cap\T)=1$ and for any $z\in U$, we have
\begin{equation}
\rho_\Omega(z)\leqslant  \mu\,\rho_\Omega(G(z))|G'(z)|.
\end{equation}
In particular, if $V=H_d(I)$ (resp. $V=H^\beta(I)$) is simply connected for some $d>0$ (resp. for some $0<\beta<\pi$), then $\mu$ depends only on $d$ (resp. $\beta$).
\end{lem}

\begin{proof}
Since $\sigma(I)< \widetilde{r}_0$, \eqref{equ:r0-pri} shows that $V\cap \widetilde{\MCV}=\emptyset$ and $\sharp(\partial V\cap \widetilde{\MCV})= 1$. By Lemma \ref{lem:take-pre}, every component $U$ of $G^{-1}(V)$ in $\EC\setminus\overline{\D}$ is a Jordan disk and $G:U\to V$ is conformal. Let $c\in\Crit(G)\cap\T$ be the unique critical point such that $G(c)=v$, where $G$ has local degree $2\ell-1$ at $c$ with $\ell\geqslant 2$. Without loss of generality, we assume that $I=(a,v)$, i.e., the critical value $v$ is the right end point of $I$. There exist $\ell$ components of $G^{-1}(V)$ in $\EC\setminus\overline{\D}$ which contain $c$ as their boundary point. We label them by $U_1$, $\cdots$, $U_\ell$ clockwise such that $\overline{U}_\ell\cap\T=[a',c]$ with $G(a')=a$ and $\overline{U}_i\cap\T=\{c\}$ for $1\leqslant  i\leqslant  \ell-1$. See Figure \ref{Fig_hyper-cone}.

\begin{figure}[!htpb]
 \setlength{\unitlength}{1mm}
 \centering
  \includegraphics[width=0.9\textwidth]{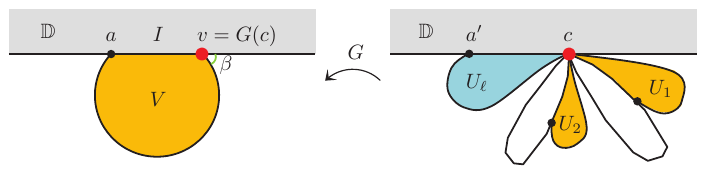}
  \caption{A sketch of the local mapping relation of $G$ near the critical point $c$ of local degree $2\ell-1$ with $\ell=3$. The strictly expanding regions $U_1$ and $U_2$ of $G$ with respect to the hyperbolic metric in $\Omega$ are colored yellow.}
 \label{Fig_hyper-cone}
\end{figure}

Let $U:=U_i$ for some $1\leqslant  i\leqslant  \ell-1$.
Since $G$ is a quasi-regular map satisfying $G(\T)=\T$, there exists a small Jordan disk $W$ containing $c$, such that $G$ can be written as $G=g\circ \varphi$, where
\begin{itemize}
\item $\varphi:W\to\varphi(W)$ is quasiconformal and $g:\varphi(W)\to G(W)$ is a proper holomorphic map having exactly one critical value $v$; and
\item $W\cap\T$ is an arc and $\varphi(W\cap\T)\subset\T$.
\end{itemize}
Note that the angle $\beta>0$ between $\partial V\setminus I$ and $\T\setminus I$ at $v$ is well-defined, i.e., as $z\in\partial V\setminus I$ and $z\to v$, the limit of $\angle\, avz$ exists and equals to $\pi-\beta$. Since the bounded turning property is preserved under quasiconformal mappings (see \cite[p.\,100]{LV73}), there exists a small number $0<\beta'<\pi/3$ depending only on $V$ (hence also on $\beta$), such that for all $z\in U\cap W$, we have
\begin{equation}\label{equ:angle-small}
\angle\, a'cz\in(\beta',\pi-\beta').
\end{equation}

Let $\Omega'$ be the component of $\EC\setminus \big(\overline{\D} \cup G^{-1}(\MP(G)\cup\overline{\D})\big)$ containing $U$. Then $\Omega' $ is a proper subset of $\Omega$ and $G:\Omega' \to \Omega$ is a holomorphic covering map (see Section \ref{subsec:quasi-Blaschke}).
By the Schwarz-Pick Lemma,
\begin{equation}
\rho_\Omega(G(z))|G'(z)|=\rho_{\Omega'}(z), \text{\quad for all }z\in\Omega'.
\end{equation}
Except for a small neighborhood of $c$, the remaining part of $U$ is compactly contained in $\Omega$. To prove this lemma, it suffices to prove that for all $z\in U\cap W$, we have
\begin{equation}
\frac{\rho_{\Omega'}(z)}{\rho_\Omega(z)}\geqslant \frac{1}{\mu'}>1,
\end{equation}
where $0<\mu'<1$ is a number depending only on $\beta'$.
Based on \eqref{equ:hyper-compare} and \eqref{equ:angle-small}, by changing coordinates, we may assume that
$\Omega=\{z\in\C:\re z>0\}$ and $\Omega'=\{z\in\C:|\arg z|<(\pi-\beta')/2\}$.
Then a direct calculation shows that one can take
\begin{equation}
\mu':=(\pi-\beta')/\pi\in(0,1).
\end{equation}

In particular, suppose $V=H_d(I)=H^\beta(I)$ for some $d>0$ and $0<\beta<\pi$. Then the angle between $\partial V\setminus I$ and $\T\setminus I$ at $v$ is $\beta$ and \eqref{equ:angle-small} holds for all $z\in U$ for a constant $\beta'>0$ depending only on $\beta$ (or $d$), but not depending on the length of $I$.  The proof is complete.
\end{proof}

Let $\Crit(G)\cap\T$ be the set of critical points of $G$ on $\T$. Counting without multiplicity, we denote
\begin{equation}\label{equ:crit-t0}
  t_0:=\sharp\,(\Crit(G)\cap\T)=\sharp\,\MCV\geqslant 1.
\end{equation}
For each $c_j=e^{\ii\theta_j}\in \Crit(G)\cap\T$ with $\theta_j\in\R$, where $1\leqslant  j\leqslant  t_0$, there exist finitely many connected components $W_{j,1}$, $\cdots$, $W_{j,\ell(j)}$ of $G^{-1}(\D)\setminus\overline{\D}$ attaching to $c_j$. Similar to the proof of \eqref{equ:angle-small}, we conclude that there exist two small numbers $\widehat{r}>0$ and $\widehat{\beta}\in(0,\frac{\pi}{2})$ such that
\begin{equation}\label{equ:hat-r-beta}
\begin{split}
&~\Big(\bigcup_{i=1}^{\ell(j)}W_{j,i}\Big)\cap\{z\in\C:|z-c_j|<\widehat{r}\} \\
\subset &~ \big\{z=c_j+r e^{\ii\theta}:0<r<\widehat{r}\text{ and }|\theta-\theta_j|<\tfrac{\pi}{2}-\widehat{\beta}\big\}.
\end{split}
\end{equation}

In the rest of this paper, we always assume that $d > 1$ is large (based on Lemma \ref{lemma:hn}, the number $d$ will be fixed after Proposition \ref{prop:three special regions}) such that
\begin{equation}\label{equ:d-choice}
H_d(I)=H^\beta(I) \text{\quad with\quad} 0<\beta=\beta(d)\leqslant \tfrac{\widehat{\beta}}{2},
\end{equation}
where $I$ is any subarc of $\T$ satisfying $|I|< r_2$ and $r_2=r_2(d)>0$ is the number introduced in Lemma \ref{lemma:swz-2}.

\section{Admissible sequences and contraction regions}\label{sec-pull-back}

For an arc $I \subset \T$ and a number $\kappa> 0$, we use $ \kappa I$ to denote the arc in $\T$ which has dynamical length $\kappa \sigma(I)$ and has the same middle point as $I$ (with respect to the dynamical length). Let $\MCV$ and $\widetilde{\MCV}$ be the finite subsets of $\T$ defined in \eqref{equ:crit-T}.
Let $I=(a,b)\subset\T$ be an arc with $\sigma(I)< r_2/2$. Then $\sharp\big((2\overline{I})\cap\widetilde{\MCV}\big)\leqslant  1$ and $H_{\tilde{d}}(2I)$ is a Jordan disk for all $\tilde{d}\in (0,d+1]$ by Lemma \ref{lemma:swz}.

\subsection{Non-critical predecessors}\label{subsec:non-crit-s}

Suppose $I\cap \MCV=\emptyset$. There exists a unique arc $J \subset \T$ such that $G(J) = I$.
We call $J$ a $\emph{non-critical predecessor}$ of $I$ and $I\to J$ a \textit{non-critical pullback}.
The pullback $I\to J$ induces a branch $\Phi$ of $G^{-1}$ such that $\Phi(I)=J$.
We say that $\Phi$ is the inverse branch of $G$ \textit{associated} to the pullback $I\to J$.

For the inverse branch $\Phi$, there exist a pair of critical points $c^-,c^+\in\T$ and a pair of critical values $v^-,v^+\in\MCV$ such that
\begin{itemize}
\item $I\subset(v^-,v^+)$ and $J\subset(c^-, c^+)$;
\item $(c^-, c^+)$ does not contain any critical point of $G$; and
\item $\Phi\big((v^-, v^+)\big)=(c^-, c^+)$.
\end{itemize}
We call $v^-$ and $v^+$ the two \textit{singular points} associated to $\Phi$.
In the following we shall see that this will not cause ambiguity when $\T$ contains exactly one critical point of $G$.

\medskip
Suppose $I\cap \big(\bigcup_{k=0}^{n-1}G^k(\MCV)\big)=\emptyset$ for some $n\geqslant 1$. Let $J \subset \T$ be the unique arc such that $G^n(J) = I$.
Then $G^k(J)$ is a non-critical predecessor of $G^{k+1}(J)$ for all $0\leqslant  k\leqslant  n-1$.
By Lemma \ref{lemma:swz-2}(b), for every $0\leqslant  k\leqslant  n-1$, the unique component $Q\subset\EC\setminus\overline{\D}$ of $G^{-(n-k)}(H_d(I))$ with $\partial Q \cap \T = G^k(J)$ satisfies $Q \subset H_{d'}\big(G^k(J)\big)$, where $d'=d+C_0'|I|<d+1$.

\subsection{Critical  predecessors}\label{subsec:crit-s}
Let $0< \delta < \eta < \frac{1}{2}$ be two numbers which will be specified later.

\medskip
\textbf{Case (1)}. Suppose $\MCV\cap I\ne \emptyset$. Then $\MCV \cap  \big((1+2\delta)I  \setminus  I \big) =  \emptyset$ since $\sharp\big((2\overline{I})\cap\widetilde{\MCV}\big)\leqslant  1$. Assume that $I$ contains a critical value $v = G(c)$ for some $c \in \Crit(G)\cap\T$ with local degree $2\ell-1$, where $\ell\geqslant 2$.

Let $a', b' \in \T$ such that $G(a') = a$ and $G(b') = b$. For any $\tilde{d}\in(0,d+1]$, there are exactly $\ell$ components of $G^{-1}(H_{\tilde{d}}(I))$ in $\EC\setminus\overline{\D}$, say $Q_1$, $\cdots$, $Q_\ell$, which attach to $c$ clockwise such that $\overline{Q}_1\cap\T=[c, b']$ and $\overline{Q}_\ell\cap\T= [a', c]$.

For each $1\leqslant  i\leqslant  \ell$, we associate an arc $J\subset\T$ to $I$ for each branch of $G^{-1}$ sending $H_{\tilde{d}}(I)$ to $Q_i$. In particular, if we consider $Q_1$, then we take $J$ such that $(c, b') \subset J$. If we consider $Q_\ell$, then we take $J$ such that $(a', c) \subset J$. If we consider $Q_i$ with $2\leqslant  i\leqslant  \ell-1$, then we take $J$ such that $(c, b') \subset J$ or $(a', c) \subset J$. In all situations, $c$ is an end point of $J$. To fix the idea, we only consider the situation $(c,b') \subset J$  and the remaining cases can be treated in the same way. Then we have the following two types (see Figure \ref{Fig_crit-pullback}):
\begin{enumerate}
\item[(i)] If $\MCV\cap I\ne \emptyset$ and $\sigma\big((v,b)\big) < (1 - \eta)\sigma(I)$, we take $J=(c,x)$ such that $\sigma(J) = (1-\eta/3) \sigma(I)$; and
\item[(ii)] If $\MCV\cap I\ne \emptyset$ and $\sigma\big((v,b)\big)\geqslant(1 - \eta)\sigma(I)$, we take $J=(c,x)$ such that $\sigma(J) = (1+\delta) \sigma(I)$.
\end{enumerate}

\begin{figure}[!htpb]
 \setlength{\unitlength}{1mm}
 \centering
  \includegraphics[width=0.92\textwidth]{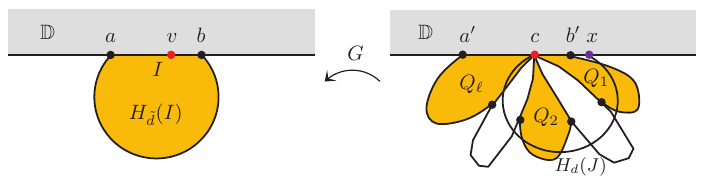}
  \caption{A sketch of critical pullbacks under $G$ of types (i) and (ii). From the arc $I=(a,b)$ we obtain a new arc $J=(c,x)$.}
 \label{Fig_crit-pullback}
\end{figure}

\medskip
\textbf{Case (2)}. Suppose $\MCV\cap I=\emptyset$ and $\MCV \cap \big((1+2\delta)I  \setminus  I \big) \ne \emptyset$. Note that in this case one may consider the non-critical predecessor of $I$ as in Section \ref{subsec:non-crit-s}. However, we can also define critical predecessors as follows.

Let $v= G(c)$ be the unique critical value in $\MCV\cap  \big((1+2\delta)I  \setminus  I \big)$ for $c \in \Crit(G)\cap\T$ of local degree $2\ell-1$ with $\ell\geqslant 2$.
To fix the idea, we only consider the situation that $v$ is on the right of $I$. Denote $\widehat{I}:=(1+2\delta)I =(\widehat{a},\widehat{b})$ and let $\widehat{a}', \widehat{b}' \in \T$ be two points such that $G(\widehat{a}') = \widehat{a}$ and $G(\widehat{b}') = \widehat{b}$. For any $\tilde{d}\in(0,d+1]$, there are $\ell$ components $\widehat{Q}_1$, $\cdots$, $\widehat{Q}_\ell$ of $G^{-1}(H_{\tilde{d}}(\widehat{I}))$ in $\EC\setminus\overline{\D}$ which attach to $c$ clockwise.

We use $Q_i$ to denote the component of $G^{-1}(H_{\tilde{d}}(I))$ which is contained in $\widehat{Q}_i$ for every $1\leqslant  i\leqslant  \ell$. Note that $Q_1$, $\cdots$, $Q_{\ell-1}$ do not attach to the unit circle.
For each $1\leqslant  i\leqslant  \ell-1$, we associate an arc $J$ to $I$ such that $c$ is an end point of $J$. Specifically, we have the following third type (see Figure \ref{Fig_crit-pullback-iii}):
\begin{enumerate}
\item[(iii)] If $\MCV \cap \big((1+2\delta)I  \setminus  I \big) \ne \emptyset$, we take $J=(c,x)$ such that $\sigma(J) = \sigma(I)/2$.
\end{enumerate}

\begin{figure}[!htpb]
 \setlength{\unitlength}{1mm}
 \centering
  \includegraphics[width=0.92\textwidth]{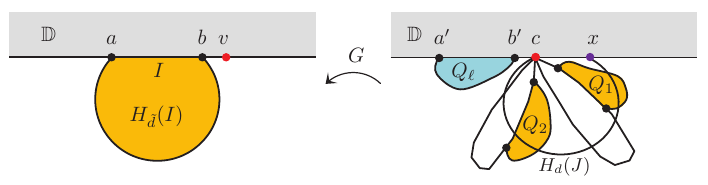}
  \caption{A sketch of critical pullbacks $Q_1$, $\cdots$, $Q_{\ell-1}$ under $G$ of type (iii), which are colored yellow. A non-critical pullback $Q_\ell$ is colored cyan. From the arc $I=(a,b)$ we obtain a new arc $J=(c,x)$.}
 \label{Fig_crit-pullback-iii}
\end{figure}

The arcs $J$ defined in above two cases are called the \emph{critical predecessors} of $I$ of type (i), (ii) and (iii) respectively. We call $I\to J$ a \textit{critical pullback}.

\medskip
If we consider the situation that $(a', c) \subset J$ in Case (1), or the situation that $v$ is on the left of $I$ in Case (2), then we label the components $Q_1$, $\cdots$, $Q_\ell$ of $G^{-1}(H_{\tilde{d}}(I))$ such that they locate around $c$ anticlockwise.

In Cases (1) and (2), $G^{-1}$ induces $\ell$ branches $\{\Phi^{(i)}:1\leqslant  i\leqslant  \ell\}$ of $G^{-1}$ near $v$ such that $\Phi^{(i)}\big(H_{\tilde{d}}(I)\big)=Q_i$.
We say that $\{\Phi^{(i)}:1\leqslant  i\leqslant  \ell-1\}$ are the inverse branches of $G$ \textit{associated} to the pullback $I\to J$.

\medskip
For the inverse branch $\Phi^{(1)}$, there exist a pair of critical points $c^-$, $c^+\in\T$ and a pair of critical values $v^-, v^+\in\MCV$ such that
\begin{itemize}
\item $J\subset(c^-, c^+)$. In particular, if $J=(c,x)$, then $c^-=c$ and $v^-=v$, and if $J=(x,c)$, then $c^+=c$ and $v^+=v$;
\item $(c^-, c^+)$ does not contain any critical point of $G$; and
\item $\Phi^{(1)}\big((v^-, v^+)\big)=(c^-, c^+)$.
\end{itemize}
We call $v^-$ and $v^+$ the two \textit{singular points} associated to $\Phi^{(1)}$.

\medskip
For every $\Phi^{(i)}$ with $2\leqslant  i\leqslant  \ell-1$, we call the critical value $v^\circ=v=G(c^\circ)\in (1+2\delta) I$ with $c^\circ =c $ the \textit{singular point} associated to $\Phi^{(i)}$.

\medskip
Let $\Omega$ be the domain defined in \eqref{equ:Omega}.
Let $J$ be a critical predecessor of $I$ and $Q_1$, $\cdots$, $Q_{\ell-1}$ be the components of $G^{-1}(H_{\tilde{d}}(I))$ of type (i), (ii) or (iii) obtained above.

\begin{lem}\label{lem:crit-pullback}
There exist $r_3=r_3(\delta,\eta,d)\in(0,r_2]$ and $K_0=K_0(\delta,\eta,d)> 0$ such that if $\sigma(I)<r_3$, then there exists a Jordan disk $B_i \subset \Omega$ with $\diam_\Omega(B_i) < K_0$ such that $Q_i\subset H_d(J) \cup B_i$ for every $1\leqslant  i\leqslant  \ell-1$.
\end{lem}

\begin{proof}
In a small neighborhood $W$ of $c$, the quasi-regular map $G$ can be written as $G=g\circ \varphi$, where $\varphi:W\to\varphi(W)$ is quasiconformal and $g:\varphi(W)\to G(W)$ is a proper holomorphic map having exactly one critical value $v$. Without loss of generality, we assume that
\begin{itemize}
\item $W$, $\varphi(W)$ and $G(W)$ are Jordan disks;
\item $W\cap\T$ is an arc and $\varphi(W\cap\T)\subset\T$; and
\item $H_{d+1}(2I)\subset G(W)$.
\end{itemize}

We only consider the situation that $Q_i$, where $1\leqslant  i\leqslant  \ell-1$, are obtained from Case (1) (i.e., type (i) or (ii)) since the proof for Case (2) (i.e., type (iii)) is completely similar.
For the above decomposition $G=g\circ\varphi$, let $\widetilde{c}$, $\widetilde{b}$, $\widetilde{x}\in\T$ and $\widetilde{Q}_i\subset\C\setminus\overline{\D}$, respectively, be the preimages of $v$, $b$, $G(x)$ and $H_{\tilde{d}}(I)$ under the holomorphic map $g$ such that $\widetilde{Q}_i=\varphi(Q_i)$. See Figure \ref{Fig_crit-pullback-bounded}.

\begin{figure}[!htpb]
 \setlength{\unitlength}{1mm}
 \centering
  \includegraphics[width=0.95\textwidth]{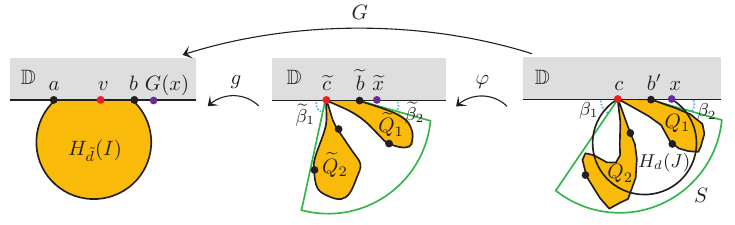}
  \caption{A sketch of the decomposition $G=g\circ\varphi$, where $g$ is holomorphic and $\varphi$ is quasiconformal. We show that every $Q_i$ is contained in the union of $H_d(J)$ and a Jordan disk $B_i$ whose hyperbolic diameter has a uniform upper bound.}
 \label{Fig_crit-pullback-bounded}
\end{figure}

By Lemma \ref{lem:comparable}, $|I|=|(a,b)|$ and $|(b,G(x))|$ are $C_0$-comparable, where $C_0=C_0(\delta,\eta)>0$ is a constant depending only on $\delta$ and $\eta$. Since $g$ is holomorphic, it follows that $\diam_{\C}(\widetilde{Q}_i)$ and $|(\widetilde{c},\widetilde{x})|$ are $C_1$-comparable for a constant $C_1=C_1(\delta,\eta,d)>0$.
For three different points $z_0,z_1,z_2\in\C\setminus\D$, let $\angle\, z_1 z_0 z_2$ be the angle measured in the logarithmic plane of $G$ (see \eqref{equ:angle}).
Then there exist $C_2=C_2(\delta,\eta,d)>0$ and two small numbers $\widetilde{\beta}_1=\widetilde{\beta}_1(\delta,\eta)>0$, $\widetilde{\beta}_2=\widetilde{\beta}_2(\delta,\eta,d)>0$ such that for every $1\leqslant  i\leqslant  \ell-1$, we have
\begin{equation}
\widetilde{Q}_i\subset
\left\{z\in\C\setminus\overline{\D}
\left|
\begin{array}{l}
\qquad\quad\dist_{\C}(z,\widetilde{c})<C_2|(\widetilde{c},\widetilde{x})|, \\
\angle\, z\widetilde{c}\widetilde{x}\in(0,\pi-\widetilde{\beta}_1) \text{ and }\angle\, z\widetilde{b}\widetilde{x}\in(\widetilde{\beta}_2,\pi)
\end{array}
\right.
\right\}.
\end{equation}
Since $H_{d+1}(2I)\subset G(W)$, without loss of generality, we assume that $\big\{z\in\C\setminus\overline{\D}:\dist_{\C}(z,\widetilde{c})\leqslant  C_2|(\widetilde{c},\widetilde{x})|\big\}\subset \varphi(W)$.

Since $\varphi:W\to\varphi(W)$ is quasiconformal, it follows that $\diam_{\C}(Q_i)$ and $|J|=|(c,x)|$ are $C_3$-comparable for a constant $C_3=C_3(\delta,\eta,d)>0$. Hence there exist $C_4=C_4(\delta,\eta,d)>0$ and two smaller numbers $\beta_1=\beta_1(\delta,\eta)>0$, $\beta_2=\beta_2(\delta,\eta,d)>0$ such that for every $1\leqslant  i\leqslant  \ell-1$, we have
\begin{equation}\label{equ:Q-i}
Q_i\subset S:=
\left\{z\in\C\setminus\overline{\D}
\left|
\begin{array}{l}
\qquad\quad\dist_{\C}(z,c)<C_4|(c,x)|, \\
\angle\, zcx\in(0,\pi-\beta_1) \text{ and }\angle\, zb'x\in(\beta_2,\pi)
\end{array}
\right.
\right\}.
\end{equation}

Let $\beta=\beta(d)\in(0,\tfrac{\pi}{2})$ be the angle formed by $\partial H_d(J)\setminus J$ and $\T\setminus J$ (see \eqref{equ:d-choice}). Note that $\beta_1=\beta_1(\delta,\eta)>0$ does not depend on $d>0$. Therefore, increasing $d>0$ and decreasing $\sigma(I)$ if necessary, we assume that $0<\beta<\beta_1/2$. In particular, there exists a small $r_3=r_3(\delta,\eta,d)\in(0,r_2]$ such that if $\sigma(I)<r_3$, then $H_{d+1}(2I)\subset G(W)$ is still satisfied and
\begin{equation}\label{equ:S-position}
S\subset\{z\in\C:1<|z|<(r_0-1)/2\},
\end{equation}
where $r_0>1$ is introduced in \eqref{equ:r-pri}.
Note that $|(b',x)|$ and $|J|$ are $C_5$-comparable, where $C_5=C_5(\delta,\eta)>0$.
It follows that for every $1\leqslant  i\leqslant  \ell-1$, $\overline{S\setminus(H_d(J)\cap Q_i)}$ is contained in a Jordan disk $B_i$ whose Euclidean diameter is $C_6$-comparable to its Euclidean distance to $\T$, where $C_6=C_6(\delta,\eta,d)>0$. Hence $Q_i\subset H_d(J) \cup B_i$ and by \eqref{equ:S-position}, we conclude that there exists $K_0=K_0(\delta,\eta,d)> 0$ such that $\diam_\Omega(B_i) < K_0$ for every $1\leqslant  i\leqslant  \ell-1$.
\end{proof}

\subsection{$(\delta,\eta)$-admissible sequence $\{I_n\}$}

Let us start with a small arc $I_0$ with $\sigma(I_0)< r_2/2$. Applying the rules in Sections \ref{subsec:non-crit-s} and \ref{subsec:crit-s} inductively, we obtain a sequence of arcs $\{I_n:0 \leqslant n <N'\}$   so that $I_{n+1}$ is a predecessor of $I_n$ for all $0\leqslant  n< N'-1$, where $N'=+\infty$ or is finite. The sequence $\{I_n:0 \leqslant n <N'\}$ is called a \textit{$(\delta,\eta)$-admissible sequence}.
A priori, there is a possibility that $\sigma(I_n)$ is large for some $n$, for instance, $I_n$ contains two critical values, and thus $I_{n+1}$ may not be defined and the number $N'$ is finite.

In the following we prove that, by choosing $0 < \delta \ll \eta<\frac{1}{2}$ appropriately, the predecessors can be defined infinitely many times provided that $I_0$ is small enough.
The basic idea behind the proof is that the critical predecessors of types (i) and (iii) appear more often than the critical predecessors of type (ii).
Let
\begin{equation}\label{equ:crit-seq}
0 \leqslant n_1 < n_2 < \cdots<n_j<\cdots
\end{equation}
be the sequence of all the integers so that $I_{n_j+1}$ is a critical predecessor of $I_{n_j}$. As we mentioned above, the sequence \eqref{equ:crit-seq} may be finite.
Let $t_0\geqslant 1$ be the number of critical points (without counting multiplicity) of $G$ on $\T$ (see \eqref{equ:crit-t0}).

\begin{lem}\label{lemma:loss}
If $0 < \delta< \eta < \frac{1}{2}$ are small enough, then for any $(\delta,\eta)$-admissible sequence $\{I_n:0\leqslant  n<N'\}$  containing $I_{n_{j+2t_0}+1}$ for some $j\geqslant 1$,
there exists $0 \leqslant i \leqslant 2t_0$, such that $I_{n_{j+i}+1}$ is a critical predecessor of $I_{n_{j+i}}$ of type \textup{(i)} or \textup{(iii)}.
\end{lem}

\begin{proof}
We prove this lemma by contradiction. Assume that $I_{n_{j+i}+1}$ is a type (ii) critical predecessor of $I_{n_{j+i}}$ containing the critical value $v_i\in\MCV$ for all $0 \leqslant i \leqslant 2t_0$.  Then from  $0 < \delta< \eta < \frac{1}{2}$  we have $G^{n_{j+i}-n_{j+i-1}}(I_{n_{j+i}}) \subset (1+4\eta)  I_{n_{j+i-1}}$ for every $1\leqslant  i\leqslant  2t_0$. Thus $G^{n_{j+i}-n_{j}}(I_{n_{j+i}}) \subset (1+4\eta)^{2t_0}  I_{n_{j}}$ for $0\leqslant  i\leqslant  2t_0$ and
\begin{equation}\label{equ:he}
G^{n_{j+i}-n_{j}}(v_i) \in (1+4\eta)^{2 t_0}  I_{n_{j}}.
\end{equation}

Let $\lambda_0>1$ be the constant guaranteed  by Lemma \ref{lemma:fr} (In fact, $\lambda_0>2$).
We claim that one can take  $0 < \delta <\eta < \frac{1}{2}$  small enough  such that
\begin{equation}\label{equ:we}
(1+4\eta)^{2t_0}<1+\frac{1}{2\lambda_0}
\end{equation}
and
\begin{equation}\label{equ:well-in-interval}
G^{k}(v)\not\in\Big(1-\frac{1}{2\lambda_0}\Big)I_{n_j}
\end{equation}
for any $0\leqslant  k\leqslant  n_{j+2t_0}-n_j$ and any $v \in   \MCV$.   In fact, we get \eqref{equ:well-in-interval} by contradiction. For otherwise, from $G^{k}(v)\in\big(1-\frac{1}{2\lambda_0}\big)I_{n_j} $ and that $0 < \delta < \eta < 1/2$ are small, it follows that $v \in (1 - 2 \eta) I_{n_j+k}$. But this implies that $I_{n_j+k+1}$ is a critical predecessor of $I_{n_j+k}$ of type (i). This contradicts the assumption and the claim holds.

\medskip
Note that $G$ has exactly $t_0$ critical points on $\T$. By \eqref{equ:he}, there exist a critical value $v \in\MCV$  and  three integers $j\leqslant j_0< j_1< j_2\leqslant j+ 2t_0$ such that $v=v_{j_l}$ for every $0 \leqslant l \leqslant 2$ and
\begin{equation}
G^{n_{j_l} - n_j}(v) \in (1 + 4\eta)^{2t_0}I_{n_j}.
\end{equation}
Denote $k_0:=n_{j_0} - n_j$. There exist two integers $k_1$, $k_2\in [k_0+1,n_{j_2}-n_j]$ such that $G^{k_1}(v)$ and $G^{k_2}(v)$ are the first and second returns of $G^{k_0}(v)$ to $(1 + 4\eta)^{2t_0}I_{n_j}$ respectively.
By Lemma \ref{lemma:fr} and \eqref{equ:we}, there exists an $0 \leqslant l \leqslant 2$   such that $G^{k_l}(v) \in\kappa I_{n_j}$, where
\begin{equation}
\kappa\leqslant \Big(1-\frac{2}{\lambda_0}\Big)(1+4\eta)^{2t_0}< \Big(1-\frac{2}{\lambda_0}\Big)\Big(1+\frac{1}{2\lambda_0}\Big)<1-\frac{1}{2\lambda_0}.
\end{equation}
This contradicts \eqref{equ:well-in-interval}. The proof is complete.
\end{proof}

\begin{lem}\label{lemma:shorten}
There exist $0 < \delta \ll \eta < \frac{1}{2}$ such that if $\sigma(I_0)< r_2/4$, then every $(\delta,\eta)$-admissible sequence $\{I_n:0\leqslant  n<N'\}$ is infinite, i.e., $N'=+\infty$. In particular,
\begin{enumerate}
\item There exists $T_1=T_1(\eta)>0$ such that $\sigma(I_{t}) < \frac{1}{2}\: \sigma(I_{s})$ for any $0 \leqslant s < t<\infty$ if $\{I_n: s\leqslant  n\leqslant  t\}$ contains at least $T_1$ critical predecessors; and
\item $\sigma(I_t)< 2\sigma(I_s)$ for any $0\leqslant  s<t<\infty$ and $\sigma(I_n) \to 0$ as $n \to \infty$.
\end{enumerate}
\end{lem}

\begin{proof}
Let $0<\delta<\eta<\frac{1}{2}$ be small enough such that Lemma \ref{lemma:loss} holds. Moreover, for the chosen small $\eta>0$, we let $\delta\in(0,\eta)$ be sufficiently small such that
\begin{equation}
(1+\delta)^{2t_0}<\tfrac{5}{4} \text{\quad and\quad} (1+\delta)^{2t_0}(1-\tfrac{\eta}{3})<1-\tfrac{\eta}{4}.
\end{equation}
For any $I_{n_j}$ with $\sigma(I_{n_j})<r_2/4$ for some $j\geqslant 1$, by Lemma \ref{lemma:loss}, we have
\begin{itemize}
\item $\sigma(I_k)=\sigma(I_{n_j+1})\leqslant (1+\delta)\sigma(I_{n_j})$, for each $n_j+1\leqslant k\leqslant n_{j+1}$;
\item $\sigma(I_k)\leqslant   (1+\delta)^{2t_0}\sigma(I_{n_j})<\tfrac{5}{4}\sigma(I_{n_j})$, for each $n_j+1\leqslant  k\leqslant  n_{j+2t_0}$; and
\item $\sigma(I_k)= \sigma(I_{n_{j+2t_0}+1})\leqslant   (1+\delta)^{2t_0}(1-\tfrac{\eta}{3})\sigma(I_{n_j})<(1-\tfrac{\eta}{4})\sigma(I_{n_j})$, for each $n_{j+2t_0}+1\leqslant  k\leqslant  n_{j+2t_0+1}$.
\end{itemize}
Therefore, we have
\begin{equation}\label{equ:sigma-I-k}
\begin{split}
\sigma(I_k)&<\tfrac{5}{4}\sigma(I_{n_j}) \text{ for } n_j\leqslant  k\leqslant  n_{j+2t_0+1} \text{\quad and}\\
\sigma(I_{n_{j+2t_0+1}})&<(1-\tfrac{\eta}{4})\sigma(I_{n_j}).
\end{split}
\end{equation}
This implies that every $(\delta,\eta)$-admissible sequence $\{I_n:0\leqslant  n<N'\}$ is infinite.

\medskip
Note that $\sigma(I_k)=\sigma(I_0)$ for each $0\leqslant k\leqslant n_1$. By \eqref{equ:sigma-I-k}, we have $\sigma(I_t)< \frac{5}{4}\sigma(I_s)<2\sigma(I_s)$ for any $0\leqslant  s<t<\infty$.
Let $k_0=k_0(\eta)\geqslant 1$ be the minimal integer which is larger than $\log 4/\log\frac{4}{4-\eta}$. Then $(1-\tfrac{\eta}{4})^{k_0}<\frac{1}{4}$.
Note that \eqref{equ:crit-seq} is an infinite sequence since $G|_{\T}$ is conjugate to the irrational rotation $R_\alpha(\zeta)=e^{2\pi\ii\alpha}\zeta$.
Set
\begin{equation}
T_1:=(2t_0+1)k_0.
\end{equation}
Still by \eqref{equ:sigma-I-k}, if $\{I_n: s\leqslant  n\leqslant  t\}$ contains at least $T_1$ critical predecessors, then
\begin{equation}
\sigma(I_{t}) < \tfrac{5}{4}(1-\tfrac{\eta}{4})^{k_0} \sigma(I_{s})<\tfrac{1}{2} \sigma(I_{s}).
\end{equation}
Hence $\sigma(I_n) \to 0$ as $n \to \infty$. The proof is complete.
\end{proof}

From now on, when we talk about a $(\delta,\eta)$-admissible sequence $\{I_n\}$, we always assume that
$\sigma(I_0)<r_2/4$
and that $0 < \delta \ll \eta < \frac{1}{2}$ are taken appropriately so that Lemma \ref{lemma:shorten} holds.

\medskip
For a $(\delta,\eta)$-admissible sequence $\{I_n\}_{n\geqslant 0}$ and $0\leqslant k_1\leqslant k_2$, we denote
\begin{equation}
\Theta(\{I_n\},k_1,k_2):=\sharp\,\big\{ \text{critical predecessors in } \{I_n:k_1\leqslant  n\leqslant  k_2\}\big\}.
\end{equation}
For $j\geqslant 1$, let $n_j$ be the critical (value) position of $\{I_n\}_{n\geqslant 0}$ defined in \eqref{equ:crit-seq}.

\begin{lem}\label{lem:crit-pull-enhance}
Let $\{I_n\}_{n\geqslant 0}$ and $\{J_n\}_{n\geqslant 0}$ be two $(\delta,\eta)$-admissible sequences with $I_n\subset J_n$ for all $n\geqslant 0$. Suppose $\sigma(J_0)\leqslant  C_1\sigma(I_0)$ for some $C_1=C_1(\delta,\eta)\geqslant 1$. Then for every $j\geqslant 1$, there exists $C_1'=C_1'(\delta,\eta,j)\geqslant 1$ such that
\begin{equation}
\Theta(\{J_n\},0, n_j)\leqslant  C_1'.
\end{equation}
\end{lem}

\begin{proof}
By the definition of $n_j$, we have $\sigma(I_{n_1})=\sigma(I_0)$. Let $i_0$ be the minimal integer such that $i_0\geqslant \frac{\log C_1}{\log 2}+1$. We claim that
\begin{equation}\label{equ:a-1}
\Theta(\{J_n\},0,n_1)\leqslant  i_0T_1.
\end{equation}
Indeed, otherwise by Lemma \ref{lemma:shorten} we have
\begin{equation}
\sigma(J_{n_1})<2\,\Big(\frac{1}{2}\Big)^{i_0}\sigma(J_0)\leqslant  \frac{1}{C_1}\sigma(J_0)\leqslant  \sigma(I_0)=\sigma(I_{n_1}).
\end{equation}
This is a contradiction since $I_{n_1}\subset J_{n_1}$. Hence \eqref{equ:a-1} holds.
By Lemma \ref{lemma:shorten}(b), we have $\sigma(J_{n_1+1})< 2\sigma(J_0)\leqslant  2C_1\sigma(I_0)$.
By the definition of critical predecessors, we have $\sigma(I_{n_1+1})\geqslant \frac{1}{2}\sigma(I_{n_1})$.
Hence
\begin{equation}\label{equ:b-1}
\sigma(J_{n_1+1})< 2C_1\sigma(I_0)=2C_1\sigma(I_{n_1})\leqslant  C_2\sigma(I_{n_1+1}), \text{\quad where }C_2:=4C_1.
\end{equation}

Note that $\sigma(I_{n_1+1})=\sigma(I_{n_2})\leqslant  2\sigma(I_{n_2+1})$. Let $i_1$ be the minimal integer such that $i_1\geqslant \frac{\log C_2}{\log 2}+1$. Then by Lemma \ref{lemma:shorten} and \eqref{equ:b-1} we have
\begin{equation}
\Theta(\{J_n\},n_1+1,n_2)\leqslant  i_1T_1
\end{equation}
and
\begin{equation}
\sigma(J_{n_2+1})< 2\sigma(J_{n_1+1})\leqslant  2C_2\sigma(I_{n_1+1})\leqslant  C_3\sigma(I_{n_2+1}),
\end{equation}
where $C_3:=4C_2=4^2 C_1$.

\medskip
Inductively, for every $l\geqslant 2$, we have
\begin{equation}\label{equ:a-l}
\Theta(\{J_n\},n_{l-1}+1,n_l)\leqslant  i_{l-1}T_1,
\end{equation}
where $i_{l-1}$ is the minimal integer such that
\begin{equation}
i_{l-1}\geqslant \frac{\log C_l}{\log 2}+1= \frac{\log (4^{l-1}C_1)}{\log 2}+1=2l-1+\frac{\log C_1}{\log 2}.
\end{equation}
Therefore, we have
\begin{equation}\label{equ:i-l-1}
i_{l-1}\leqslant  2l+\frac{\log C_1}{\log 2}.
\end{equation}
Denote $n_0:=-1$. Combining \eqref{equ:a-1}, \eqref{equ:a-l} and \eqref{equ:i-l-1}, we have
\begin{equation}
\Theta(\{J_n\},0,n_j)= \sum_{l=1}^j \Theta(\{J_n\},n_{l-1}+1,n_l)\leqslant  T_1\sum_{l=1}^j i_{l-1}
\leqslant  C_1',
\end{equation}
where $C_1'=\big(\frac{\log C_1}{\log 2}+j+1\big)jT_1$.
The proof is complete.
\end{proof}

\begin{defi}[Deviation]
Let $\{I_n\}_{n\geqslant 0}$ be a $(\delta,\eta)$-admissible sequence. For two integers $k\geqslant 1$ and $n\geqslant 0$, the \textit{deviation} from $I_{n+k}$ to $I_n$ is defined as the dynamical length of the shorter arc in $\T$ between $G^k(I_{n+k})$ and $I_n$.
\end{defi}

By definition, if $G^k(I_{n+k})$ intersects $I_n$, then the deviation from $I_{n+k}$ to $I_n$ is zero. In particular, if $I_{j+1}$ is a non-critical predecessor of $I_j$ for all $n \leqslant j \leqslant n+k-1$, then the deviation from $I_{n+k}$ to $I_n$ is zero.
Note that the deviation is caused by critical predecessors.
Since $\sigma(I_{n+k})< 2\sigma(I_n)$ for all $k\geqslant 1$, as a consequence of Lemma \ref{lemma:shorten}, we have the following immediate corollary.

\begin{cor}\label{cor:deviation}
Let $\{I_n\}_{n\geqslant 0}$ be a $(\delta,\eta)$-admissible sequence. Then for every $l \geqslant 1$, there is a number $K(\delta,\eta, l) > 1$ depending only on $\delta$, $\eta$ and $l$, such that for any $n \geqslant 0$, the deviation from $I_{n+k}$ to $I_n$ is at most $K(\delta,\eta, l)\sigma(I_n)$ provided the number of the critical predecessors  between $n$ and $n+k$ is not more than $l$.
\end{cor}

\subsection{Contraction regions associated to  $\{I_n\}$}

Consider a $(\delta,\eta)$-admissible sequence $\{I_n\}_{n\geqslant 0}$, where $0<\delta\ll\eta<\frac{1}{2}$ are two fixed small numbers.
For $n \geqslant 0$, let $\Phi_n$ be the inverse branch of $G$ associated to the pullback $I_n\to I_{n+1}$.
If $I_n\to I_{n+1}$ is a non-critical pullback (see Section \ref{subsec:non-crit-s}) or a critical pullback with $\Phi_n=\Phi_n^{(1)}$ (see Section \ref{subsec:crit-s}), we use $v_n^-=G(c_n^-)$ and $v_n^+=G(c_n^+)$ to denote the singular points associated to $\Phi_n$, where $c_n^-,c_n^+\in\Crit(G)\cap\T$. Otherwise, let $v_n^\circ=G(c_n^\circ)$ be the singular point associated to $\Phi_n$, where $c_n^\circ\in\Crit(G)\cap\T$.

\begin{defi}[The rays $L_n^\pm$ associated to $v_n^\pm$]
Let $v_n^\pm = G(c_n^\pm)$ be the singular points associated to $\Phi_n$, and let $L_n^\pm$ be the rays (in the logarithmic plane of $G$) starting from $v_n^\pm$ which form an angle $0<\beta<\pi/3$ with $\T$ (clockwise for $v_n^+$ and anticlockwise for $v_n^-$), where $\beta$ will be specified later.
By Lemma \ref{lemma:Pe}, $\Phi_n$ contracts the hyperbolic metric $\rho_\Omega(z)|dz|$ strictly in $V_n^\pm\subset\Omega$, where $V_n^\pm$ consists of the points which are below $L_n^\pm$ and in a small neighborhood of $v_n^\pm$ (see Figure \ref{Fig_G-pull-back}). In the following we refer $L_n^\pm$ \emph{the rays associated to $v_n^\pm$}.
\end{defi}

\begin{figure}[!htpb]
  \setlength{\unitlength}{1mm}
  \centering
  \includegraphics[width=0.95\textwidth]{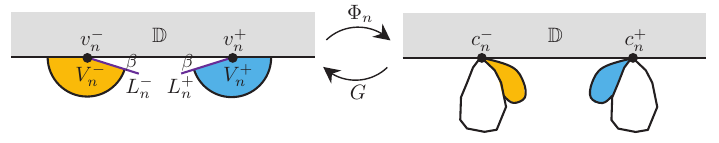}
  \caption{The rays $L_n^{\pm}$ associated to the singular points $v_n^{\pm}$ of the branch $\Phi_n$ of $G^{-1}$. The map $\Phi_n$ contracts the hyperbolic metric $\rho_\Omega(z)|dz|$ in two domains $V_n^\pm$ colored dark cyan and yellow respectively.}
  \label{Fig_G-pull-back}
\end{figure}

By Lemma \ref{lemma:Pe}, if the singular point associated to $\Phi_n$ is $v_n^\circ$, then $\Phi_n$ contracts the hyperbolic metric $\rho_\Omega(z)|dz|$ strictly in $V_n^\circ\subset\Omega$, where $V_n^\circ$ consists of the points which are in a small neighborhood of $v_n^\circ$ and are outside of $\overline{\D}$.

\medskip
By Lemma \ref{lemma:shorten}, $\sigma(I_n)< 2\sigma(I_0)$ for all $n\geqslant 0$ and $\sigma(I_n) \to 0$ as $n \to \infty$.
Denote $s_0:= r_1/8$, where $r_1=r_1(d)>0$ is introduced in Lemma \ref{lemma:swz}.
For $n\geqslant 0$, we define
\begin{equation}\label{equ:D-n-j}
D_n^{s_0} := \big\{z\in\C\setminus\overline{\D}:\,\dist_{\C}(z, I_n) < s_0 \big\}.
\end{equation}
Then $D_n^{s_0}\cap \MP(G)=\emptyset$ for all $n\geqslant 0$. Note that $D_n^{s_0}$ is simply connected and  $\Phi_n$ is defined in $D_n^{s_0}$ by analytic continuation of $\Phi_n|_{I_n}$. Let $r_3=r_3(\delta,\eta,d)>0$ be introduced in Lemma \ref{lem:crit-pullback}. Without loss of generality, we assume that $r_3$ is small such that if $\sigma(I_0)<r_3$, then $H_{d+1}(I_n)\subset D_n^{s_0}$ for all $n\geqslant 0$.

\begin{lem}\label{lem:observation}
For any integer $T'\geqslant 1$, there exists $r_4=r_4(\delta,\eta,d,T')\in(0,r_3]$ such that if $\sigma(I_0)<r_4$, then for any $n\geqslant 0$ and $m\geqslant n$, if $\{I_k: n\leqslant  k\leqslant  m+1\}$ contains at most $T'$ critical predecessors, then
\begin{equation}\label{equ:Phi-n-composition}
\Phi_{m}\circ \cdots\circ\Phi_{n+1}\circ\Phi_n\big(H_d(I_n)\big) \subset D_{m+1}^{s_0}.
\end{equation}
\end{lem}


\begin{proof}
Let $n_j$ be the smallest integer $\geqslant n$ such that $I_{n_j}\to I_{n_j+1}$ is a critical pullback. If $\sigma(I_0)$ is small, then for every $n\leqslant  k\leqslant  n_j-1$, by Lemma \ref{lemma:swz-2}(b), $\Phi_{k}\circ \cdots\circ\Phi_{n+1}\circ\Phi_n\big(H_d(I_n)\big)$ is defined and contained in $H_{d'}(I_{k+1})$ with $d'\in[d,d+1]$, where $H_{d'}(I_{k+1})\subset D_{k+1}^{s_0}$.
By Lemma \ref{lem:crit-pullback}, $\Phi_{n_j}\circ \cdots\circ\Phi_{n+1}\circ\Phi_n\big(H_d(I_n)\big)$ is defined and contained in $H_d(I_{n_j+1})\cup B_{n_j+1}^1$, where $B_{n_j+1}^1$ is a Jordan disk in $\Omega$ satisfying
\begin{equation}
B_{n_j+1}^1\cap H_d(I_{n_j+1})\neq\emptyset \text{\quad and\quad} \diam_{\Omega}(B_{n_j+1}^1)< K_0=K_0(\delta,\eta,d).
\end{equation}
If $\sigma(I_0)$ is small enough, then $H_d(I_{n_j+1})\cup B_{n_j+1}^1\subset D_{n_j+1}^{s_0}$.

Since $D_{n_j+1}^{s_0}$ is simply connected and $D_{n_j+1}^{s_0}\cap \MP(G)=\emptyset$, $\Phi_{n_{j+1}}\circ \cdots\circ\Phi_{n_j+1}\big(D_{n_j+1}^{s_0}\big)$ is well defined.
By Lemmas \ref{lemma:swz-2}(b) and \ref{lem:crit-pullback}, $\Phi_{n_{j+1}}\circ \cdots\circ\Phi_{n_j+1}\big(H_d(I_{n_j+1})\big)$ is contained in $H_d(I_{n_{j+1}+1})\cup B_{n_{j+1}+1}^1$, where $B_{n_{j+1}+1}^1$ is a Jordan disk in $\Omega$ satisfying
\begin{equation}
B_{n_{j+1}+1}^1\cap H_d(I_{n_{j+1}+1})\neq\emptyset \text{\quad and\quad} \diam_{\Omega}(B_{n_{j+1}+1}^1)< K_0.
\end{equation}
Moreover, by the Schwarz-Pick Lemma, $B_{n_{j+1}+1}^{(1)}:=\Phi_{n_{j+1}}\circ \cdots\circ\Phi_{n_j+1}\big(B_{n_j+1}^1\big)$ satisfies
\begin{equation}
B_{n_{j+1}+1}^{(1)}\cap \big(H_d(I_{n_{j+1}+1})\cup B_{n_{j+1}+1}^1\big)\neq\emptyset \text{\quad and\quad} \diam_{\Omega}(B_{n_{j+1}+1}^{(1)})< K_0.
\end{equation}
If $B_{n_{j+1}+1}^{(1)}\cap H_d(I_{n_{j+1}+1})\neq\emptyset$, we denote $B_{n_{j+1}+1}^2:=B_{n_{j+1}+1}^{(1)}$. Otherwise, $B_{n_{j+1}+1}^{(1)}\cap \big(B_{n_{j+1}+1}^1 \setminus H_d(I_{n_{j+1}+1})\big)\neq\emptyset$. In this case, let $B_{n_{j+1}+1}^2$ be the interior of the union of $\overline{B_{n_{j+1}+1}^1\cup B_{n_{j+1}+1}^{(1)}}$ and all the bounded components of $\C\setminus\big(\overline{B_{n_{j+1}+1}^1\cup B_{n_{j+1}+1}^{(1)}}\big)$. Then $B_{n_{j+1}+1}^2$ is a Jordan domain in $\Omega$ (by assuming $\sigma(I_0)$ is small) and
\begin{equation}
B_{n_{j+1}+1}^{2}\cap H_d(I_{n_{j+1}+1}) \neq\emptyset \text{\quad and\quad} \diam_{\Omega}(B_{n_{j+1}+1}^2)< 2K_0.
\end{equation}
Thus $\Phi_{n_{j+1}}\circ \cdots\circ\Phi_{n+1}\circ\Phi_n\big(H_d(I_n)\big)$ is defined and contained in $H_d(I_{n_{j+1}+1})\cup B_{n_{j+1}+1}^1\cup B_{n_{j+1}+1}^2$, where each $B_{n_{j+1}+1}^k$ with $k=1,2$ is a Jordan disk in $\Omega$ satisfying
\begin{equation}
B_{n_{j+1}+1}^k\cap H_d(I_{n_{j+1}+1})\neq\emptyset \text{\quad and\quad} \diam_{\Omega}(B_{n_{j+1}+1}^k)< 2K_0.
\end{equation}
Hence if $\sigma(I_0)$ is small enough, then $H_d(I_{n_{j+1}+1})\cup B_{n_{j+1}+1}^1\cup B_{n_{j+1}+1}^2\subset D_{n_{j+1}+1}^{s_0}$.

Let $T'\geqslant 1$ be given. Repeating the above process, by Lemmas \ref{lemma:swz-2}(b) and \ref{lem:crit-pullback}, there exists a small $r_4=r_4(\delta,\eta,d,T')\in(0,r_3]$ such that if $\sigma(I_0)<r_4$, then for every $n\leqslant  \ell\leqslant  m$, there exists $1\leqslant  p_\ell\leqslant  T'$ such that
\begin{equation}\label{equ:W-p}
\Phi_{\ell}\circ \cdots\circ\Phi_{n+1}\circ\Phi_n\big(H_d(I_n)\big)\subset H_{\tilde{d}}(I_{\ell+1})\cup \bigcup_{k=1}^{p_\ell} B_{\ell+1}^k \subset D_{\ell+1}^{s_0},
\end{equation}
where each $B_{\ell+1}^k$ with $1\leqslant  k\leqslant  p_\ell$ is a Jordan disk in $\Omega$ satisfying
\begin{equation}\label{equ:B-k-m}
B_{\ell+1}^k\cap H_{\tilde{d}}(I_{\ell+1})\neq\emptyset \text{\quad and\quad}
\diam_{\Omega}(B_{\ell+1}^k)< p_\ell K_0\leqslant   T' K_0,
\end{equation}
and moreover, $\tilde{d}=d$ if $I_{\ell}\to I_{\ell+1}$ is a critical pullback and $\tilde{d}=d'\in[d,d+1]$ if $I_{\ell}\to I_{\ell+1}$ is a non-critical pullback.
The proof is complete.
\end{proof}

Denote
\begin{equation}\label{equ:T-0}
T_0:=t_0+1, \text{\quad where\quad} t_0=\sharp\,(\Crit(G)\cap\T)=\sharp\,\MCV\geqslant 1.
\end{equation}
Let $n_j$ be defined in \eqref{equ:crit-seq}.
By Lemma \ref{lem:observation}, $\Phi_{n_{j+T_0}} \circ \cdots \circ \Phi_{n_j+1}\circ \Phi_{n_j}$ is defined in $D_{n_j}^{s_0}$ if $\sigma(I_0)<r_4:=r_4(\delta,\eta,d,T_0)=r_4(\delta,\eta,d)$.
The following result is the most important ingredient in the proof of the Main Lemma.

\begin{prop}\label{prop:three special regions}
There exist $d>1$, $r_5=r_5(\delta,\eta,d)\in(0,r_4]$, $C_0=C_0(\delta,\eta,d)>1$ and $0<\mu_0=\mu_0(d)<1$, such that if $\sigma(I_0)<r_5$, then for every $j\geqslant 1$, there exists an arc $J \subset \T$ containing $I_{n_{j+T_0}+1}$ with $\sigma(J) < C_0 \sigma(I_{n_{j+T_0}+1})$ and a triangle $Z\subset D_{n_j}^{s_0}$ such that the preimage $w=\Phi_{n_{j+T_0}} \circ \cdots \circ \Phi_{n_j+1}\circ \Phi_{n_j}(z)$ with $z \in D_{n_j}^{s_0}$ satisfies the following:
\begin{enumerate}
\item $w\in H_d(J)$ for any $z\in Z$; and
\item $\rho_\Omega(w) < \mu_0\,\rho_\Omega(z)|(G^{n_{j+T_0}-n_j+1})'(w)| $ for any $z\in D_{n_j}^{s_0}\setminus Z$.
\end{enumerate}
\end{prop}

\begin{proof}
By Corollary \ref{cor:deviation}, there exists a number $K_1=K_1(\delta,\eta,T_0)>1$ depending only on $\delta$, $\eta$ and $T_0$ such that for any $n_j\leqslant  k\leqslant  n_{j+T_0}+2$, if $x\in(1+2\delta) I_k$, then
\begin{equation}\label{equ:position-G}
G^{k-n_j}(x)\in (2K_1+1) I_{n_j}.
\end{equation}
For $j\leqslant i\leqslant j+T_0$, let $v_{n_i} \in \MCV$  be the unique critical value such that $v_{n_i} \in (1+2\delta)I _{n_i}$.
Then $G^{n_i-n_j}(v_{n_i})\in (2K_1+1) I_{n_j}$ for every $j \leqslant i \leqslant j+T_0$.

Without loss of generality, we assume that $v_{n_j}=v_{n_j}^-$. The proof for $v_{n_j}=v_{n_j}^+$ is completely the same and for $v_{n_j}=v_{n_j}^\circ$, Part (b) can be obtained by applying Lemma \ref{lemma:Pe} directly and the proof for Part (a) is the same to $v_{n_j}=v_{n_j}^-$ (see the following Case 1).
Let $L=L_{n_j}^-$ be the ray associated to $v:=v_{n_j}^-$ which forms an angle $0<\beta<\pi/3$ with $\T$.
Since $\sigma(I_{n_j})< 2\sigma(I_0)$, decreasing $\sigma(I_0)$ if necessary, there exists $l\geqslant 5$ such that
the closest return $G^{q_l}(v)$ is on the same side as $L$, i.e., the angle formed by $L$ and $[v,G^{q_l}(v)]$ is $\beta$, and moreover,
\begin{equation}\label{equ:choice-q-l}
\sigma\big((v,G^{q_{l+2}}(v))\big) \leqslant (K_1+2)\sigma(I_{n_j})< \sigma\big((v,G^{q_l}(v))\big).
\end{equation}

Since $s_0=r_1/8$ and $\sigma(I_{n_j})< 2\sigma(I_0)$, decreasing $\sigma(I_0)$ if necessary, without loss of generality, we assume that $[G^{q_{l-1}}(v),G^{q_{l-4}}(v)]$ is contained in the interior of $\partial D_{n_j}^{s_0}\cap\T=[x_0',x_0]$ and $\sigma\big([x_0',x_0]\big)<r_1$.
Let $X$ be the subregion of $D_{n_j}^{s_0}$ which is below the ray $L$.
Let $Y$ be the closed trapezoid  in $\C\setminus\overline{\D}$ which is bounded by the ray $L$, the interval $[G^{ q_{l-4}}(v),x_0]$, and the two vertical straight segments passing through $G^{q_{l-4}}(v)$ and $x_0$, respectively.
Let $Z$ denote the closed right triangle in $\C\setminus\overline{\D}$ bounded by  the interval $[v, G^{q_{l-4}}(v)]$, the ray $L$, and the vertical segment passing through $G^{q_{l-4}}(v)$.
See Figure \ref{Fig_neigh-XYZ}.

\begin{figure}[!htpb]
  \setlength{\unitlength}{1mm}
  \centering
  \includegraphics[width=0.92\textwidth]{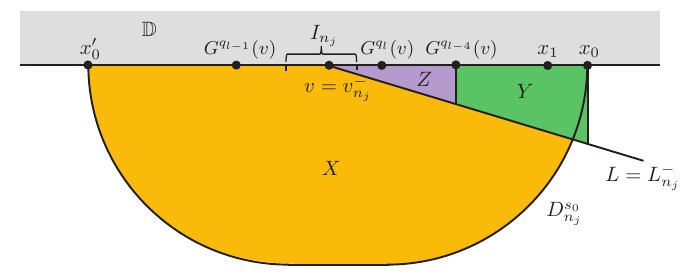}
  \caption{The interval $I_{n_j}$ and its half neighborhood $D_{n_j}^{s_0}$ with a partition.}
  \label{Fig_neigh-XYZ}
\end{figure}

\medskip
\textbf{Case 1}:
Suppose $z \in Z$. Denote $I_{n_j}=(a,b)$. If $a\in (v,b)$, we denote $J_0: = (v, G^{q_{l-4}}(v))$. Otherwise, we denote $J_0: = (a, G^{q_{l-4}}(v))$. Then we have $I_{n_j}\subset J_0$. By \eqref{equ:choice-q-l}, there exists a constant $C_1=C_1(\delta,\eta,T_0)=C_1(\delta,\eta)\geqslant 1$ such that
\begin{equation}\label{equ:J-0-compare}
\sigma(J_0)\leqslant  C_1\sigma(I_{n_j}).
\end{equation}
Without loss of generality, we assume that the angle $\beta\in(0,\tfrac{\pi}{3})$ between the ray $L$ and $\T$ is small enough such that the half hyperbolic neighborhood $H_d(J_0)$ is simply connected and $Z \subset H_d(J_0)$. Let $N := n_{j+T_0} - n_j+1$. For $k = 0,1, \cdots, N$, we consider the $(\delta,\eta)$-admissible sequence of pullbacks of $J_0$ (which is well-defined if $\sigma(I_0)$ is small) such that the pullback $J_k \to J_{k+1}$ is induced by $I_{n_j+ k} \to I_{n_j+k+1}$ and satisfies $I_{n_j+k} \subset J_k$ for all $1\leqslant  k\leqslant  N$.

By Lemma \ref{lem:crit-pull-enhance} and \eqref{equ:J-0-compare}, the number of critical predecessors in the sequence $\{J_k: 0 \leqslant k \leqslant N\}$ is bounded above by a constant $C_2=C_2(\delta,\eta,T_0)=C_2(\delta,\eta)\geqslant 1$.
By a similar arguments to \eqref{equ:W-p} and \eqref{equ:B-k-m}, there exists $1\leqslant  p\leqslant  C_2$ such that for any $z\in Z$ we have
\begin{equation}
w=\Phi_{n_{j+T_0}} \circ \cdots \circ \Phi_{n_j+1}\circ \Phi_{n_j}(z)\in W_p:=H_d(J_{N})\cup \bigcup_{k=1}^p B_{N}^k,
\end{equation}
where each $B_{N}^k$ with $1\leqslant  k\leqslant  p$ is a Jordan disk in $\Omega$ satisfying
\begin{equation}
B_{N}^k\cap H_d(J_{N})\neq\emptyset \text{\quad and\quad}
\diam_{\Omega}(B_{N}^k)< C_2 K_0.
\end{equation}
This implies that if $\sigma(I_0)$ is small enough, then there is a constant $C_3=C_3(\delta,\eta,d)>1$ such that $w\in W_p\subset H_d(J)$ with $J:=C_3 J_{N}$. Hence
\begin{equation}\label{equ:J-1}
\sigma(J)=C_3\sigma(J_{N})<2C_3\sigma(J_0).
\end{equation}
Since $I_{n_j+N} \subset J_{N}$, we have $I_{n_{j+T_0}+1}\subset J$.

By \eqref{equ:choice-q-l}, there is a number $C_4=C_4(\delta,\eta,d)>1$ such that $\sigma(J_0)< C_4\sigma(I_{n_j})$.
By the definitions of critical and non-critical predecessors, $\sigma(I_{n_j})\leqslant  2^{T_0+1}\sigma(I_{n_{j+T_0}+1})$.
Therefore, there is a constant $C_0=C_0(\delta,\eta,d)>1$ such that $\sigma(J)<C_0\sigma(I_{n_{j+T_0}+1})$.
This implies that Part (a) holds.

\medskip
\textbf{Case 2}: Suppose $z \in X$.  Then by Lemma \ref{lemma:Pe}, the map $\Phi_{n_j}$ contracts the hyperbolic metric at $z$ by a definite amount. In particular, there is some $\mu_1=\mu_1(\beta)\in (0,1)$ which depends only on the angle $\beta$ formed by $L$ and $\T$ such that
\begin{equation}\label{equ:mu-1}
\rho_\Omega(\Phi_{n_j}(z)) < \mu_1\,\rho_\Omega(z)|G'(\Phi_{n_j}(z))|.
\end{equation}
Then by the Schwarz-Pick Lemma, Part (b) holds for all $z\in X$ with $\mu_0:=\mu_1$.

\medskip
\textbf{Case 3}: Suppose $z \in Y$. We shall use finitely many half hyperbolic neighborhoods (see \eqref{equ:Y-cover}) to cover $Y$ and prove that for each $z$ in these half hyperbolic neighborhoods, there exists one integer $k\in[n_j,n_{j+T_0}]$ such that $\Phi_k$ contracts the hyperbolic metric of $\Omega$ at $z$.

\medskip
\textit{Step 1. Some orbit point of a singular value in $[G^{q_l}(v)$, $G^{q_{l-4}}(v)]$}.
For $n_j+1\leqslant  k\leqslant  n_{j+T_0}+1$, if $v_k=v_k^\circ$ is the singular point associated to $\Phi_k$, then by definition we have $v_k^\circ \in (1+2\delta)I_k$. According to \eqref{equ:position-G} and \eqref{equ:choice-q-l} we have
\begin{equation}\label{equ:G-pos-circ}
G^{k-n_j}(v_k^\circ)\in (2K_1+1) I_{n_j}\subset \big(G^{q_{l-1}}(v),G^{q_l}(v)\big).
\end{equation}

We claim that there exists an integer $k\in [n_j+1, n_{j+T_0}]$ such that the arc $[G^{q_l}(v)$, $G^{q_{l-4}}(v)]$ contains $G^{k-n_j}(v_k)$, where $v_k$ is a singular point associated to $\Phi_k$.
Indeed, for every $1\leqslant  i\leqslant  T_0=t_0+1$, the arc $(1+2\delta)I_{n_{j+i}}$ contains a critical value $v_{n_{j+i}}\in\MCV\subset\T$. Note that $\sharp\,\MCV=t_0$. By \eqref{equ:position-G} and \eqref{equ:choice-q-l} (see also the right inclusion relationship of \eqref{equ:G-pos-circ}), there exist two integers $k_1',k_2'\in [n_j+1, n_{j+T_0}]$ with $k_1'<k_2'$ and a critical value $v'\in\MCV\cap (1+2\delta)I_{k_1'}\cap (1+2\delta)I_{k_2'}$ such that $[G^{q_{l-1}}(v),G^{q_l}(v)]$ contains $G^{k_1'-n_j}(v')$ and $G^{k_2'-n_j}(v')$.
This implies that $k_2'-k_1'\geqslant q_{l-2}$ (see Section \ref{subsec:return} for the first return properties). Then
\begin{equation}\label{equ:G-k1}
G^{k_1'-n_j+q_{l-2}}(v')\in [G^{q_{l-1}+q_{l-2}}(v),G^{q_l+q_{l-2}}(v)]\subset [G^{q_l}(v),G^{q_{l-4}}(v)]
\end{equation}
and
\begin{equation}
n_j+1\leqslant  k_1'\leqslant  k_1'+q_{l-2}\leqslant  k_2'\leqslant  n_{j+T_0}.
\end{equation}
Set $k:=k_1'+q_{l-2}\in[n_j+1,n_{j+T_0}]$. Then by \eqref{equ:G-k1} we have
\begin{equation}
G^{k-n_j}(v') \in [G^{q_l}(v),G^{q_{l-4}}(v)]\subset[G^{q_l}(v),x_0).
\end{equation}
By the definition of $\widetilde{r}_0$ in \eqref{equ:r0-pri} and note that $\sigma\big([x_0',x_0]\big)<r_1\leqslant  \widetilde{r}_0$, we have
\begin{equation}
\big\{G^{k-n_j}(\widetilde{v}):\,\widetilde{v}\in\MCV\big\}\cap [x_0',x_0]=\{G^{k-n_j}(v')\}\subset[G^{q_l}(v),x_0).
\end{equation}
By \eqref{equ:position-G} and \eqref{equ:choice-q-l}, it follows that $v_k:=v'$ is a singular point associated to $\Phi_k$.

\medskip
\textit{Step 2. A decomposition of $[G^{q_{l-4}}(v), x_0]$ by the orbits of singular points}.
By the claim above there exists a least integer $k_1 \in[n_j+1,n_{j+T_0}]$ such that
\begin{equation}
x_1 := G^{k_1-n_j}(v_{k_1})\in [G^{q_l}(v),x_0),
\end{equation}
where $v_{k_1}$ is a singular point associated to $\Phi_{k_1}$.
If $x_1\in[G^{q_l}(v),G^{q_{l-4}}(v)]$, we stop. Otherwise, there exists a least integer $k_2 \in[ k_1+1,n_{j+T_0}]$ such that
\begin{equation}
x_2 := G^{k_2-n_j}(v_{k_2})\in [G^{q_l}(v),x_1),
\end{equation}
where $v_{k_2}$ is a singular point associated to $\Phi_{k_2}$.
Repeating the procedure, we obtain two finite sequences:
\begin{equation}\label{equ:m}
x_m< \cdots< x_2< x_1< x_0 \text{\quad and\quad} n_j+1\leqslant  k_1<k_2<\cdots<k_m\leqslant  n_{j+T_0},
\end{equation}
where $x_i:=G^{k_i-n_j}(v_{k_i})\in(G^{q_{l-4}}(v),x_{i-1})$ for all $1\leqslant  i<m$ with $m\geqslant 1$ and $x_m:=G^{k_m-n_j}(v_{k_m})\in[G^{q_l}(v),G^{q_{l-4}}(v)]$. Moreover, $v_{k_i}$ is a singular point associated to $\Phi_{k_i}$ for every $1\leqslant  i\leqslant  m$.

Let $\Xi := \{x_0, x_1, \cdots, x_m\}$ and $h \geqslant 1$ be the least integer such that $x_0\in [G^{q_h}(v)$, $G^{q_{h-2}}(v)]$. Without loss of generality, we assume that $h$ and $l$ are even and $l-h\geqslant 8$ is also even (decreasing $\sigma(I_0)$ if necessary). Then the union of all
\begin{equation}\label{equ:interval}
[G^{q_{2k}}(v),G^{q_{2k-2}}(v)] \text{ with } h +4\leqslant 2k \leqslant l-4 \text{\quad and\quad} [G^{q_{h+2}}(v),x_0]
\end{equation}
covers $[G^{q_{l-4}}(v),x_0]$.  Each of these intervals is contained in a smallest interval with end points in $\Xi$, say $[x_{l_k}, x_{l_k'}]$ with $m\geqslant l_k>l_k'\geqslant 0$. In particular, the interval $[G^{q_{h+2}}(v),x_0]$ is contained in some $[x_{l_k}, x_{l_k'}]=[x_{l_k}, x_0]$ with $2k=h+2$.

\medskip
\textit{Step 3. Covering $Y$ by half hyperbolic neighborhoods}.
From the fact that the rotation number $\alpha$ of $G$ is of bounded type and that the number of the critical points on $\T$ is bounded, it follows that there is a uniform positive integer $N_0$, which is independent of $k$, such that each of the intervals in \eqref{equ:interval} contains at most $N_0$ points in $\Xi$. Let $Y_k$ denote the closed trapezoid in $\C\setminus\overline{\D}$ which is bounded by the ray $L$, the interval $[x_{l_k}, x_{l_k'}]$, and the two vertical straight segments passing through $x_{l_k}$ and $x_{l_k'}$. Then the interval $[x_{l_k}, x_{l_k'}]$ contains at most $N_0+2$ points in $\Xi$. See Figure \ref{Fig_neigh-Y-k}.

\begin{figure}[!htpb]
  \setlength{\unitlength}{1mm}
  \centering
  \includegraphics[width=0.8\textwidth]{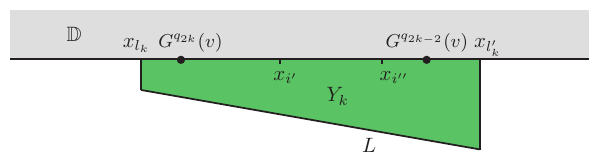}
  \caption{The trapezoid $Y_k$ with some marked boundary points.}
  \label{Fig_neigh-Y-k}
\end{figure}

By Lemma \ref{lemma:hn}, there exist $\widetilde{r}_2=\widetilde{r}_2(N_0)>0$, $\omega=\omega(N_0)\in(0,1)$ and $\beta_0=\beta_0(N_0)\in(0,\pi)$ such that for every $I_k'=(x_{l_k},x_{l_k'})\subset \T$ with $|I_k'|< \widetilde{r}_2$, the following two statements hold:
\begin{itemize}
\item $H^{\beta_0}\big((x_i, x_{i-1})\big)$ is a Jordan disk for all $l_k' +1\leqslant  i\leqslant  l_k$; and
\item $\overline{S_{\omega}(I_k')}\setminus\T\subset \bigcup_{i=l_k' +1}^{l_k} H^{\beta_0}\big((x_i, x_{i-1})\big)$.
\end{itemize}

Since the rotation number $\alpha$ is of bounded type, it follows that there exists $C'>1$ which is independent of $k$ such that the dynamical lengths of $(v,G^{q_{2k}}(v))$ and $(G^{q_{2k}}(v),G^{q_{2k-2}}(v))$ are $C'$-comparable. By Lemma \ref{lem:comparable}, this implies that their Euclidean lengths are $C''$-comparable, where $C''>1$ is a number which is also independent of $k$.
Decreasing the angle $\beta\in(0,\frac{\pi}{3})$ between the ray $L$ and $\T$ to a new angle $\beta=\beta(N_0)\in(0,\frac{\pi}{3})$ if necessary, one can guarantee that $Y_k\subset S_{\omega}\big(I_k'\big)$ and hence $Y_k\subset \bigcup_{i=l_k' +1}^{l_k}H^{\beta_0}\big((x_i, x_{i-1})\big)$ for all $k$.
This implies that
\begin{equation}\label{equ:Y-cover}
Y\subset\bigcup_{i=1}^m H^{\beta_0}\big((x_i, x_{i-1})\big).
\end{equation}
By \eqref{equ:H-beta}, we have $H_{d_0}\big((x_i, x_{i-1})\big)=H^{\beta_0}\big((x_i, x_{i-1})\big)$ for every $1\leqslant  i\leqslant  m$, where $d_0=\log\cot(\frac{\beta_0}{4})$.

In the following, we fix $d:=d_0$. Decreasing $\sigma(I_0)$ and $r_1=r_1(d_0)$ in Lemma \ref{lemma:swz} (see also the definition of $s_0$ in \eqref{equ:D-n-j}) if necessary, we assume that $H_{\tilde{d}}(I)$ is simply connected for any arc $I\subset\T$ with $\min\{|I|,\sigma(I)\}\leqslant \sigma\big((x_0',x_0)\big)$ and any $\tilde{d}\in(0,d_0+1]$.

\medskip
\textit{Step 4. Contraction with respect to the hyperbolic metric}.
For every $n_j\leqslant  k\leqslant  n_{j+T_0}$, we denote
\begin{equation}\label{equ:D-k+1}
D_{k+1}:=\Phi_{k} \circ \cdots \circ \Phi_{n_j+1}\circ \Phi_{n_j}(D_{n_j}^{s_0}).
\end{equation}
Then $D_{k+1}$ is a simply connected domain contained in $\Omega$.
We have the following two cases (I) and (II).

\medskip
\textbf{(I)} Suppose there is a least integer $k_1'\in[n_j+1,k_1]$ such that $\big(G^{q_{l-1}}(v),G^{q_l}(v)\big)$ contains
$x_1':=G^{k_1'-n_j}(v_{k_1'}^+)$ or $G^{k_1'-n_j}(v_{k_1'}^\circ)$. Note that we don't need to consider the forward orbits of singular points in $[x_0',G^{q_{l-1}}(v)]$ since we have assumed that $v=v_{n_j}=v_{n_j}^-$.
For every $n_j\leqslant  k\leqslant  k_1'-1$, we denote
\begin{equation}
P_{k+1,1}:=\Phi_{k}\circ \cdots\circ \Phi_{n_j+1}\circ \Phi_{n_j}\big(P_1\big),
\end{equation}
where $P_1:=H_{d_0}\big((G^{q_{l-4}}(v),x_0)\big)$.
Let $J_1' \subset \T$ be the arc satisfying $G^{k_1'-n_j}(J_1') = (G^{q_{l-4}}(v), x_0)$. By the definition of $k_1'$, it follows that for every $n_j\leqslant  k\leqslant  k_1'-1$, $D_{k+1}\cap\T$ and $\partial P_{k+1,1}\cap\T$ are proper subarcs of $\T$ satisfying
\begin{equation}
\partial D_{k+1}\cap\T \supset  I_{k+1}\cup G^{k_1'-1-k}(J_1') \text{\quad and\quad}
\partial P_{k+1,1}\cap\T = G^{k_1'-1-k}(\overline{J_1'}).
\end{equation}
For $z\in Y\setminus\T\subset P_1$, we have $w_1':=\Phi_{k_1'-1}\circ\cdots\circ\Phi_{n_j+1}\circ\Phi_{n_j}(z)\in P_{k_1',1}\setminus\T\subset H_{d_0+1}(J_1')$ by Lemma \ref{lemma:swz-2}(b).

Let $\widetilde{J}_1\subset\T$ be the arc containing $J_1'$ such that $G^{k_1'-n_j}(\widetilde{J}_1) = (x_1', x_0)$.
Note that $\Phi_{k_1'}$ is defined in the simply connected domain $D_{k_1}\cup P_{k_1',1}$ by analytic continuation of $\Phi_{k_1'}|_{I_{k_1'}}$.
Since the left end point of $\widetilde{J}_1$ is $v_{k_1'}^+$ or $v_{k_1'}^\circ$, by Lemma \ref{lemma:Pe} there exists a constant $\mu_2=\mu_2(d_0)\in(0,1)$ depending only on $d_0$ such that
\begin{equation}\label{equ:mu-0-first}
\rho_\Omega(\Phi_{k_1'}(w_1')) <\mu_2\,\rho_\Omega(w_1')|G'(\Phi_{k_1'}(w_1'))| .
\end{equation}
Hence Part (b) is proved for $z\in Y$ with $\mu_0:=\mu_2$ by the Schwarz-Pick Lemma.

\medskip
\textbf{(II)} Suppose for any $k\in[n_j+1,k_1]$, the arc $\big(G^{q_{l-1}}(v),G^{q_l}(v)\big)$ contains neither
$G^{k-n_j}(v_k^+)$ nor $G^{k-n_j}(v_k^\circ)$.
For every $n_j\leqslant  k\leqslant  k_1-1$, we denote
\begin{equation}
\widetilde{P}_{k+1,1}:=\Phi_{k}\circ \cdots\circ \Phi_{n_j+1}\circ \Phi_{n_j}\big(\widetilde{P}_1\big),
\end{equation}
where $\widetilde{P}_1:=H_{d_0}\big((x_1,x_0)\big)$.
Let $J_1 \subset \T$ be the arc satisfying $G^{k_1-n_j}(J_1) = (x_1, x_0)$. Similar to Case (I), for every $n_j\leqslant  k\leqslant  k_1-1$, $D_{k+1}\cap\T$ and $\partial \widetilde{P}_{k+1,1}\cap\T$ are proper subarcs of $\T$ satisfying
\begin{equation}
\partial D_{k+1}\cap\T \supset I_{k+1}\cup G^{k_1-1-k}(J_1) \text{\quad and\quad}
\partial \widetilde{P}_{k+1,1}\cap\T =  G^{k_1-1-k}(\overline{J}_1).
\end{equation}
For $z\in \widetilde{P}_1$, we have $w_1:=\Phi_{k_1-1}\circ\cdots\circ\Phi_{n_j+1}\circ\Phi_{n_j}(z)\in \widetilde{P}_{k_1,1}\subset H_{d_0+1}(J_1)$ by Lemma \ref{lemma:swz-2}(b).

Note that the left end point of $J_1$ is $v_{k_1}=v_{k_1}^+$ (it cannot be $v_{k_1}^\circ$ by \eqref{equ:G-pos-circ} and cannot be $v_{k_1}^-$ by \eqref{equ:position-G} and \eqref{equ:choice-q-l}). Still by Lemma \ref{lemma:Pe} we have
\begin{equation}
\rho_\Omega(\Phi_{k_1}(w_1)) < \mu_2\,\rho_\Omega(w_1)|G'(\Phi_{k_1}(w_1))|,
\end{equation}
where $\mu_2\in(0,1)$ is the same number as that in \eqref{equ:mu-0-first}.
Hence Part (b) is proved for $z\in \widetilde{P}_1=H_{d_0}\big((x_1,x_0)\big)$ with $\mu_0=\mu_2$ by the Schwarz-Pick Lemma. In particular, if $m=1$, then Part (b) is proved for $z\in Y$.

\medskip
In the following we assume that we are in Case (II) and $m\geqslant 2$. We again have two subcases (II-i) and (II-ii).

\medskip
\textbf{(II-i)} Suppose there is a least integer $k_2'\in[k_1+1,k_2]$ such that $\big(G^{q_{l-1}}(v),G^{q_l}(v)\big)$ contains
$x_2':=G^{k_2'-n_j}(v_{k_2'}^+)$ or $G^{k_2'-n_j}(v_{k_2'}^\circ)$.
For every $n_j\leqslant  k\leqslant  k_2'-1$, we denote
\begin{equation}
P_{k+1,2}:=\Phi_{k}\circ \cdots\circ \Phi_{n_j+1}\circ \Phi_{n_j}\big(P_2\big),
\end{equation}
where $P_2:=H_{d_0}\big((G^{q_{l-4}}(v),x_1)\big)$.
Let $J_2' \subset \T$ be the arc satisfying $G^{k_2'-n_j}(J_2') = (G^{q_{l-4}}(v), x_1)$. By the definition of $k_2'$, it follows that for every $n_j\leqslant  k\leqslant  k_2'-1$, $D_{k+1}\cap\T$ and $\partial P_{k+1,2}\cap\T$ are proper subarcs of $\T$ satisfying
\begin{equation}
\partial D_{k+1}\cap\T \supset  I_{k+1}\cup G^{k_2'-1-k}(J_2') \text{\quad and\quad}
\partial P_{k+1,2}\cap\T = G^{k_2'-1-k}(\overline{J_2'}).
\end{equation}
For $z\in P_2$, we have $w_2':=\Phi_{k_2'-1}\circ\cdots\circ\Phi_{n_j+1}\circ\Phi_{n_j}(z)\in P_{k_2',2}\subset H_{d_0+1}(J_2')$ by Lemma \ref{lemma:swz-2}(b).

Let $\widetilde{J}_2\subset\T$ be the arc containing $J_2'$ such that $G^{k_2'-n_j}(\widetilde{J}_2) = (x_2', x_1)$.
Similar to Case (I), since the left end point of $\widetilde{J}_2$ is $v_{k_2'}^+$ or $v_{k_2'}^\circ$, by Lemma \ref{lemma:Pe} we have
\begin{equation}
\rho_\Omega(\Phi_{k_2'}(w_2')) < \mu_2\,\rho_\Omega(w_2')|G'(\Phi_{k_2'}(w_2'))|,
\end{equation}
where $\mu_2\in(0,1)$ is the same number as that in \eqref{equ:mu-0-first}.
Note that $Y\subset P_2\cup \widetilde{P}_1=H_{d_0}\big((G^{q_{l-4}}(v),x_1)\big)\cup H_{d_0}\big((x_1,x_0)\big)$.
Hence Part (b) is proved for $z\in Y$ with $\mu_0=\mu_2$ by the Schwarz-Pick Lemma.

\medskip
\textbf{(II-ii)} Suppose for any $k\in[k_1+1,k_2]$, the arc $\big(G^{q_{l-1}}(v),G^{q_l}(v)\big)$ contains neither
$G^{k-n_j}(v_k^+)$ nor $G^{k-n_j}(v_k^\circ)$.
For every $n_j\leqslant  k\leqslant  k_2-1$, we denote
\begin{equation}
\widetilde{P}_{k+1,2}:=\Phi_{k}\circ \cdots\circ \Phi_{n_j+1}\circ \Phi_{n_j}\big(\widetilde{P}_2\big),
\end{equation}
where $\widetilde{P}_2:=H_{d_0}\big((x_2,x_1)\big)$.
Let $J_2 \subset \T$ be the arc satisfying $G^{k_2-n_j}(J_2) = (x_2, x_1)$. Similarly, for every $n_j\leqslant  k\leqslant  k_2-1$, $D_{k+1}\cap\T$ and $\partial \widetilde{P}_{k+1,2}\cap\T$ are proper subarcs of $\T$ satisfying
\begin{equation}
\partial D_{k+1}\cap\T \supset I_{k+1}\cup G^{k_2-1-k}(J_2) \text{\quad and\quad}
\partial \widetilde{P}_{k+1,2}\cap\T =  G^{k_2-1-k}(\overline{J}_2).
\end{equation}
For $z\in \widetilde{P}_2$, we have $w_2:=\Phi_{k_2-1}\circ\cdots\circ\Phi_{n_j+1}\circ\Phi_{n_j}(z)\in \widetilde{P}_{k_2,2}\subset H_{d_0+1}(J_2)$ by Lemma \ref{lemma:swz-2}(b).

Note that the left end point of $J_2$ is $v_{k_2}=v_{k_2}^+$. Still by Lemma \ref{lemma:Pe} we have
\begin{equation}
\rho_\Omega(\Phi_{k_2}(w_2)) < \mu_2\,\rho_\Omega(w_2)|G'(\Phi_{k_2}(w_2))|,
\end{equation}
where $\mu_2\in(0,1)$ is the same number as that in \eqref{equ:mu-0-first}.
Then Part (b) is proved for $z\in \widetilde{P}_2\cup \widetilde{P}_1=H_{d_0}\big((x_2,x_1)\big)\cup H_{d_0}\big((x_1,x_0)\big)$ with $\mu_0=\mu_2$ by the Schwarz-Pick Lemma. In particular, if $m=2$, then Part (b) is proved for $z\in Y$.

\medskip
Next assume that we are in Case (II-ii) and $m\geqslant 3$. Inductively on $m$ as Cases (II-i) and (II-ii),  Part (b) holds by the Schwarz-Pick Lemma for all $z\in Y$ with the constant $\mu_2=\mu_2(d_0)\in (0,1)$ depending only on $d_0$ since $Y$ is covered by the union of finitely many half hyperbolic neighborhoods $H_{d_0}\big((x_i, x_{i-1})\big)$, where $1 \leqslant i \leqslant m$ (see \eqref{equ:Y-cover}). Based on \eqref{equ:mu-1}, we set $\mu_0:=\max\{\mu_1,\mu_2\}<1$.
The proof is complete.
\end{proof}

The two conclusions in Proposition \ref{prop:three special regions} are not mutually exclusive. For example, when $v_{n_j}=v_{n_j}^\circ$, the points in $Z$ also satisfy Part (b).
In the rest of this paper, we fix a large number $d>1$ such that Proposition \ref{prop:three special regions} holds.

\section{Proof of the Main Lemma$'$} \label{sec:key-thm}

\subsection{Improved $(\delta,\eta)$-admissible sequences}\label{subsec:improv-d-e}

Let $I \subset \T$ be a small arc and $d>1$ a large number fixed before. Let $\Omega$ be the unique component of $\EC \setminus(\MP(G) \cup \overline{\D})$ such that $\T \subset \partial \Omega$ (see \eqref{equ:Omega}). For $\tilde{d}\in[d,d+1]$, we use $\Sigma_{I, \tilde{d}}$ to denote the class consisting of all the sets $F$ such that
$$
F = \bigcup_{k\in\K} B_k,
$$
where $\K$ is an at most countable set, and each $B_k$ is a Jordan disk compactly contained in $\Omega$ satisfying $B_k\cap H_{\tilde{d}}(I)\neq\emptyset$.
Define
\begin{equation}
\Upsilon_\Omega(H_{\tilde{d}}(I), F) :=\sup_{k\in\K} \big\{ \diam_\Omega(B_k) \big\}.
\end{equation}

Let $\{I_n\}_{n\geqslant 0}$ be a $(\delta,\eta)$-admissible sequence.
For $n \geqslant 0$, let $\Phi_n$ be the inverse branch of $G$ associated to the pullback $I_n\to I_{n+1}$, which is defined in any simply connected domain in $\EC\setminus(\overline{\D}\cup\MP(G))$ containing $D_n^{s_0}$ (see \eqref{equ:D-n-j}).
Let $r_5=r_5(\delta,\eta,d)>0$ be introduced in Proposition \ref{prop:three special regions}.

\begin{lem}\label{lemma:improved sequence}
There exist positive numbers $r_6\in(0,r_5]$, $K'$, $K$, $T_2$ depending only on $\delta,\eta,d$, such that for any $(\delta,\eta)$-admissible sequence $\{I_n\}_{n\geqslant 0}$  with $\sigma(I_0) < r_6$ and any $F_0 \in \Sigma_{I_0, d}$ with $\Upsilon_\Omega(H_d(I_0), F_0) < K' $ and $H_d(I_0)\cup F_0\subset D_0^{s_0}$, there exist
\begin{itemize}
\item $\ell\geqslant 1$ and $F_i \in \Sigma_{I_i, d_i}$ with $d_i\in[d,d+1]$, where $1 \leqslant i\leqslant \ell$; and
\item $F \in \Sigma_{J, d}$ for a subarc $J \subset \T$ containing $I_\ell$ with $\sigma(J) <\sigma(I_0)/2$,
\end{itemize}
such that
\begin{itemize}
\item $\{I_i:0\leqslant  i\leqslant  \ell\}$ contains exactly $T_2$ critical predecessors;
\item $\Phi_{i-1}\circ\cdots\circ\Phi_1\circ\Phi_0\big(H_d(I_0)\cup F_0\big)\subset H_{d_i}(I_i)\cup F_i\subset D_i^{s_0}$ for all $1\leqslant  i\leqslant  \ell$;
\item $\Upsilon_\Omega(H_{d_i}(I_i), F_i)  < K$ for all $1\leqslant  i\leqslant  \ell-1$; and
\item $H_{d_\ell}(I_{\ell}) \cup F_{\ell} \subset H_d(J) \cup F$ and $\Upsilon_\Omega(H_d(J), F)  < K'$.
\end{itemize}
\end{lem}

\begin{proof}
Let $K_0=K_0(\delta,\eta,d) > 0$ and $T_1=T_1(\eta)\geqslant 1$ be introduced in Lemmas \ref{lem:crit-pullback} and \ref{lemma:shorten} respectively. For the constants $C_0=C_0(\delta,\eta,d) > 1$ and $0<\mu_0<1$ introduced in Proposition \ref{prop:three special regions}, let $k_0 \geqslant 3$ be the minimal integer such that
\begin{equation}
(\tfrac{1}{2})^{k_0-2} C_0 < \tfrac{1}{2}.
\end{equation}
Let $T_0=t_0+1$ be introduced in \eqref{equ:T-0}, where $t_0=\sharp\,(\Crit(G)\cap\T)\geqslant 1$. Set
$$
K' := \frac{\mu_0 k_0 T_1K_0}{1 - \mu_0}.
$$

Let $F_0 \in \Sigma_{I_0, d}$ with $\Upsilon_\Omega(H_d(I_0), F_0) < K'$.
We assume that $\sigma(I_0)$ is small enough such that $H_d(I_0)\cup F_0\subset D_0^{s_0}$ (see \eqref{equ:D-n-j}).
Let $j:=k_0T_1+1$ and $n_j$ be the $j$-th integer such that $I_{n_j+1}$ is a critical predecessor of $I_{n_j}$ (see \eqref{equ:crit-seq}).
Based on a similar proof to \eqref{equ:W-p} and \eqref{equ:B-k-m} (decreasing $\sigma(I_0)$ if necessary), by Lemmas \ref{lemma:swz-2}(b), \ref{lem:crit-pullback} and \ref{lemma:shorten}, we obtain $F_i \in \Sigma_{I_i, d_i}$ for $1 \leqslant i \leqslant n_j$ with $d_i\in[d,d+1]$, such that
\begin{itemize}
\item There are exactly $k_0T_1$ critical predecessors in $\{I_i:0\leqslant  i\leqslant  n_j\}$;
\item $\sigma(I_{n_j}) < \sigma(I_0)/2^{k_0-1}$;
\item $\Phi_{i-1}\circ\cdots\circ\Phi_1\circ\Phi_0\big(H_d(I_0)\cup F_0\big)\subset H_{d_i}(I_i)\cup F_i\subset D_i^{s_0}$ for all $1\leqslant  i\leqslant  n_j$; and
\item $\Upsilon_\Omega(H_{d_i}(I_i), F_i) < K'+k_0T_1K_0$ for all $1\leqslant  i\leqslant  n_j$.
\end{itemize}

Considering the pullbacks from $I_{n_j}$ to $I_\ell$ with $\ell:=n_{j+T_0}+1$, by Lemmas \ref{lemma:swz-2}(b), \ref{lem:crit-pullback} and \ref{lemma:shorten} we obtain $F_i \in \Sigma_{I_i, d_i}$ for $n_j+1 \leqslant i \leqslant \ell$ with $d_i\in[d,d+1]$, such that
\begin{itemize}
\item There are exactly $T_2$ critical predecessors in $\{I_i:0\leqslant  i\leqslant  \ell\}$, where
\begin{equation}
T_2 := k_0 T_1+ T_0+1;
\end{equation}
\item $\sigma(I_\ell) < 2\sigma(I_{n_j})<\sigma(I_0)/2^{k_0-2}$;
\item $\Phi_{i-1}\circ\cdots\circ\Phi_1\circ\Phi_0\big(H_d(I_0)\cup F_0\big)\subset H_{d_i}(I_i)\cup F_i\subset D_i^{s_0}$ for all $1\leqslant  i\leqslant  \ell$; and
\item $\Upsilon_\Omega(H_{d_i}(I_i), F_i) < K$ for all $1\leqslant  i\leqslant  \ell-1$, where
\begin{equation}
K:=(K'+k_0T_1K_0)+T_0 K_0.
\end{equation}
\end{itemize}
By Lemma \ref{lem:crit-pullback}, $d_{\ell}$ is chosen to be $d$.

Let $D_{n_j}^{s_0}$ be the set defined in \eqref{equ:D-n-j}. Without loss of generality, we assume that $\sigma(I_0)$ is small enough such that
\begin{itemize}
  \item $H_{d_{n_j}}(I_{n_j})  \cup F_{n_j}  \subset D_{n_j}^{s_0}$; and
  \item The shortest geodesic curve connecting any different points $z_1$, $z_2$ in every Jordan disk of $F_{n_j}$ with respect to $\rho_\Omega(z)|dz|$ is contained in $D_{n_j}^{s_0}$.
\end{itemize}
By Proposition \ref{prop:three special regions} (see also Figure \ref{Fig_neigh-XYZ}), there exists a closed triangle $Z\subset D_{n_j}^{s_0}$ such that
\begin{equation}
\widetilde{\Phi}:=\Phi_{\ell-1}\circ\cdots\circ\Phi_{n_j+1}\circ\Phi_{n_j}
\end{equation}
contracts the hyperbolic metric $\rho_\Omega(z)|dz|$ in $D_{n_j}^{s_0}\setminus Z$ strictly and $\widetilde{\Phi}(Z)\subset H_d(J)$, where $J\subset\T$ is an arc containing $I_\ell$ which satisfies
\begin{equation}
\sigma(J)<C_0\sigma(I_\ell)<C_0\sigma(I_0)/2^{k_0-2}<\sigma(I_0)/2.
\end{equation}
If $z\in H_{d_{n_j}}(I_{n_j})$, then there exists $F_\ell'\in\Sigma_{I_\ell,d}$ such that
\begin{equation}\label{equ:bd-1}
\widetilde{\Phi}(z)\in H_d(I_\ell)\cup F_\ell', \text{\quad where }\Upsilon_\Omega(H_d(I_\ell), F_\ell')<(T_0+1)K_0.
\end{equation}
By the construction of $J$, increasing the number $C_3=C_3(\delta,\eta,d)>1$ in \eqref{equ:J-1} if necessary, without loss of generality by \eqref{equ:bd-1} we assume that
\begin{equation}\label{equ:posi-1}
\widetilde{\Phi}\big(H_{d_{n_j}}(I_{n_j})\big)\subset H_d(I_\ell)\cup F_\ell'\subset H_d(J).
\end{equation}

Note that $\Upsilon_\Omega(H_{d_{n_j}}(I_{n_j}), F_{n_j}) < K'+k_0T_1K_0$ and $F_{n_j}\setminus Z$ can be written as
\begin{equation}
F_{n_j}\setminus Z=\bigcup_{k\in\K} B_k,
\end{equation}
where $\K\neq\emptyset$ is an at most countable set, and each $B_k$ is a Jordan disk compactly contained in $\Omega$ satisfying
\begin{equation}
\overline{B}_k\cap (\partial Z\setminus\T) \neq\emptyset \text{\quad and\quad}
\diam_\Omega(B_k)<K'+k_0T_1K_0.
\end{equation}
Since $\widetilde{\Phi}\big(H_{d_{n_j}}(I_{n_j})\cup Z\big)\subset H_d(J)$, we have $\widetilde{\Phi}(B_k)\cap H_d(J)\neq\emptyset$ for all $k\in\K$.
Define $F:=\bigcup_{k\in\K}\widetilde{\Phi}(B_k)$. Then $F\in\Sigma_{J,d}$ and we have
\begin{equation}
\Upsilon_\Omega(H_d(J), F) <\mu_0(K'+k_0T_1K_0)=K'.
\end{equation}
The proof is complete.
\end{proof}

To every $(\delta,\eta)$-admissible sequence $\{I_n\}_{n\geqslant 0}$, we associate a sequence of triples $\{(J_n,F_n,d_n)\}_{n\geqslant 0}$ as following.
Set $F_0 =\emptyset$. Let $\ell \geqslant 1$, $F_i \in \Sigma_{I_i, d_i}$ for $1 \leqslant i\leqslant \ell-1$ with $d_i\in[d,d+1]$, $J \subset \T$ and $F\subset \Sigma_{J, d}$ be guaranteed by Lemma \ref{lemma:improved sequence}.
Define
\begin{equation}
\begin{split}
J_i &:= I_i \text{\quad for }0 \leqslant i \leqslant \ell-1  \text{\quad and} \\
J_{\ell} &:= J, \quad F_{\ell} := F \text{\quad and\quad} d_\ell:=d.
\end{split}
\end{equation}
Let $I_0' := J_{\ell}$ and we define a new $(\delta,\eta)$-admissible sequence $\{I_n' \}_{n\geqslant 0}$ such that
the pullback $I_n' \to I_{n+1}'$ is induced by $I_{n+ \ell} \to I_{n+\ell+1}$ and $I_{n+\ell} \subset I_n'$ for all $n\geqslant 0$.
Denote $F_0' := F_{\ell} $.
Then $\Upsilon_\Omega(H_d(I_0'), F_0') < K' $.

\medskip
By applying Lemma \ref{lemma:improved sequence} to $\{I_n' \}_{n\geqslant 0}$ and $F_0'$, we obtain
$\ell'\geqslant 1$, $F_i' \in \Sigma_{I_i', d_i'}$ for $1 \leqslant i \leqslant \ell'-1$ with $d_i'\in[d,d+1]$, $J' \subset \T$ and $ F' \in \Sigma_{J', d}$.
Define
\begin{align}
J_{\ell+i} &:=I_i', & F_{\ell+i}&:= F_i', & d_{\ell+i}&:=d_i'\text{\quad for }1 \leqslant i \leqslant \ell'-1;  \text{\quad and}\\
J_{\ell+\ell'} &:= J', & F_{\ell+\ell'} &:= F', & d_{\ell+\ell'}&:=d.
\end{align}
Inductively, by taking $I_0'' :=J_{\ell+\ell'}$, $F_0'' := F_{\ell+\ell'}$ and repeating the above process, we thus obtain an \textit{improved $(\delta,\eta)$-admissible sequence} of triples $\{(J_n,F_n,d_n)\}_{n\geqslant 0}$.

\begin{cor}\label{cor:small}
For any $\varepsilon>0$, there exists $\widetilde{r}=\widetilde{r}(\varepsilon)\in(0,r_6]$ such that for any improved $(\delta,\eta)$-admissible sequence $\{(J_n,F_n,d_n)\}_{n\geqslant 0}$ with $\sigma(J_0)=\sigma(I_0) < \widetilde{r}$, we have
$\diam_{\EC}(H_{d_n}(J_n)\cup F_n)<\varepsilon$ for all $n\geqslant 0$.
\end{cor}

\begin{proof}
For each $n\geqslant 0$, suppose
\begin{equation}
F_n = \bigcup_{k\in \K_n} B_k^n,
\end{equation}
where $\K_n$ is an at most countable set and each $B_k^n$ is a Jordan disk in $\Omega$ satisfying $B_k^n\cap H_{d_n}(J_n)\neq\emptyset$. Then
\begin{equation}\label{equ:diam-1}
\diam_{\EC}\big(H_{d_n}(J_n)\cup F_n\big)\leqslant  \diam_{\EC}\big(H_{d_n}(J_n)\big)
+ 2\sup_{k\in \K_n}\big\{\diam_{\EC}(B_k^n)\big\}.
\end{equation}
By Lemma \ref{lemma:improved sequence}, we have
\begin{equation}\label{equ:diam-2}
\sup_{k\in \K_n}\big\{\diam_{\Omega}(B_k^n)\big\}<K.
\end{equation}

From the construction of $\{(J_n,F_n,d_n)\}_{n\geqslant 0}$ and by Lemma \ref{lemma:improved sequence}, we have $d_n\in[d,d+1]$ and $\sigma(J_n) < 2\sigma(I_0)$ for all $n\geqslant 0$. Hence for given $\varepsilon>0$,  by \eqref{equ:diam-1} and \eqref{equ:diam-2}, if $\sigma(I_0)$ is small enough, then $\diam_{\EC}\big(H_{d_n}(J_n)\cup F_n\big)<\varepsilon$ for all $n\geqslant 0$.
\end{proof}

By definition and Corollary \ref{cor:small}, we have
\begin{equation}
\Phi_{n-1}\circ\cdots\circ\Phi_1\circ\Phi_0\big(H_d(I_0)\big)\subset H_{d_n}(J_n)\cup F_n, \text{\quad for all } n\geqslant 1.
\end{equation}

\subsection{Jumping off preimages} \label{subsec:jump-off}

Let $V_0\subset H_d(I_0)$ be a Jordan disk and $\{V_n\}_{n\geqslant 0}$ be a pullback sequence of $V_0$ under $G$ in $\EC\setminus\overline{\D}$. By Lemma \ref{lem:take-pre}, every $V_n$ is a Jordan disk.
For any given arc $I_0^*\subset\T$ with $\sigma(I_0^*)< r_6$, we consider the family of all the possible $(\delta,\eta)$-admissible sequences:
\begin{equation}
\MI:=
\left\{\tau=\{I_n\}_{n\geqslant 0}
\left|
\begin{array}{l}
\tau \text{ is a } (\delta,\eta) \text{-admissible sequence} \\
\text{beginning with } I_0=I_0^*
\end{array}
\right.
\right\}
\end{equation}
and the corresponding improved $(\delta,\eta)$-admissible sequences $\big\{\tau'=\{(J_n,F_n,d_n)\}_{n\geqslant 0}\big\}$.
Define
\begin{equation}
n(\tau):=\sup\{n\in\N: ~ V_k \subset H_{d_k}(J_k) \cup F_k \text{ for all } 0 \leqslant k \leqslant n\}
\end{equation}
and
\begin{equation}
\chi(I_0^*, \{V_n\}_{n\geqslant 0})  := \sup_{\tau\in\MI} \, \{n(\tau)\}.
\end{equation}

\begin{lem}\label{lem:shrink}
For any $\varepsilon>0$, let $\widetilde{r}=\widetilde{r}(\varepsilon)>0$ be the number in Corollary \ref{cor:small}. Then for any pullback sequence $\{V_n\}_{n\geqslant 0}$ of a Jordan disk $V_0\subset H_d(I_0)$ with $\sigma(I_0) < \widetilde{r}$, we have the following two cases:
\begin{enumerate}
\item If $\chi(I_0,\{V_n\}_{n\geqslant 0}) = +\infty$, then $\diam_{\EC}(V_n) < \varepsilon$ for all $n\geqslant 0$;
\item If $\chi(I_0,\{V_n\}_{n\geqslant 0}) = N<+\infty$, then $\diam_{\EC}(V_n) < \varepsilon$ for all $0\leqslant  n\leqslant  N$.
\end{enumerate}
\end{lem}

\begin{proof}
We only prove Part (a) since the proof for Part (b) is completely the same.
Suppose $\chi(I_0,\{V_n\}_{n\geqslant 0}) = +\infty$. Then there exists a sequence of improved $(\delta,\eta)$-admissible sequences $\big\{\{(J_n^{(m)}, F_n^{(m)}, d_n^{(m)})\}_{n\geqslant 0}:m\geqslant 0\big\}$ such that for any $m\geqslant 0$,
\begin{equation}
V_n \subset H_{d_n^{(m)}}(J_n^{(m)}) \cup F_n^{(m)} \text{\quad for all } 0\leqslant  n\leqslant  m.
\end{equation}
By Corollary \ref{cor:small}, if $\sigma(I_0) < \widetilde{r}$, then
\begin{equation}
\diam_{\EC}(V_n)\leqslant  \diam_{\EC}\big(H_{d_n^{(m)}}(J_n^{(m)}) \cup F_n^{(m)}\big)<\varepsilon \text{\quad for all } 0\leqslant  n\leqslant  m.
\end{equation}
This implies that $\diam_{\EC}(V_n) < \varepsilon$ for all $n\geqslant 0$.
\end{proof}

\begin{defi}[Jump off from $\T$]
If $\chi(I_0, \{V_n\}_{n\geqslant 0}) = N<+\infty$, we say that $V_{N+1}$ is the \textit{first jump off} from $\T$ with respect to $I_0$.
\end{defi}

For a piecewise smooth curve $\gamma$ in $\Omega$, we use $l_\Omega(\gamma)$ to denote the length of $\gamma$ with respect to the hyperbolic metric $\rho_\Omega(z)|dz|$.
Let $V$ be a Jordan disk in $\Omega$ with $\overline{V}\subset\Omega$. For two different points $z_1,z_2\in V$, let $\Gamma_V(z_1,z_2)$ be the collection of all smooth curves in $V$ connecting $z_1$ with $z_2$. We define a variant hyperbolic diameter of $V$ in $\Omega$:
\begin{equation}\label{equ:diam-V-variant}
\widetilde{\diam}_{\Omega}(V):=\sup_{z_1,z_2\in V}\Big\{\inf_{\gamma\in\Gamma_V(z_1,z_2)}\big\{l_\Omega(\gamma)\big\}\Big\}.
\end{equation}
By definition we have $\diam_\Omega(V)\leqslant  \widetilde{\diam}_{\Omega}(V)$.

For an annulus $A$ in $\EC$, we use $\Mod(A)$ to denote the conformal modulus of $A$. Let $\MP(G)$ be the postcritical set of $G$ defined in \eqref{equ:P-G}. We denote
\begin{equation}\label{equ:P-G-hat}
\widehat{\MP}(G):=\{z\in\MP(G)\setminus\T:\exists\, n\geqslant 1 \text{ such that } G^n(z)\in\T\}.
\end{equation}
Let $r_6>0$ be the number introduced in Lemma \ref{lemma:improved sequence}.

\begin{lem}\label{lemma:jumps}
Let $M > 0$ and $r' \in(0, r_6]$ be given. Then there exists $r''=r''(M,r')\in(0,r']$ such that for any pullback sequence $\{V_n\}_{n\geqslant 0}$ of a Jordan disk $V_0 \subset H_d(I_0)$ with $I_0 \subset \T$ and $\sigma(I_0) < r''$, if $V_{N+1}$ is the first jump off from $\T$ with respect to $I_0$,  then there exists a Jordan disk $U_{N+1}$ such that one of the following holds:
\begin{enumerate}
\item The first type jump off: $V_{N+1}\subset U_{N+1}\subset\EC\setminus\overline{\D}$, $\Mod(U_{N+1}\setminus \overline{V}_{N+1})>M$ and $U_{N+1}\cap \MP(G)\subset \widehat{\MP}(G)$ with $\sharp (U_{N+1}\cap \MP(G))\leqslant  1$; or
\item The second type jump off: $V_{N+1}\subset U_{N+1}\subset H_d(J')$ for an arc $J' \subset \T$ with $\sigma(J')  = r'$,
\begin{equation}
\quad\qquad \rho_\Omega(w) \leqslant\mu\,\rho_\Omega(G(w))|G'(w)| \text{ for } w\in V_{N+1}
\text{\quad and\quad}
\widetilde{\diam}_\Omega(U_{N+1}) < K,
\end{equation}
where $0<\mu<1$ and $K>0$ are constants depending only on $\delta$, $\eta$ and $d$.
\end{enumerate}
\end{lem}

\begin{proof}
Let $\{(J_n,F_n,d_n)\}_{n\geqslant 0}$ be an improved $(\delta,\eta)$-admissible sequence such that $V_N\subset H_{d_N}(J_N)\cup F_N$.
By Corollary \ref{cor:small},  if $\sigma(I_0)> 0$ is small enough, then $H_{d_N}(J_N)\cup F_N$ can be arbitrarily close to $\T$ and its spherical diameter can be arbitrarily small.
Note that one of the following happens:
\begin{itemize}
  \item[(i)] $v\in (1+2\delta)J_N$ for some critical value $v\in \MCV$; or
  \item[(ii)] $v\not\in(1+2\delta)J_N$ for any critical value $v\in \MCV$.
\end{itemize}

Suppose Case (i) happens. By the definition of jump off, it follows that the branch of $G^{-1}$  determined by $V_N \to V_{N+1}$, say $\Phi$, is not associated to any $(\delta,\eta)$-critical or non-critical pullback of $J_N$.  So $\Phi(v)$  does not belong to $\mathbb T$ and  $V_{N+1}$  is  bounded away from $\T$.  In particular, for given $M>0$, there exists $s_1=s_1(M)>0$ such that if $\sigma(I_0)<s_1$, then $\diam_{\EC}(V_N)$ is sufficiently small and there exists a Jordan disk $U_{N+1}$ such that Part (a) holds.

\medskip
Suppose Case (ii) happens.
For three different points $z_0,z_1,z_2\in\C\setminus\D$, let $\angle\, z_1 z_0 z_2$ be the angle measured in the logarithmic plane of $G$ (see \eqref{equ:angle}).
Let $\widetilde{V}_{N+1}$ be the component of $G^{-1}(H_{d_N}(J_N)\cup F_N)$ which contains $V_{N+1}$.
Let $r'\in(0,r_6]$ be given.
We claim that there exist $s_2=s_2(r')>0$, constants $0 < C_1 <  C_2 < 1$ and $K_1>1$ depending only on $\delta$, $\eta$, $d$ such that if $\sigma(I_0)<s_2$, then one of the following is true (see Figure \ref{Fig_crit-pullback-iii-useful}):
\begin{itemize}
\item[(1)] ${\rm dist}_{\C} (\widetilde{V}_{N+1}, \T) >  C_1 r'$; or
\item[(2)] There exist an arc $J'=(x_1,x_2)\subset\T$ containing a critical point $c \in \T$ with $\sigma(J') = r'$ and constants $\beta\in(0,\pi/2)$, $r_{int},r_{out}\in(0, C_2 r')$ depending only on $\delta,\eta,d$ with $1<r_{out}/r_{int} < K_1$ such that $\widetilde{V}_{N+1} \subset U_{N+1}\subset H_d(J')$, where
\begin{equation}
\qquad  U_{N+1} = \big\{z\in\C \,\big|\, r_{int} < |z - c| < r_{out} \text{ and } \angle\,zcx_2\in(\beta,\pi-\beta)\big\}.
\end{equation}
\end{itemize}

\begin{figure}[!htpb]
 \setlength{\unitlength}{1mm}
 \centering
  \includegraphics[width=0.92\textwidth]{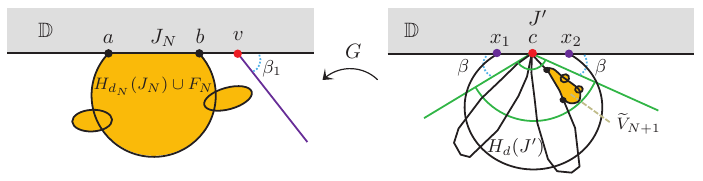}
  \caption{A sketch of the second type jump off.}
 \label{Fig_crit-pullback-iii-useful}
\end{figure}

\medskip
We first prove the claim. In fact, if ${\rm dist}_{\C}(\widetilde{V}_{N+1}, \T) < r'$ (note that $r' > 0$ is small), there must be a critical point $c \in \T$ such that $\widetilde{V}_{N+1}$ is contained in a small neighborhood of $c$. This implies that $H_{d_N}(J_N)\cup F_N$ is contained in a small neighborhood of the critical value $v = G(c)$.
Without loss of generality, we assume that $v$ is on the right of $J_N=(a,b)$.
By \eqref{equ:diam-2}, there exists $s_2'=s_2'(\delta,\eta,d)>0$, such that if $\sigma(I_0)<s_2'$, then $H_{d_N}(J_N)\cup F_N\subset H_{d_1}\big((a, b)\big)\subset H_{d_1}\big((a, v)\big)$, where $H_{d_1}\big((a, v)\big)$ is a simply connected half hyperbolic neighborhood for some $d_1=d_1(\delta,\eta,d)>0$, and is thus contained in a cone spanned at $v$ and
bounded by $\T$ and a ray which forms an angle $\beta_1=\beta_1(\delta,\eta,d)\in(0,\pi/2)$ with $\T$ (see also \eqref{equ:H-beta}).
By Lemmas \ref{lem:comparable} and \ref{lemma:swz-2}(a), there exist constants $C_3=C_3(\delta,\eta,d)>0$ and $C_4=C_4(\delta,\eta,d)>0$ such that
\begin{equation}\label{equ:diam-H-d-N}
{\rm diam}_{\C} (H_{d_N}(J_N)\cup F_N) <C_3 |J_{N}| <C_4\,{\rm dist}_{\C} (J_N, v).
\end{equation}
Note that in a small neighborhood of the critical point $c$, the map $G$ can be written as the composition of a power mapping like $z\mapsto (z-c)^{2\ell-1}$ and a quasiconformal mapping (see Figure \ref{Fig_crit-pullback-bounded}). By \eqref{equ:diam-H-d-N}, there exist $\beta=\beta(\delta,\eta,d)\in(0,\pi/2)$, $K_1=K_1(\delta,\eta,d)>1$ and $0 < r_{int} < r_{out}$ with $r_{out}/r_{int} < K_1$ such that
\begin{equation}\label{equ:V-U-K}
\widetilde{V}_{N+1} \subset U_{N+1}=\big\{z \,|\, r_{int} < |z - c| < r_{out} \text{ and } \angle\,zcx_2\in(\beta,\pi-\beta)\big\}.
\end{equation}
Let $0<C_1<1$ be a small number. If ${\rm dist}_{\C} (\widetilde{V}_{N+1}, \T)  \leqslant   C_1 r'$, then $r_{int} <C_1 C_5 r'$ for a constant $C_5=C_5(\delta,\eta,d)>1$. Since $r_{out} < K_1 r_{int}$, we get $0< C_2= K_1C_1C_5< 1$ and $0< r_{int} < r_{out} < C_2 r'$. Let $J'$ be the arc with $c$ being the middle point and $\sigma(J')  = r'$. There exists $s_2''=s_2''(r')>0$ such that if $\sigma(I_0)<s_2''$ and $0<C_1<1$ is small enough, then $U_{N+1} \subset H_d(J')$.
Then the claim follows by taking $s_2:=\min\{s_2',s_2''\}$.

The lemma follows immediately by assuming the claim. Indeed, if (1) is true, then similar to the argument for Case (i), there exists $s_3=s_3(M)>0$ such that if $\sigma(I_0)<s_3$, then $\widetilde{V}_{N+1}$ is sufficiently small and there exists a Jordan disk $U_{N+1}$ such that Part (a) holds.
If (2) is true, then by Lemma \ref{lemma:Pe}, there exists $\mu=\mu(\beta_1)=\mu(\delta,\eta,d)<1$ such that $\rho_\Omega(w) \leqslant\mu\,\rho_\Omega(G(w))|G'(w)|$ for $w\in V_{N+1}$. Moreover, by \eqref{equ:V-U-K}, there exists $K=K(\delta,\eta,d)>0$ such that $\widetilde{V}_{N+1}\subset U_{N+1}$ and $\widetilde{\diam}_\Omega(U_{N+1}) < K$. Hence Part (b) holds.
The proof is complete by setting $r'':=\min\{s_1,s_2,s_3\}$.
\end{proof}

\begin{lem}\label{lem:back-1}
For any $\varepsilon> 0$, there exists $r> 0$  such that if $V_0\subset H_d(I_0)$ for an arc $I_0\subset\T$ with $\sigma(I_0)< r$, then for any pullback sequence $\{V_n\}_{n\geqslant 0}$ of $V_0$, $\diam_{\EC} (V_n) < \varepsilon$ for all $n \geqslant 0$.
\end{lem}

\begin{proof}
By the assumption in the Main Lemma, there exists $D_0\geqslant 2$ such that for every $z\in\T$, there exists a small Jordan disk $W_0$ containing $z$, such that for any sequence $\{W_n\}_{n\geqslant 1}$ of pullbacks of $W_0$ in $\EC\setminus\overline{\D}$, we have $\deg(G^n: W_n\to W_0)\leqslant  D_0$.

For any given $\varepsilon>0$, there exist $\varepsilon'>0$ and $M>0$ depending only on $\varepsilon$ such that for any Jordan disks $U$ and $V$ in $\EC\setminus\overline{\D}$ with $\overline{V}\subset U$, we have
\begin{equation}\label{equ:criterion-epsilon}
\text{if } \diam_{\EC\setminus\overline{\D}}(V) < \varepsilon' \text{ or }\Mod(U\setminus \overline{V})> M/D_0, \text{ then } \diam_{\EC}(V) < \varepsilon.
\end{equation}
In particular, if $\diam_{\Omega}(V) < \varepsilon'$, then $\diam_{\EC}(V) < \varepsilon$. For such $\varepsilon'>0$, let $k_0\geqslant 3$ be the smallest integer such that
\begin{equation}\label{equ:mu-K}
\mu^{k_0-1} K < \varepsilon',
\end{equation}
where $0<\mu<1$ and $K>0$ are constants guaranteed by Lemma \ref{lemma:jumps}(b).

\medskip

Let $r_1':=\widetilde{r}$, where $\widetilde{r}=\widetilde{r}(\varepsilon)>0$ is the number introduced in Corollary \ref{cor:small}. For such $r_1'$, by Lemma \ref{lemma:jumps}, we get $r_2'=r_2'(M,r_1')\in (0, r_1']$. Inductively, we get a sequence of numbers:
\begin{equation}
r_{j+1}'=r_{j+1}'(M,r_j')\in(0,r_j'], \text{\quad where } 1\leqslant  j\leqslant  k_0.
\end{equation}
We claim that $r:= r_{k_0+1}'$ is the required number.

\medskip
In fact, let $V_0\subset H_d(I_0)$ with $\sigma(I_0)< r_{k_0+1}'\leqslant  r_1'$. If $\chi(I_0,\{V_n\}_{n\geqslant 0}) = +\infty$, then by Lemma \ref{lem:shrink}(a) we have $\diam_{\EC}(V_n) < \varepsilon$ for all $n\geqslant 0$.
Otherwise, we have $\chi(I_0,\{V_n\}_{n\geqslant 0}) = N_1<+\infty$, i.e., $V_{N_1+1}$ is the first jump off from $\T$ with respect to $I_0$. Then by Lemma \ref{lem:shrink}(b) we have
\begin{equation}\label{equ:diam-N-1}
\diam_{\EC}(V_n)<\varepsilon \text{\quad for all\quad}0\leqslant  n\leqslant  N_1.
\end{equation}
By Lemma \ref{lemma:jumps}, there exists a Jordan disk $U_{N_1+1}$ such that one of the following holds:
\begin{itemize}
\item[(1-i)] $V_{N_1+1}\subset U_{N_1+1}\subset\EC\setminus\overline{\D}$, $\Mod(U_{N_1+1}\setminus \overline{V}_{N_1+1})>M$ and $U_{N_1+1}\cap \MP(G)\subset \widehat{\MP}(G)$ with $\sharp (U_{N_1+1}\cap \MP(G))\leqslant  1$; or
\item[(1-ii)] $V_{N_1+1}\subset U_{N_1+1}\subset H_d(J_1')$ for an arc $J_1' \subset \T$ with $\sigma(J_1') =r_{k_0}'$, and
\begin{equation}
\widetilde{\diam}_\Omega(U_{N_1+1}) < K.
\end{equation}
\end{itemize}
If we are in Case (1-i), then $\diam_{\EC}(V_n) < \varepsilon$ for all $n\geqslant 0$ by \eqref{equ:criterion-epsilon} and \eqref{equ:diam-N-1}.

\medskip
Suppose we are in Case (1-ii). We consider all the possible improved $(\delta,\eta)$-admissible sequence beginning with $J_1'$.
For $n\geqslant N_1+1$, let $W_n$ be the component of $G^{-(n-N_1-1)}(U_{N_1+1})$ containing $V_n$. Then $G^{n-N_1-1}:W_n\to U_{N_1+1}$ is conformal by Lemma \ref{lem:take-pre}.
If $\chi(J_1',\{W_n\}_{n\geqslant N_1+1}) = +\infty$, then by Lemma \ref{lem:shrink}(a) we have
$\diam_{\EC}(W_n)< \varepsilon$ for all $n\geqslant N_1+1$. Combining \eqref{equ:diam-N-1}, we have $\diam_{\EC}(V_n)< \varepsilon$ for all $n\geqslant 0$.
Otherwise, we have $\chi(J_1',\{W_n\}_{n\geqslant N_1+1}) = N_2-N_1-1<+\infty$, i.e., $W_{N_2+1}$ is the first jump off from $\T$ with respect to $J_1'$. Then by Lemma \ref{lem:shrink}(b) we have
\begin{equation}\label{equ:diam-N-2}
\diam_{\EC}(V_n)<\varepsilon \text{\quad for all\quad}0\leqslant  n\leqslant  N_2.
\end{equation}
By Lemma \ref{lemma:jumps}, one of the following holds:
\begin{itemize}
\item[(2-i)] There exists a Jordan disk $U_{N_2+1}$ such that $W_{N_2+1}\subset U_{N_2+1}\subset\EC\setminus\overline{\D}$, $\Mod(U_{N_2+1}\setminus \overline{V}_{N_2+1})\geqslant \Mod(U_{N_2+1}\setminus \overline{W}_{N_2+1})>M$ and $U_{N_2+1}\cap \MP(G)\subset \widehat{\MP}(G)$ with $\sharp (U_{N_2+1}\cap \MP(G))\leqslant  1$; or
\item[(2-ii)] $W_{N_2+1}\subset H_d(J_2')$ for an arc $J_2' \subset \T$ with $\sigma(J_2') =r_{k_0-1}'$,
\begin{equation}
\rho_\Omega(w) \leqslant\mu\,\rho_\Omega(G(w))|G'(w)| \text{ for }w\in W_{N_2+1}
\text{\quad and\quad}
\widetilde{\diam}_\Omega(W_{N_2+1}) < \mu K.
\end{equation}
\end{itemize}
If we are in Case (2-i), then $\diam_{\EC}(V_n) < \varepsilon$ for all $n\geqslant 0$ by \eqref{equ:criterion-epsilon} and \eqref{equ:diam-N-2}.

\medskip
Suppose we are in Case (2-ii). Then we consider the value of $\chi(J_2',\{W_n\}_{n\geqslant N_2+1})$ and see whether it is equal to $+\infty$.
Let $k_0\geqslant 3$ be the integer introduced in \eqref{equ:mu-K}.
Inductively, repeating the above process finitely many times, we have the following two cases:

\medskip
\textbf{Case (I)}.
There exist an integer $2\leqslant  i\leqslant  k_0-1$, a sequence of arcs $J_1'$, $\cdots$, $J_i'$ in $\T$ and a sequence of integers $N_1$, $\cdots$, $N_i$ such that
\begin{equation}
\chi(J_{k}',\{W_n\}_{n\geqslant N_{k}+1}) =N_{k+1}-N_{k}-1 <+\infty \text{\quad for all\quad}1\leqslant  k\leqslant  i-1,
\end{equation}
and either
\begin{itemize}
\item $\chi(J_i',\{W_n\}_{n\geqslant N_i+1}) = +\infty$; or
\item $\chi(J_i',\{W_n\}_{n\geqslant N_i+1})=N_{i+1}-N_i-1 < +\infty$ and there exists a Jordan disk $U_{N_{i+1}+1}$ such that $W_{N_{i+1}+1}\subset U_{N_{i+1}+1}\subset\EC\setminus\overline{\D}$, $\Mod(U_{N_{i+1}+1}\setminus \overline{V}_{N_{i+1}+1})\geqslant \Mod(U_{N_{i+1}+1}\setminus \overline{W}_{N_{i+1}+1})>M$ and $U_{N_{i+1}+1}\cap \MP(G)\subset \widehat{\MP}(G)$ with $\sharp (U_{N_{i+1}+1}\cap \MP(G))\leqslant  1$.
\end{itemize}
In either case we have $\diam_{\EC}(V_n) < \varepsilon$ for all $n\geqslant 0$ by Lemma \ref{lem:shrink} and \eqref{equ:criterion-epsilon}.

\medskip
\textbf{Case (II)}.
There exist a sequence of arcs $J_1'$, $\cdots$, $J_{k_0}'$ in $\T$ and a sequence of integers $N_1$, $\cdots$, $N_{k_0}$ such that for all $1\leqslant  i\leqslant  k_0-1$,
\begin{itemize}
\item $\chi(J_i',\{W_n\}_{n\geqslant N_i+1}) =N_{i+1}-N_i-1 <+\infty$; and
\item $W_{N_{i+1}+1}\subset H_d(J_{i+1}')$ with $\sigma(J_{i+1}') =r_{k_0-i}'$,
\begin{equation}
\rho_\Omega(w) \leqslant\mu\,\rho_\Omega(G(w))|G'(w)| \text{ for }w\in W_{N_{i+1}+1}
\text{\quad and\quad}\widetilde{\diam}_\Omega(W_{N_{i+1}+1}) < \mu^{i} K.
\end{equation}
\end{itemize}
By \eqref{equ:mu-K} we have
\begin{equation}
\diam_\Omega(V_{N_{k_0}+1})\leqslant  \widetilde{\diam}_\Omega(W_{N_{k_0}+1}) < \mu^{k_0-1} K<\varepsilon'.
\end{equation}
By Lemma \ref{lem:shrink} and \eqref{equ:criterion-epsilon}, $\diam_{\EC} (V_n) < \varepsilon$ for all $n \geqslant 0$.
The proof is complete.
\end{proof}

\subsection{Proof of the Main Lemma$'$} \label{subsec:key-thm}

Let $f: \EC\to \EC$ be a rational map of degree at least $2$. A sequence $\{V_n\}_{n\geqslant 0}$ is called \textit{a pullback sequence} of a Jordan disk $V_0$ under $f$ if $V_{n+1}$ is a connected component of $f^{-1}(V_n)$ for all $n\geqslant 0$. The following result was proved in \cite[p.\,86]{LM97} (see also \cite{Man93}, \cite{TY96}).

\begin{lem}[Shrinking Lemma]\label{lem:semi-hyperbolic}
Let $D\geqslant 1$ and $U_0, V_0$ be two Jordan disks in $\EC$. Suppose $U_0$ is not contained in any rotation domain of $f$ and $V_0$ is compactly contained in $U_0$.
Then for any $\varepsilon>0$, there exists an $N\geqslant 1$ such that for any pullback sequence $\{U_n\}_{n\geqslant 0}$ satisfying $\deg(f^n:U_n\to U_0)\leqslant  D$, $\diam_{\EC}(V_n)<\varepsilon$ for all $n\geqslant N$, where $V_n$ is any component of $f^{-n}(V_0)$ contained in $U_n$.
\end{lem}

For the map $f$ in the Main Lemma, let $\MP(G)$ be the postcritical set of the quasi-Blaschke model $G$ of $f$ defined in \eqref{equ:P-G}.

\begin{lem}\label{lem:eventually-short}
Let $V_0\subset \EC\setminus \overline{\D}$ be a Jordan disk such that $\overline{V}_0 \cap \MP(G)$ is a subarc on $\T$.
Then for any $r>0$, there exists $N\geqslant 1$ such that for any pullback sequence $\{V_n\}_{n\geqslant 0}$ of $V_0$ in $\EC\setminus\overline{\D}$, we have $\sigma(\overline{V}_n\cap\T)<r$ for all $n\geqslant N$.
\end{lem}

\begin{proof}
By the assumption, $I_0=\overline{V}_0\cap \MP(G)$ is a proper subarc of $\T$. Note that the forward orbit of any critical point $c_0\in\Crit(G)\cap\T$ is dense in $\T$ and in particular dense in $I_0$.
For any $r>0$, there exists $N=N(r)\geqslant 1$ such that
\begin{equation}\label{equ:max-I}
\max\left\{\sigma(I')
\left|
\begin{array}{l}
I' \text{ is a connected component}\\
\text{of } I_0\setminus\{G^n(c_0):1\leqslant n\leqslant N\}
\end{array}
\right.
\right\}<r.
\end{equation}
Denote $I_n=\overline{V}_n\cap \T$ for $n\geqslant 1$, where $\{V_n\}_{n\geqslant 0}$ is any pullback sequence of $V_0$ in $\EC\setminus\overline{\D}$. Then $I_n$ is either an arc, a singleton or empty. If $I_n$ is a singleton or empty, then $I_m$ is also for all $m>n$. If $I_n$ is an arc containing a critical value $v=G(c)$ which divides $I_n$ into two subarcs $I_n'$ and $I_n''$, where $c\in\Crit(G)\cap\T$, then
\begin{equation}\label{equ:sigma-max}
\sigma(I_{n+1})\leqslant\max\{\sigma(I_n'),\sigma(I_n'')\}.
\end{equation}
By \eqref{equ:max-I} and \eqref{equ:sigma-max} we have $\sigma(I_n)<r$ for all $n\geqslant N$.
\end{proof}

\begin{proof}[Proof of the Main Lemma$'$]
Let $V_0\subset \EC\setminus \overline{\D}$ be a Jordan disk such that $\overline{V}_0 \cap \MP(G)\neq\emptyset$ and $\overline{V}_0 \cap \MP(G) \subset \T$.
Since $\dist_{\C}(\MP(G)\setminus\T,\T)>0$, we claim that $\EC\setminus(\overline{\D}\cup \overline{V}_0)$ has at most finitely many connected components, say $U_1$, $\cdots$, $U_m$, such that
\begin{equation}\label{equ:U-i}
U_i\cap\big(\MP(G)\setminus\T\big)\neq\emptyset, \text{\quad for every }1\leqslant i\leqslant m.
\end{equation}
Indeed, otherwise there are infinitely many $U_i$'s satisfying \eqref{equ:U-i} and there exists $z_0\in\T\cap\partial V_0$ such that $\partial V_0$ is not locally connected at $z_0$, which is a contradiction and hence the claim holds.
By filling all (except at most one) connected components of $\EC\setminus(\overline{\D}\cup \overline{V}_0)$ which are disjoint with $\MP(G)$, we conclude that $V_0$ is contained in a Jordan disk $V_0'$ such that $\overline{V_0'} \cap \MP(G)$ is the union of finitely many subarcs (singletons are allowed) on $\T$.
Without loss of generality, we assume that $V_0=V_0'$ and $\overline{V}_0 \cap \MP(G)$ is the union of finitely many subarcs on $\T$ which are not singletons.

By considering a homeomorphism defined from $\overline{\D}$ onto $\overline{V}_0$, it is easy to see that $V_0$ can be written as the union of finitely many Jordan disks $V_0^1$, $\cdots$, $V_0^\ell$ such that $\overline{V_0^i} \cap \MP(G)$ is a subarc on $\T$ for each $1\leqslant i\leqslant \ell$.
For any sequence of pullbacks $\{V_n\}_{n\geqslant 0}$ of $V_0$ under $G$ in $\EC\setminus\overline{\D}$, $G^n:V_n\to V_0$ is conformal and $G^n:\overline{V}_n\to \overline{V}_0$ is a homeomorphism by Lemma \ref{lem:take-pre}.
So without loss of generality, we assume that $\ell=1$ and that $V_0\subset \EC\setminus \overline{\D}$ is a Jordan disk such that $\overline{V}_0 \cap \MP(G)$ is a subarc on $\T$.

\medskip
For given $\varepsilon>0$, let $r=r(\varepsilon/2)>0$ be the number guaranteed by Lemma \ref{lem:back-1}. By Lemma \ref{lem:eventually-short}, there exists $N_0=N_0(\varepsilon)\geqslant 1$ such that for any pullback sequence $\{V_n\}_{n\geqslant 0}$ of $V_0$ in $\EC\setminus\overline{\D}$,
\begin{equation}
\sigma(\overline{V}_n\cap\T)<r/2 \text{\quad for any }n\geqslant N_0.
\end{equation}
Let $\varsigma:\overline{\D}\to\overline{V}_{N_0}$ be a homeomorphism. By the uniform continuity of $\varsigma$, the Jordan disk $V_{N_0}$ can be written as the union of finitely many small Jordan disks, such that some of them are covered by $H_d(I_0)$ and the rest of them are $W_1$, $\cdots$, $W_{M_0}$, where
\begin{itemize}
\item $I_0\subset\T$ is an arc containing $\overline{V}_{N_0}\cap\T$ with $\sigma(I_0)=3r/4$; and
\item Each $W_i$ with $1\leqslant i\leqslant M_0$ is compactly contained in $\EC\setminus\overline{\D}$ satisfying $\overline{W}_i \cap \MP(G)\subset\widehat{\MP}(G)$ with $\sharp\,(\overline{W}_i \cap \MP(G))\leqslant  1$ and $\widehat{\MP}(G)$ is defined in \eqref{equ:P-G-hat}.
\end{itemize}

Since there are only finitely many choices of $V_{N_0}$ among all pullback sequences of $V_0$ in $\EC\setminus\overline{\D}$, it follows that the number $M_0\geqslant 1$ above can be chosen uniformly for all pullback sequences of $V_0$.
Note that every $W_i$ is compactly contained in a bigger Jordan disk $\widetilde{W}_i$, such that for any $n> N_0$, the degree of the restriction of $G^{n-N_0}$ on any connected component $\widetilde{W}_i^n$ of $G^{-(n-N_0)}(\widetilde{W}_i)$ in $\EC\setminus\overline{\D}$ has uniform upper bound, where $\widetilde{W}_i^n$ satisfies $G^k(\widetilde{W}_i^n)\subset\EC\setminus\overline{\D}$ for all $1\leqslant  k\leqslant  n-N_0$.
By Lemma \ref{lem:semi-hyperbolic}, there exists a uniform $N> N_0$, such that if $n\geqslant N$, every component of $G^{-(n-N_0)}(W_i)$ in $\widetilde{W}_i^n$ has spherical diameter less than $\varepsilon/(3M_0)$. Combining this with Lemma \ref{lem:back-1}, it follows that if $n\geqslant N$, for any pullback sequence $\{V_n\}_{n\geqslant 0}$ of $V_0$ in $\EC\setminus\overline{\D}$, we have
\begin{equation}
\diam_{\EC}(V_n)<\frac{\varepsilon}{2}+\frac{\varepsilon}{3M_0}\cdot M_0<\varepsilon.
\end{equation}
The proof is complete.
\end{proof}

\section{Proof of the Main Theorem}\label{sec-rational-case}

To prove the local connectivity of a Julia set in $\EC$, we use the following criterion (see {\cite[Theorem 4.4, p.\,113]{Why42}}).

\begin{lem}[LC criterion]\label{lem:LC-criterion}
A compact subset $X$ in $\EC$ is locally connected if and only if the following two conditions hold:
\begin{enumerate}
\item  The boundary of every component of $\EC\setminus X$ is locally connected; and
\item  For any given $\varepsilon>0$, there are only finitely many components of $\EC\setminus X$ whose spherical diameters are greater than $\varepsilon$.
\end{enumerate}
\end{lem}

Let $f$ be a rational map in the Main Theorem. By the assumption, the periodic Fatou components of $f$ can only be attracting basins (including super-attracting) or bounded type Siegel disks and the periodic points in $J(f)$ are all repelling.
Moreover, $f$ has at least one cycle of bounded type Siegel disks.

\subsection{Local connectivity of the boundaries of attracting basins}

In this section, we assume that $f$ has at least one attracting basin. The main goal in this section is to prove the following result:

\begin{prop}\label{prop:attracting}
The boundary of every immediate attracting basin of $f$ is locally connected.
\end{prop}

Note that $J(f)$ is connected by the assumption. Without loss of generality, we assume that
\begin{itemize}
\item[(i)] All periodic Fatou components of $f$ have period one and they consist of $p_0\geqslant 1$ fixed Siegel disks $\Delta_p$ with $1\leqslant  p\leqslant  p_0$ and $q_0\geqslant 1$ fixed immediate attracting basins $A_q$ with $1\leqslant  q\leqslant  q_0$ (Considering an appropriate iteration of $f$ if necessary); and
\item[(ii)] $f$ is postcritically finite in the attracting basins (By a standard quasiconformal surgery \cite[Chapter 4]{BF14a} since every Fatou component of $f$ is simply connected).
\end{itemize}
By the assumption (ii), we conclude that every immediate attracting basin $A_q$ is super-attracting and contains exactly one critical point $a_q$ (without counting multiplicity).
Let $\MP(f)$ be the postcritical set of $f$. Then the following set is finite or empty:
\begin{equation}\label{equ:post-finite}
\MP_1(f):=\MP(f)\setminus\Big(\bigcup_{p=1}^{p_0} \overline{\Delta}_p\cup \bigcup_{q=1}^{q_0}\{a_q\}\Big).
\end{equation}
For every $1\leqslant  q\leqslant  q_0$, there exists a small quasi-disk $B_q$ containing $a_q$ such that $f(\overline{B}_q)\subset B_q$. Denote
\begin{equation}\label{equ:W}
W:=\EC\setminus\Big(\MP(f)\cup\bigcup_{p=1}^{p_0}\overline{\Delta}_p \cup\bigcup_{q=1}^{q_0}\overline{B}_q\Big)
\text{\quad and\quad}
W_1:=W\cup\MP_1(f).
\end{equation}
Then $f^{-1}(W)\subset W$ and $f^{-1}(W_1)\subset W_1$. According to \cite{Zha11}, every $\partial\Delta_p$ is a quasi-circle and $\overline{\Delta}_{p'}\cap\overline{\Delta}_{p''}=\emptyset$ for any different integers $p',p''\in[1,p_0]$.

\begin{lem}\label{lem:contrac}
There exists $\delta_0>0$ such that for any $\varepsilon>0$, there exists an integer $N\geqslant 1$, such that for any Jordan disk $V_0$ in $W_1$ with $\diam_{\EC}(V_0)< \delta_0$, we have $\diam_{\EC}(V_n)<\varepsilon$ for all $n\geqslant N$, where $V_n$ is any component of $f^{-n}(V_0)$.
\end{lem}

\begin{proof}
Note that $W_1$ is a domain whose boundary consists of $p_0+q_0$ quasi-circles. Then there exist finitely many Jordan disks $\{U_k:1\leqslant  k\leqslant  k_0\}$ in $W_1$ and a small number $\delta_0>0$ such that
\begin{itemize}
\item[(i)] $W_1=\bigcup_{k=1}^{k_0} U_k$;
\item[(ii)] For every $1\leqslant  k\leqslant  k_0$, $\sharp(U_k\cap \MP(f))=\sharp(U_k\cap \MP_1(f))\leqslant  1$ or $\overline{U}_k\cap \MP(f)\subset\partial \Delta_p$ for some $1\leqslant  p\leqslant  p_0$; and
\item[(iii)] Any Jordan disk $V_0$ in $W_1$ with $\diam_{\EC}(V_0)< \delta_0$ is contained in some $U_k$ with $1\leqslant  k\leqslant  k_0$.
\end{itemize}
Indeed, let $D$ be a domain in $\EC$ whose boundary consists of $p_0+q_0$ spherical circles and let $\varsigma:\overline{D}\to\overline{W}_1$ be a homeomorphism. By the finiteness of $\MP_1(f)$ in $W_1$, there exists a partition $D=\bigcup_{k=1}^{k_0} \widetilde{U}_k$ such that (i) and (ii) hold for $\{U_k=\varsigma(\widetilde{U}_k):1\leqslant k\leqslant k_0\}$, where each $\widetilde{U}_k$ is a Jordan disk. Moreover, $\{\widetilde{U}_k:1\leqslant k\leqslant k_0\}$ can be chosen further such that if $I$ is an arc on $\partial D$ with the Euclidean length less than a constant $\nu_0>0$, then $I$ is contained in $\partial\widetilde{U}_k\cap\partial D$ for some $1\leqslant k\leqslant k_0$.
Adopting the proof by contradiction which is similar to the Lebesgue number theorem, there exists $\widetilde{\delta}_0>0$ such that any Jordan disk $\widetilde{V}_0$ in $D$ with $\diam_{\EC}(\widetilde{V}_0)< \widetilde{\delta}_0$ is contained in some $\widetilde{U}_k$ with $1\leqslant  k\leqslant  k_0$.
Then the existence of $\delta_0>0$ in (iii) follows by the uniform continuity of $\varsigma^{-1}:\overline{W}_1\to\overline{D}$.

Let $\varepsilon>0$ be given. By Lemma \ref{lem:semi-hyperbolic} and the Main Lemma, there exists $N\geqslant 1$ such that for all $n\geqslant N$, any component of $f^{-n}(U_k)$ with $1\leqslant  k\leqslant  k_0$ has spherical diameter less than $\varepsilon$.
\end{proof}

Let $\N_+:=\{1,2,3,\cdots\}$ be the set of all positive integers. We use the following definition from \cite[Appendix E]{Mil06}.

\begin{defi}[Orbifolds]
A pair $(S,\nu)$ consisting of a Riemann surface $S$ and a ramification function $\nu:S\to\N_+$ which takes the value $\nu(z)=1$ except on a discrete closed subset is called a \textit{Riemann surface orbifold} (\textit{orbifold} in short). A point $z\in S$ is called a \textit{ramified point} if $\nu(z)\geqslant 2$.
\end{defi}

A map $h:(\widetilde{S},\widetilde{\nu})\to (S,\nu)$ between two orbifolds is called a \textit{branched covering} if $h:\widetilde{S}\to S$ is a branched covering and $\nu(h(\zeta))=\deg_{\zeta}(h)\,\widetilde{\nu}(\zeta)$ for all $\zeta\in \widetilde{S}$, where $\deg_{\zeta}(h)$ is the local degree of $h$ at $\zeta$. In particular, if $\widetilde{S}$ is simply connected and $\widetilde{\nu}\equiv 1$ on $\widetilde{S}$, then $h:\widetilde{S}\to (S,\nu)$ is called a \textit{universal covering}.

\medskip
For the hyperbolic Riemann surface $W_1$ introduced in \eqref{equ:W}, we define an orbifold $(W_1,\nu)$ as following.
If $z\in W_1\setminus\MP_1(f)$, define $\nu(z)=1$. If $z\in\MP_1(f)$, define $\nu(z)$ as the least common multiple of $\{\deg_{\zeta}(f)\,\nu(\zeta): \zeta\in f^{-1}(z)\}$. Since $\MP_1(f)$ is a finite set, it is easy to see that $(W_1,\nu)$ is an orbifold.

According to \cite[Theorem E.1]{Mil06}, the orbifold $(W_1,\nu)$ has a universal covering $\pi:\D\to (W_1,\nu)$. Let $W_1'\subset W_1$ be any connected component of $f^{-1}(W_1)$. Note that $f:W_1'\to W_1$ is a branched covering between Riemann surfaces and $\nu(f(\zeta))$ is some integral multiple of $\deg_\zeta(f)\nu(\zeta)$ for all $\zeta\in W_1'$. By \cite[Section 19, pp.\,212--213]{Mil06}, $f^{-1}$ lifts to a single-valued holomorphic map $\widetilde{F}:\D\to\D$ such that
\begin{equation}
f\circ\pi\circ \widetilde{F}(w)=\pi(w) \text{\quad for all }w\in\D.
\end{equation}
Since $f$ has repelling periodic points in $W_1$, it follows that $\widetilde{F}:\D\to\D$ decreases the hyperbolic metric $\rho_{\D}(w)|dw|$ in $\D$. The universal covering $\pi:\D\to W_1$ induces an \textit{orbifold metric} $\sigma(z)|dz|$ in $W_1$ satisfying
\begin{equation}
\sigma(\pi(w))|\pi'(w)|=\rho_{\D}(w).
\end{equation}

The following result shows that $f^{-1}$ contracts the orbifold metric $\sigma(z)|dz|$ in $W_1$. See \cite[Section 19, p.\,213]{Mil06}.

\begin{lem}\label{lem:orbifold}
For any compact subset $K$ of $f^{-1}(W_1)$, there exists a number $\lambda=\lambda(K)>1$ such that for any $z\in K$, if $z$ and $f(z)$ are not ramified points, then
\begin{equation}
\sigma(f(z))|f^{\#}(z)|\geqslant \lambda\,\sigma(z),
\end{equation}
where $f^{\#}$ denotes the spherical derivative.
\end{lem}

For a given immediate super-attracting basin $A:=A_q$, where $1\leqslant  q\leqslant  q_0$, there exists a conformal map $\psi: \EC\setminus\overline{\D}\to A$ which conjugates $\zeta\mapsto \zeta^u: \EC\setminus\overline{\D}\to \EC\setminus\overline{\D}$ to $f: A\to A$ for some integer $u\geqslant 2$.
For $r>0$ and $\theta\in\R/\Z$, we denote
\begin{equation}
R_\theta(r)=E_r(\theta):=\psi(e^{r+2\pi\ii\theta}).
\end{equation}
The images
\begin{equation}
R_\theta :=R_\theta\big((0,+\infty)\big) \text{\quad and\quad} E_r:=E_r(\R/\Z)
\end{equation}
are the \textit{internal ray} of angle $\theta$ and the \textit{equipotential curve} of potential $r>0$ in $A$ respectively.

\begin{defi}[Ray segments]
For every $\theta\in\R/\Z$ and integer $n\geqslant 1$, the curve
\begin{equation}
R_{\theta,n}:=R_\theta\big([\tfrac{1}{u^n},1]\big)
\end{equation}
is called a \textit{ray segment} in $A$.
\end{defi}

\begin{lem}\label{lem:essential-bd}
For any $\varepsilon>0$, there exists $C=C(\varepsilon)>0$ such that for any ray segment $R_{\theta,n}$ with $\theta\in\R/\Z$ and $n\geqslant 1$, there exists a continuous curve $\widetilde{R}_n$ which is homotopic to $R_{\theta,n}$ in $\EC\setminus\MP(f)$ relative to their end points,  and moreover, $\widetilde{R}_n$ is the union of two continuous curves $\widetilde{R}_n^{ess}$ and $\widetilde{R}_n^{end}$ satisfying
\begin{equation}
l_{\EC}(\widetilde{R}_n^{ess})< C \text{\quad and\quad} \diam_{\EC}(\widetilde{R}_n^{end})<\varepsilon,
\end{equation}
where $l_{\EC}(\cdot)$ denotes the length with respect to the spherical metric.
\end{lem}

\begin{proof}
Since the closures of the Siegel disks $\{\overline{\Delta}_p:1\leqslant  p\leqslant  p_0\}$ of $f$ are pairwise disjoint quasi-disks, there exists a quasiconformal mapping $\Psi:\EC\to\EC$ such that the restriction $\Psi:\EC\setminus \bigcup_{p=1}^{p_0}\overline{\Delta}_p\to \EC\setminus\bigcup_{p=1}^{p_0}\overline{D}_p$ is conformal, where $\overline{D}_1$, $\cdots$, $\overline{D}_{p_0}$ are pairwise disjoint closed spherical disks. Define
\begin{equation}
F:=\Psi\circ f\circ \Psi^{-1}, \quad \MP(F):=\Psi(\MP(f)) \text{\quad and\quad} \MP_1(F):=\Psi(\MP_1(f)).
\end{equation}
For convenience, we use the same notations in the dynamical plane of $f$ to denote the corresponding objects in the dynamical plane of the quasiregular map $F$ under $\Psi$. For example, we use $R_{\theta,n}$ (not $\Psi(R_{\theta,n})$) to denote the ray segment of $F$ etc. It suffices to prove that the lemma holds in the dynamical plane of $F$ and that $\widetilde{R}_n^{ess}$ lies in the outside of a neighborhood (depending only on the given $\varepsilon$) of $\bigcup_{p=1}^{p_0}\overline{D}_p$ in which $\Psi$ is conformal with bounded distortion.

\medskip
Without loss of generality, we assume that the corresponding domains $W$ and $W_1$ (see \eqref{equ:W}) in the dynamical plane of $F$ have the following form:
\begin{equation}\label{equ:W-new}
W=\EC\setminus\Big(\MP(F)\cup\bigcup_{p=1}^{p_0}\overline{D}_p \cup\bigcup_{q=1}^{q_0}\overline{B_{\delta'}(a_q)}\Big) \text{\quad and\quad} W_1=W\cup\MP_1(F),
\end{equation}
where $F(\overline{B_{\delta'}(a_q)})\subset B_{\delta'}(a_q)=\{z\in\EC:\dist_{\EC}(z,a_q)<\delta'\}$.
Let $E_1$ be the equipotential curve in $A$ with potential $1$. Decreasing the size of $B_{\delta'}(a_q)$ if necessary, we assume that $E_1\subset W$. Then there exists a constant $C_0>0$ such that
\begin{equation}\label{equ:l-W}
l_\sigma(R_{\theta,1})< C_0 \text{\quad for all }\theta\in \R/\Z,
\end{equation}
where $l_\sigma(\cdot)$ denotes the length with respect to the orbifold metric $\sigma(z)|dz|$ in $W_1$.

\begin{figure}[!htpb]
  \setlength{\unitlength}{1mm}
  \centering
  \includegraphics[width=0.85\textwidth]{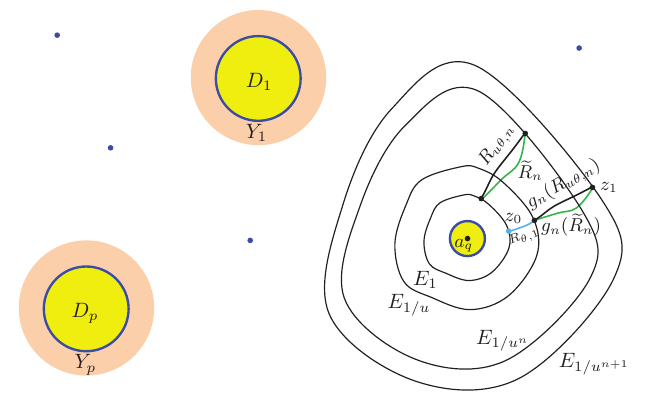}
  \caption{The dynamical plane of $F=\Psi\circ f\circ \Psi^{-1}$. The boundary $\partial W$ is marked by blue (including some spherical circles and discrete points). The annuli $\{Y_p:1\leqslant  p\leqslant  p_0\}$ and some curves in an immediate super-attracting basin are drawn. }
  \label{Fig-homotopy}
\end{figure}

\medskip
It suffices to prove the lemma for $\varepsilon>0$ which is sufficiently small.
For $1\leqslant  p\leqslant  p_0$, we define
\begin{equation}
Y_p:=\big\{z\in W: \dist_{\EC}(z,\partial D_p)<\varepsilon/2\big\}.
\end{equation}
For two continuous curves $\gamma_1$ and $\gamma_2$ in $W$ with the same end points, we denote
\begin{equation}
\gamma_1\simeq \gamma_2 \text{\quad in~~}W,
\end{equation}
if $\gamma_1$ is homotopic to $\gamma_2$ in $W$ relative to their end points.

\medskip
We prove the lemma by induction. For $n=1$, we take
\begin{equation}
\widetilde{R}_n=\widetilde{R}_n^{ess}:=R_{\theta,1} \text{\quad and\quad} \widetilde{R}_n^{end}=\emptyset.
\end{equation}
Let $D(E_1)$ be the component of $\EC\setminus E_1$ containing the super-attracting fixed point $a_q$. We assume that $\varepsilon>0$ is small enough such that $\overline{D(E_1)}\cap \overline{Y}_p=\emptyset$ for all $1\leqslant  p\leqslant  p_0$.
Suppose there exists $n\geqslant 1$ such that
\begin{itemize}
  \item[(i)] For any ray segment $R_{\theta,n}$ with $\theta\in\R/\Z$, there exists a continuous curve $\widetilde{R}_n\subset W\setminus D(E_1)$ such that $\widetilde{R}_n \simeq R_{\theta,n}$ in $W$, where $\widetilde{R}_n$ can be written as the union of two continuous curves $\widetilde{R}_n^{ess}$ and $\widetilde{R}_n^{end}$ (may be empty);
  \item[(ii)] $l_{\sigma}(\widetilde{R}_n^{ess})< C'$ and $\widetilde{R}_n^{ess}\cap Y_p=\emptyset$ for all $1\leqslant  p\leqslant  p_0$, where $C'=C'(\varepsilon)\in [C_0,+\infty)$ is a number which will be specified later; and
  \item[(iii)] $\diam_{\EC}(\widetilde{R}_n^{end})<\varepsilon$.
\end{itemize}

Let $g_n$ be the inverse branch of $F$ which maps $R_{u\theta,n}$ to $R_\theta([\tfrac{1}{u^{n+1}},\tfrac{1}{u}])$, where $u=\deg(F:A\to A)\geqslant 2$. Then we have
\begin{equation}
R_{\theta,n+1}=R_{\theta,1}\cup g_n(R_{u\theta,n}).
\end{equation}
Note that $g_n$ can be analytically extended to any continuous curve in $W$ which is homotopic to $R_{u\theta,n}$.
Therefore, for $\widetilde{R}_n$ obtained in the inductive assumption which satisfies $\widetilde{R}_n=\widetilde{R}_n^{ess}\cup\widetilde{R}_n^{end} \simeq R_{u\theta,n}$ in $W$, $g_n(\widetilde{R}_n)$ is well-defined and homotopic to $g_n(R_{u\theta,n})$. Let
\begin{equation}
\gamma_{n+1}':= R_{\theta,1}\cup g_n(\widetilde{R}_n^{ess}) \text{\quad and\quad} \gamma_{n+1}'':=g_n(\widetilde{R}_n^{end}).
\end{equation}
Then
\begin{equation}
\gamma_{n+1}:=\gamma_{n+1}'\cup \gamma_{n+1}'' \simeq R_{\theta,n+1} \text{\quad in } W.
\end{equation}
Since $\widetilde{R}_n\subset W\setminus D(E_1)$, it follows that $\gamma_{n+1}'$ and $\gamma_{n+1}''$ are continuous curves in $W\setminus D(E_1)$.
See Figure \ref{Fig-homotopy}.
In the following we deform $\gamma_{n+1}$ in $W\setminus D(E_1)$ such that the inductive assumptions hold for step $n+1$.

\medskip
By (i), (ii), \eqref{equ:l-W} and Lemma \ref{lem:orbifold}, there exists $\mu=\mu(\varepsilon)\in (0,1)$ which is independent of $n$ such that
\begin{equation}\label{equ:l-W-mu}
l_{\sigma}\big(\gamma_{n+1}'\big)< C_0 +\mu C'.
\end{equation}
By (iii), there exists a small number $\varepsilon'>0$ depending only on $\varepsilon$ such that
\begin{equation}\label{equ:end-bd}
\diam_{\EC}(\gamma_{n+1}'')<\varepsilon'.
\end{equation}

Let $z_0=R_\theta(1)$ and $z_1=R_\theta(1/u^{n+1})$ be the two end points of $\gamma_{n+1}$. We have the following two cases:
\begin{itemize}
\item[(a)] $z_1\not\in Y_p$ for any $1\leqslant  p\leqslant  p_0$; or
\item[(b)] $z_1\in Y_p$ for some $1\leqslant  p\leqslant  p_0$.
\end{itemize}

\medskip
Suppose Case (a) holds. Note that $\varepsilon>0$ is assumed to be sufficiently small. By \eqref{equ:l-W-mu} and \eqref{equ:end-bd}, there exist a number $C_1=C_1(\varepsilon)>0$ and a continuous curve $\widetilde{\gamma}_{n+1}$ such that $\widetilde{\gamma}_{n+1}\simeq \gamma_{n+1}$ in $W\setminus D(E_1)$, $\widetilde{\gamma}_{n+1}\cap Y_p=\emptyset$ for any $1\leqslant  p\leqslant  p_0$ and
\begin{equation}
l_{\sigma}\big(\widetilde{\gamma}_{n+1}\big)\leqslant  l_{\sigma}\big(\gamma_{n+1}'\big)+C_1<C_0+\mu C'+C_1.
\end{equation}
In this case we define $\widetilde{R}_{n+1}:=\widetilde{R}_{n+1}^{ess}\cup \widetilde{R}_{n+1}^{end}$, where
\begin{equation}
\widetilde{R}_{n+1}^{ess}:=\widetilde{\gamma}_{n+1} \text{\quad and\quad} \widetilde{R}_{n+1}^{end}:=\emptyset.
\end{equation}
Then the induction at step $n+1$ is finished by setting $C':=(C_0+C_1)/(1-\mu)$.

\medskip
Suppose Case (b) holds.
Let $z_*\in \partial Y_p\setminus\partial D_p$ be the point in $W$ such that the segment $[z_*,z_1]$ is perpendicular to $\partial Y_p$ (see Figure \ref{Fig-homo-seg}). Note that $\varepsilon'$ in \eqref{equ:end-bd} depends on $\varepsilon$. Hence there exist a number $C_2=C_2(\varepsilon)>0$ and a continuous curve $\widehat{\gamma}_{n+1}\subset W\setminus D(E_1)$ such that $\widehat{\gamma}_{n+1}\cup[z_*,z_1]\simeq \gamma_{n+1}$ in $W$, $\widehat{\gamma}_{n+1}\cap Y_p=\emptyset$ for any $1\leqslant  p\leqslant  p_0$ and
\begin{equation}
l_{\sigma}\big(\widehat{\gamma}_{n+1}\big)<C_0 +\mu C'+C_2.
\end{equation}
In this case we define $\widetilde{R}_{n+1}:=\widetilde{R}_{n+1}^{ess}\cup \widetilde{R}_{n+1}^{end}$, where
\begin{equation}
\widetilde{R}_{n+1}^{ess}:=\widehat{\gamma}_{n+1} \text{\quad and\quad} \widetilde{R}_{n+1}^{end}:=[z_*,z_1].
\end{equation}
Note that $\diam_{\EC}(\widetilde{R}_{n+1}^{end})<\varepsilon$. Then the induction at step $n+1$ is finished by setting $C':=(C_0+C_2)/(1-\mu)$.
Denote $C_3:=\max\{C_1,C_2\}$.

\begin{figure}[!htpb]
  \setlength{\unitlength}{1mm}
  \centering
  \includegraphics[width=0.7\textwidth]{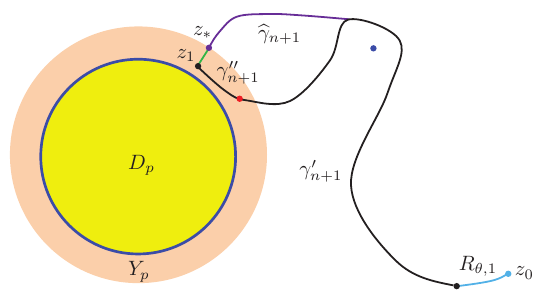}
  \caption{A homotopic deformation of $\gamma_{n+1}=\gamma_{n+1}'\cup \gamma_{n+1}''$. }
  \label{Fig-homo-seg}
\end{figure}

\medskip
To sum up, we have proved that for any $\theta\in\R/\Z$ and any $n\geqslant 1$, there exists a continuous curve $\widetilde{R}_n$ such that $\widetilde{R}_n\simeq R_{\theta,n}$ in $W$, where $\widetilde{R}_n=\widetilde{R}_n^{ess}\cup\widetilde{R}_n^{end}$ with $\widetilde{R}_n^{ess}\cap Y_p=\emptyset$ for all $1\leqslant  p\leqslant  p_0$ and
\begin{equation}
l_{\sigma}(\widetilde{R}_n^{ess})< C'=(C_0+C_3)/(1-\mu) \text{\quad and\quad} \diam_{\EC}(\widetilde{R}_n^{end})<\varepsilon.
\end{equation}
Note that there exists a constant $b>0$ such that $\sigma(z)>b$ for all $z\in W$. Hence we have $l_{\EC}(\widetilde{R}_n^{ess})<C:=C'/b$.
The proof is complete.
\end{proof}

\begin{proof}[Proof of Proposition \ref{prop:attracting}]
Let $A$ be an immediate super-attracting basin and suppose $f: A\to A$ has degree $u\geqslant 2$.
For every $n\geqslant 0$, we parameterize the equipotential curves as
\begin{equation}
\Gamma_n(\theta):=E_{\frac{1}{u^n}}(\theta)=\psi(e^{\frac{1}{u^n}+2\pi\ii\theta}), \text{\quad where }\theta\in\R/\Z.
\end{equation}
Hence every ray segment $R_{\theta,n}$ connects $\Gamma_0(\theta)$ with $\Gamma_n(\theta)$, where $\theta\in\R/\Z$ and $n\geqslant 1$.

For any given $\varepsilon>0$, let $C=C(\varepsilon)>0$ and $\widetilde{R}_n=\widetilde{R}_n^{ess}\cup \widetilde{R}_n^{end}\simeq R_{\theta',n}$ be determined by Lemma \ref{lem:essential-bd}, where $\theta'\in\R/\Z$. For any $m\geqslant 1$, the map $f^m:R_\theta([\tfrac{1}{u^{n+m}},\tfrac{1}{u^m}])\to R_{\theta',n}$ with $\theta'=u^m\theta$ is a homeomorphism, where $R_\theta([\tfrac{1}{u^{n+m}},\tfrac{1}{u^m}])$ is a continuous curve in $W$ connecting $\Gamma_{n+m}(\theta)$ with $\Gamma_m(\theta)$. Let $\widetilde{R}_{n,m}$ be the pullback curve of $\widetilde{R}_n$ under $f^m$ such that $\widetilde{R}_{n,m}\simeq R_\theta([\tfrac{1}{u^{n+m}},\tfrac{1}{u^m}])$ in $W$.

Since $l_{\EC}(\widetilde{R}_n^{ess})<C=C(\varepsilon)$, there exist a continuous parametrization $\widetilde{R}_n^{ess}(t)$ with $t\in[0,1]$ and finitely many points $0=t_0<t_1<\cdots<t_M=1$ with  $M=M(\varepsilon)\geqslant 1$ such that
\begin{equation}
l_{\EC}\big(\widetilde{R}_n^{ess}([t_i,t_{i+1}])\big)<\delta_0/2, \text{\quad for all } 0\leqslant  i\leqslant  M-1,
\end{equation}
where $\delta_0>0$ is the constant introduced in Lemma \ref{lem:contrac}. Note that each subcurve $\widetilde{R}_n^{ess}([t_i,t_{i+1}])$ is contained in a Jordan disk $U_i\subset W_1$ with $\diam_{\EC}(U_i)<\delta_0$, where $0\leqslant  i\leqslant  M-1$.
Without loss of generality, we assume that $\varepsilon>0$ is small such that $\diam_{\EC}(\widetilde{R}_n^{end})<\varepsilon<\delta_0/2$.

By Lemma \ref{lem:contrac}, there exists an integer $N\geqslant 1$ such that for all $m\geqslant N$, every component of $f^{-m}(U_i)$ and $f^{-m}(\widetilde{R}_n^{end})$ have spherical diameter less than $\varepsilon/(2M)$. Hence if $m\geqslant N$, for all $\theta\in\R/\Z$ we have
\begin{equation}
\dist_{\EC}\big(\Gamma_{m}(\theta),\Gamma_{m+n}(\theta)\big)\leqslant  \diam_{\EC}\big(\widetilde{R}_{n,m}\big)<M\cdot\frac{\varepsilon}{2M}+\frac{\varepsilon}{2M}<\varepsilon.
\end{equation}
This implies that $\{\Gamma_n(\theta):\theta\in\R/\Z\}$ converges uniformly to a continuous curve $\Gamma(\theta):\R/\Z\to \overline{A}$. In particular, $\Gamma(\R/\Z)=\partial A$ is locally connected.
\end{proof}

\begin{rmk}
If $f$ is a polynomial in the Main Theorem, then the Julia set of $f$ is locally connected by Proposition \ref{prop:attracting} since $J(f)=\partial A$.
\end{rmk}

\subsection{Proof of the Main Theorem}

We give a proof of the Main Theorem in this subsection based on the criterion in Lemma \ref{lem:LC-criterion}.

\begin{proof}[Proof of the Main Theorem]
Let $f$ be a rational map in the Main Theorem. Then every periodic Fatou component of $f$ is either a bounded type Siegel disk or an attracting basin. By \cite{Zha11} and Proposition \ref{prop:attracting}, the boundaries of all periodic Fatou components of $f$ are locally connected. This implies that the boundaries of all Fatou components of $f$ are locally connected. By Lemma \ref{lem:LC-criterion}, it suffices to prove that for every $\varepsilon>0$ there are at most finitely many Fatou components of $f$ with the spherical diameter $>\varepsilon$.

Iterating $f$ if necessary,  we assume that all periodic Fatou components of $f$ have period one.
Let $U_0$ be a fixed Fatou component of $f$. If $U_0$ is completely invariant, then $J(f)=\partial U_0$ is locally connected by Proposition \ref{prop:attracting}. Suppose $U_0$ is not completely invariant. Then $f^{-1}(U_0)\setminus U_0$ consists of at least one and at most finitely many Fatou components $\{U_{1,1}, \cdots, U_{1,i_1}\}$, where $i_1\geqslant 1$. There exists an integer $k_0\geqslant 1$ such that any connected component $U_{k_0+1,j}$ of $f^{-k_0}(U_{1,i})$ with $1\leqslant  i\leqslant  i_1$ is disjoint with $\MP(f)$. Therefore, for any connected component $U_{k+k_0+1,\ell}$ of $f^{-k}(U_{k_0+1,j})$ with $k\geqslant 1$, the map $f^k:U_{k+k_0+1,\ell} \to U_{k_0+1,j}$ is conformal.

Let $\delta_0>0$ be the constant introduced in Lemma \ref{lem:contrac}. Then there exists a uniform constant $M\geqslant 1$ such that every $U_{k_0+1,j}$ can be covered by $M$ Jordan disks $\{D_1,\cdots,D_M\}$ in $W_1$ whose spherical diameters are less than $\delta_0$. By Lemma \ref{lem:contrac}, there exists an integer $N\geqslant 1$ such that for all $k\geqslant N$, any connected component of $f^{-k}(D_i)$ has spherical diameter less than $\varepsilon/M$.
This implies that if $k\geqslant N$, then any connected component $U_{k+k_0+1,\ell}$ of $f^{-k}(U_{k_0+1,j})$ satisfies
$\diam_{\EC}(U_{k+k_0+1,\ell})<\varepsilon$.
Note that $f^{-n}(U_0)$ has only finitely many components for every $1\leqslant  n\leqslant  N+k_0$ and $f$ has only finitely many fixed Fatou components. Thus there are only finitely many Fatou components of $f$ whose spherical diameters are greater than $\varepsilon$.
\end{proof}

\bibliographystyle{amsalpha}
\bibliography{E:/Latex-model/Ref1}

\end{document}